\documentclass[11pt,reqno]{amsart}
\usepackage[letterpaper,margin=1in,footskip=0.25in]{geometry}
\usepackage{mathrsfs}
\usepackage{microtype}
\usepackage{mathtools}
\usepackage{enumitem}
\PassOptionsToPackage{dvipsnames}{xcolor}
\usepackage{tikz-cd}
\PassOptionsToPackage{pdfusetitle,pagebackref,colorlinks}{hyperref}
\usepackage{bookmark}
\hypersetup{citecolor=OliveGreen,linkcolor=Mahogany,urlcolor=Plum}
\usepackage[hyphenbreaks]{breakurl}
\usepackage{amssymb}
\usepackage{bbm}
\usepackage{enumitem}
\usepackage{upgreek}
\usepackage{bm}
\usepackage{stmaryrd}
\usepackage{mathrsfs}
\usepackage{comment}
\usepackage{graphicx}

\allowdisplaybreaks

\makeatletter
\providecommand{\leftsquigarrow}{%
	\mathrel{\mathpalette\reflect@squig\relax}%
}
\newcommand{\reflect@squig}[2]{%
	\reflectbox{$\m@th#1\rightsquigarrow$}%
}
\makeatother

\newcommand{\A}{\mathbb{A}}

\newcommand{\G}{\mathbb{G}}

\newcommand{\T}{\mathbb{T}}

\newcommand{\cB}{\mathcal{B}}
\newcommand{\cC}{\mathcal{C}}
\newcommand{\cD}{\mathcal{D}}
\newcommand{\cE}{\mathcal{E}}

\newcommand{\cJ}{\mathcal{J}}
\newcommand{\cL}{\mathcal{L}}
\newcommand{\cM}{\mathcal{M}}

\newcommand{\cO}{\mathcal{O}}

\newcommand{\cR}{\mathcal{R}}
\newcommand{\cS}{\mathcal{S}}

\newcommand{\cU}{\mathcal{U}}

\newcommand{\cZ}{\mathcal{Z}}
\newcommand{\cX}{\mathcal{X}}
\newcommand{\cY}{\mathcal{Y}}

\newcommand{\fa}{\mathfrak{a}}

\newcommand{\fm}{\mathfrak{m}}
\newcommand{\fn}{\mathfrak{n}}

\newcommand{\bP}{\mathbb{P}}
\newcommand{\bC}{\mathbb{C}}
\newcommand{\bR}{\mathbb{R}}
\newcommand{\bA}{\mathbb{A}}
\newcommand{\bQ}{\mathbb{Q}}
\newcommand{\bZ}{\mathbb{Z}}
\newcommand{\bD}{\mathbb{D}}

\newcommand{\bG}{\mathbb{G}}
\newcommand{\bF}{\mathbb{F}}

\newcommand{\bN}{\mathbb{N}}
\newcommand{\bk}{\mathbbm{k}}
\newcommand{\bK}{\mathbb{K}}

\newcommand{\sF}{\mathscr{F}}

\newcommand{\cI}{\mathcal{I}}

\newcommand{\rmA}{\mathrm{A}}
\newcommand{\rmY}{\mathrm{Y}}

\newcommand{\bfb}{\mathbf{b}}
\newcommand{\bfT}{\mathbf{T}}
\newcommand{\bfM}{\mathbf{M}}

\newcommand{\STR}{\overline{\mathrm{ST}}_R}

\DeclareMathOperator{\Aut}{Aut}

\DeclareMathOperator{\codim}{codim}

\DeclareMathOperator{\Div}{Div}

\DeclareMathOperator{\Exc}{Exc}

\DeclareMathOperator{\Hilb}{Hilb}

\DeclareMathOperator{\Spec}{Spec}

\DeclareMathOperator{\Pic}{Pic}

\DeclareMathOperator{\im}{Im}
\DeclareMathOperator{\ord}{ord}

\DeclareMathOperator{\Proj}{Proj}

\DeclareMathOperator{\Supp}{Supp}

\newcommand{\PGL}{\mathrm{PGL}}

\newcommand{\Hom}{\mathrm{Hom}}
\newcommand{\bT}{\mathbb{T}}

\newcommand{\ind}{\mathrm{ind}}
\newcommand{\tX}{\widetilde{X}}
\newcommand{\tZ}{\widetilde{Z}}

\newcommand{\tf}{\tilde{f}}
\newcommand{\edim}{\textrm{edim}}

\newcommand{\tB}{\widetilde{B}}
\newcommand{\K}{\mathrm{K}}

\newcommand{\oS}{\overline{S}}

\newcommand{\oX}{\overline{X}}
\newcommand{\oC}{\overline{C}}
\newcommand{\oD}{\overline{D}}
\newcommand{\oB}{\overline{B}}
\newcommand{\oG}{\overline{G}}
\newcommand{\oR}{\overline{R}}
\newcommand{\ocX}{\overline{\mathcal{X}}}
\newcommand{\ocD}{\overline{\mathcal{D}}}
\newcommand{\ocB}{\overline{\mathcal{B}}}
\newcommand{\ocG}{\overline{\mathcal{G}}}
\newcommand{\ocS}{\overline{\mathcal{S}}}
\newcommand{\ocR}{\overline{\mathcal{R}}}
\newcommand{\tx}{\tilde{x}}
\newcommand{\tw}{\tilde{w}}
\newcommand{\oz}{\bar{z}}
\newcommand{\CY}{\mathrm{CY}}
\newcommand{\KSBA}{\mathrm{KSBA}}
\newcommand{\sn}{\mathrm{sn}}
\newcommand{\la}{\lambda}

\newcommand{\bfa}{\mathbf{a}}
\newcommand{\DP}{\mathrm{DP}}
\newcommand{\cDP}{\mathcal{D}\mathcal{P}}
\newcommand{\bfr}{\mathbf{r}}
\newcommand{\bfc}{\mathbf{c}}
\newcommand{\lcd}{\mathrm{lcd}}
\newcommand{\rmO}{\mathrm{O}}
\newcommand{\Hodge}{\mathrm{Hodge}}

\newcommand{\bmu}{\bm{\mu}}
\newcommand{\ormY}{\overline{\mathrm{Y}}}

\newcommand{\YL}[1]{{\textcolor{blue}{[Yuchen: #1]}}}

\newcommand{\HB}[1]{{\textcolor{red}{[Blum: #1]}}}

\numberwithin{equation}{section}

\newtheorem{prop} {Proposition} [section]
\newtheorem{thm}[prop] {Theorem} 

\newtheorem{lem}[prop] {Lemma}

\newtheorem{prop-def}[prop]{Proposition-Definition}

\newtheorem{theorem}[prop]{Theorem}
\newtheorem{lemma}[prop]{Lemma}
\newtheorem{corollary}[prop]{Corollary}

\newtheorem{proposition}[prop]{Proposition}
\newtheorem{thm-defn}[prop]{Theorem-Definition}

\theoremstyle{definition}

\newtheorem{defn}[prop]{Definition}

\newtheorem{expl}[prop] {Example}

\theoremstyle{remark}
\newtheorem{rem}[prop]{Remark} 
\newtheorem{remark}[prop]{Remark}

\newtheorem{definition}[prop]{Definition}

\title{Good moduli spaces for boundary polarized Calabi--Yau surface pairs}

\author{Harold Blum}
\address{Department of Mathematics\\
	University of Utah\\
	Salt Lake City, UT 84112, USA.}
\email{blum@math.utah.edu}

\author{Yuchen Liu}
\address{Department of Mathematics, Northwestern University, Evanston, IL 60208, USA.}
\email{yuchenl@northwestern.edu}

\date{\today}

\begin{document}
	
	\begin{abstract}
		We construct  projective asymptotically good moduli spaces parametrizing boundary polarized CY surface pairs, which are projective slc Calabi--Yau pairs $(X,D)$ such that $D$ is ample and $X$ has dimension $2$.
		The moduli space provides a wall crossing  between certain  KSBA and K-moduli spaces  and is the ample model of the Hodge line bundle. 
		In the case of K3 surfaces with a non-symplectic automorphism, the moduli space gives a modular interpretation for the Baily--Borel compactification. 
	\end{abstract}
	
	\maketitle
	
	\setcounter{tocdepth}{1}
	
	\tableofcontents
	
	
	\section{Introduction}

	A central problem in algebraic geometry is to construct compact moduli spaces of varieties. 
	From the perspective of the Minimal Model Program, the three main building blocks of varieties are canonically polarized, Calabi--Yau, and Fano varieties.
	In two of these classes, there is a relatively complete moduli theory.
	An approach that began in work of Koll\'ar, Shephard-Barron, and Alexeev  (KSBA) produces projective coarse moduli spaces for canonically polarized varieties with slc singularities \cite{KSB88, Ale94, Kol08,BCHM10,AH11, HX13, KP17, HMX18, Fuj18,Kol23}. 
	An approach using K-stability produces  projective good moduli spaces parametrizing K-semistable Fano varieties \cite{Jia20, CP21, LWX21, BX19, ABHLX20, Xu20, BLX22, XZ20, XZ21, BHLLX21, LXZ22}.

	The remaining case of Calabi--Yau varieties is less well understood.
	While techniques from Hodge Theory, mirror symmetry, log KSBA theory, and log K-stability theory provide various approaches to construct moduli of CY varieties, there is no unifying and complete  solution.  
	Recently, the two authors with K. Ascher, D. Bejleri, K. DeVleming, G. Inchiostro,  and X. Wang introduced a new approach for constructing moduli spaces  of certain CY pairs that, when it exists, interpolates between the KSBA and K-moduli theories  \cite{BABWILD}.
	The goal of this paper is to complete the latter approach in dimension two.

	Following \cite{BABWILD},  we study the moduli of  \emph{boundary polarized CY surface  pairs}, which are projective slc pairs $(X,D)$ such that $\dim X=2$,  $K_X+D\sim_{\bQ}0$, and $D$ is an ample $\bQ$-Cartier $\bQ$-divisor.
	Interesting examples of such pairs arise in the study of del Pezzo surfaces and K-trivial varieties in low dimension.
	
	\begin{enumerate}
		\item If $X$ is a smooth del Pezzo surface and $H \in |-mK_X|$ is a smooth divisor, then $(X, \tfrac{1}{m}H)$ is boundary polarized CY surface pair. 
		
		\item   A general hyperelliptic K3 surface is the double cover of a pair in example (1) when 
		$m=2$, $H\in |-2K_X|$ is a smooth curve, and $X$ is either $\bP^2$, $\bP^1\times \bP^1$ or $\bF_1$.
		
		\item If $E$ is an elliptic curve with an ample line bundle $L$, then $(C_p(E,L),H)$ is a boundary polarized CY pair, where $C_p(E,L)$ is the projectivized cone over $E$ with polarization $L$  and $H\subset C_{p}(E,L)$ is the divisor at infinity.
	\end{enumerate}

	A key feature of boundary polarized CY pairs is that they encode CY geometry, as well as sometimes Fano and even  canonically polarized geometry. Indeed, if $(X,D)$ is a klt boundary polarized CY surface pair, then 
	\begin{itemize}
		\item $(X,(1-\varepsilon)D)$ is a K-stable  log Fano pair,
		\item $(X, (1+\varepsilon) D)$ is a KSBA-stable canonically polarized pair
	\end{itemize}
	for $0< \varepsilon \ll1$. 
	Thus techniques from Fano K-moduli and KSBA-moduli theories can be used to construct compact moduli spaces $M^{\rm K}$ and $M^{\rm KSBA}$ parametrizing boundary polarized CY pairs satisfying an additional stability condition for the perturbed pair.
	Ideally, there would exist a moduli space of CY pairs at $\varepsilon=0$ that interpolates between $M^{\K}$ and $M^{\KSBA}$.

	
	Motivated  in part by this expectation,  \cite{BABWILD} initiated a program to study the moduli theory of boundary polarized CY pairs.
	The first step of this program, which was carried out in \emph{loc. cit.}, is to define a moduli stack parametrizing boundary polarized CY pairs and  study its properties. 
	The next step is to construct some notion of a moduli space for this stack. 
	This step turns out to be quite subtle, since the irreducible components of the moduli stack of boundary polarized CY pairs are often not of finite type.
	Therefore standard techniques from moduli theory do not apply to construct the desired moduli space.

	
	\subsection{Main result}
	The main result of this paper is the construction of  a projective moduli space for boundary polarized CY surface pairs.

	To state the the theorem, let $(\cX, \cD)\to T$ be a family of marked boundary polarized CY  surface pairs over a finite type scheme $T$. Write $\cM^{\CY}$ for the seminormalization of the closure of the image of $T$ in the relevant moduli stack of boundary polarized CY pairs (see Section \ref{ss:moduli-stack} for the definition of the moduli stack). 
	Let  $\cM^{\rm K}$  and $\cM^{\rm KSBA}$ 
	denote the open substacks of $\cM^{\CY}$ parametrizing  pairs $(X,D)$ in $\cM^{\CY}$ such that $(X,(1-\varepsilon)D)$ is K-semistable and $(X,(1+\varepsilon)D)$ is KSBA-stable for $0<\varepsilon\ll 1$, respectively.

	\begin{thm}\label{t:mainintro}
		There exists  an asymptotically good moduli space morphism
		\[
		\cM^{\CY}
		\to 
		M^{\CY}
		\]
		to a projective  scheme satisfying the following. 
		\begin{enumerate}
			\item $M^{\CY}(\bk)$ parametrizes S-equivalence classes of pairs in $\cM^{\CY}(\bk)$.
			\item The open inclusions of moduli stacks
			\[
			\cM^{\K} \hookrightarrow \cM^{\CY} 
			\hookleftarrow \cM^{\KSBA}
			\]
			induce projective morphisms of (asymptotically) good moduli spaces
			\[
			M^{\K} \to M^{\CY} \leftarrow M^{\KSBA}
			.\]
			\item The Hodge line bundle is ample on $M^{\CY}$.
		\end{enumerate}
	\end{thm}

	In the special case when the general fibers of $\cX\to T$ are isomorphic to $\bP^2$, Theorem \ref{t:mainintro} was previously proven in \cite{BABWILD}, which  confirmed a conjecture in \cite{ADL19}.
	The latter proof relied at various key steps on properties of slc degenerations of $\bP^2$ that do not hold for more general slc Fano surfaces.

	The notion of an \emph{asymptotically good moduli space morphism}, which appears in the theorem, was defined in \cite{BABWILD} to
	allow for a non-quasicompact stack to admit a finite type  moduli space. 
	The definition states that there exists an infinite chain of open substacks 
	\[
	\cU_1\subset \cU_2 \subset \cU_3 \subset \cdots \subset \cM^{\CY}
	\]
	such that $\cM^{\CY} = \cup_{i \geq 1} \cU_i$ and   the composition  $\cU_i\hookrightarrow \cM^{\CY}\to  M^{\rm CY}$ is  a good moduli space morphism for each $i$.
	As a consequence of properties of good moduli space morphisms, $\cM^{\CY} \to M^{\CY}$ is universal among maps from $\cM^{\CY}$ to algebraic spaces and identifies two boundary polarized CY pairs  if and only if they are S-equivalent, i.e. they admit isotrivial degenerations via test configurations to a common boundary polarized CY pair.

	
	The CY moduli space may be viewed as an analogue of the Baily--Borel compactification of the moduli space of K3 surface or abelian varieties in our setting. 
	Indeed, the Hodge line bundle is ample on our moduli space and the boundary parametrizes lower dimensional CY pairs. 
	As we will see in Theorem \ref{t:nonsymplecticautintro}, in special cases the moduli space even agrees with a Baily--Borel compactification up to normalization.

	There are a number of interesting related conjectures in the literature by Shokurov, Odaka \cite{Oda22}, and Laza \cite{Laz24} predicting that there should exist a compactification of the moduli space of  polarized CY varieties on which the Hodge line bundle is ample and the boundary parametrizes lower dimensional CY varieties.
	(These conjectures are also linked to the b-semiampleness conjecture of Prokhorov and Shokurov \cite{PS09} in birational geometry.) 
	Theorem \ref{t:mainintro}  confirms these expectations in the case of boundary polarized CY surface pairs.
	A novel feature of the theorem is that its proof primarily uses birational geometry
	and good moduli space theory, rather than Hodge Theory, to construct the moduli space.

	In the body of the paper, we work in the more general setup of boundary polarized CY surface pairs of the form  $(X,B+D)$, which are projective slc pairs such that $\dim X=2$,  $K_{X}+B+D\sim_{\bQ} 0$ and $D$ is ample. 
	Theorem \ref{t:mainintro} extends to families of such pairs when the marked coefficients of $B$ are in the standard set $\{\tfrac{1}{2} ,\tfrac{2}{3}, \tfrac{3}{4}, \ldots , 1\}$ and $D$ has arbitrary rational coefficients. See Theorem \ref{t:mainCY}.

	
	\subsection{Special cases}
	
	We list two important special cases of the above theorem.
	See Section \ref{s:examples} for a more detailed discussion on explicit examples.
	
	
	\subsubsection{Del Pezzo surfaces with pluri-anti-canonical divisors}
	For each integers $1\leq d \leq 9$ and  $r\geq 1$, let $\mathcal{DP}_{d,r}$ denote the stack parametrizing pairs $(X,\tfrac{1}{r} C)$, where $X$ is a smooth del Pezzo surface of degree $d$ and $C\in |-rK_X|$ is a smooth curve.
	The stack $\mathcal{DP}_{d,r}$ is a Deligne--Mumford stack and admits a quasi-projective coarse moduli space  $\DP_{d,r}$.
	K-moduli and KSBA-moduli theory produces projective compactifications $\DP_{d,r}^{\K}$ and $\DP_{d,r}^{\KSBA}$  of $\DP_{d,r}$.
	These compactifications have been studied in various explicit examples in \cite{Hac04, GHK15, Laz16, GMGS18, ADL19, ADL20, HKY20, DH21,  Gol21, AET23, AEH21, AE22,PSW23, Zha22, MGPZ24}. 
	
	To construct the CY moduli space, 
	let $\mathcal{DP}_{d,r}^{\rm CY}$ denote the seminormalization of the closure of $\cDP_{d,r}$ in the the relevant moduli stack of boundary polarized CY pairs. 
	Theorem \ref{t:mainintro} applied to this setup   produces the following result.

	\begin{thm}\label{t:delpezzointro}
		There exists a projective seminormal scheme $\DP_{d,r}^{\CY}$ whose closed points parametrize S-equivalence classes of pairs in $\mathcal{DP}_{d,r}^{\CY}(\bk)$.
		Furthermore,  the Hodge line bundle is ample on $\DP_{d,r}^{\CY}$ and there are wall crossing morphisms
		\[
		\DP^{\K}_{d,r} \rightarrow  \DP^{\CY}_{d,r}\leftarrow \DP^{\KSBA}_{d,r}.
		\]
	\end{thm}
	
	As an  application of Theorem \ref{t:delpezzointro}, in Theorem \ref{thm:dP-lines} we show that there exists a log Calabi--Yau wall crossing diagram that connects K-moduli and KSBA-moduli compactifications of smooth del Pezzo surfaces with sum of lines through an asymptotically good moduli space of boundary polarized CY pairs. 
	Combining with results on the wall crossing of K-moduli and KSBA moduli spaces in \cite{ GKS21, Sch23, Zha23}, we obtain a complete wall crossing framework interpolating between K-moduli spaces of del Pezzo surfaces \cite{MM93, OSS16} and the KSBA moduli spaces of del Pezzo surfaces with the sum of lines as the reduced boundary \cite{HKT09}. 
	See Section \ref{sec:dP} for further discussion.
	
	\subsubsection{K3 surfaces with a non-symplectic automorphism}\label{ss:intrononsymplectic}
	Following \cite{AEH21}, let $S$ be a smooth K3 surface with a  non-symplectic automorphism $\sigma$ of finite order $n>1$ whose fixed locus contains a smooth curve $R$ of genus $\geq2$,
	and write $\rho $ for the induced isometry of the K3 lattice $L_{\rm K3}:= H^2(S,\bZ)$. 
	The ample model of $R$ is a K3 surface $\oS$ with ADE singularities. Then  $\sigma$ descends to a non-symplectic automorphism $\overline{\sigma}$ on $\oS$ of order $n$. 
	In this setting, the quotient $\oS/\langle \overline{\sigma} \rangle$ is a boundary polarized CY pair of the form $(X, \tfrac{n-1}{n}C)$. 
	A well known special case of this construction is that a general degree 2 K3 surface admits an anti-symplectic involution and its quotient is the pair $(\bP^2, \tfrac{1}{2} C)$, where $C\subset \bP^2$ is a sextic curve. 
	
	Let  $F_{\rho}^{\rm ade}$ 
	denote the  coarse moduli space of  K3 surfaces $\oS$ with at worse ADE singularities and  a non-symplectic automorphism $\overline{\sigma}$ of type $\rho$. Then the period map gives an open immersion $F_{\rho}^{\rm ade}\hookrightarrow \bD_{\rho}/\Gamma_{\rho}$ to an arithmetic quotient of a  Hermitian symmetric domain which admits a Baily--Borel compactification $(\bD_{\rho}/\Gamma_{\rho})^{\rm BB}$. Let $\sF_{\rho}$ denote the moduli stack of boundary polarized CY pairs $(X, \tfrac{n-1}{n}D)$ that are quotients of K3 surfaces $(\oS, \overline{\sigma})$. Write $\sF_{\rho}^{\CY}$ for the seminormalization of the closure of $\sF_{\rho}$ in the relevant moduli stack of boundary polarized CY pairs. By construction,  the coarse moduli space $F_{\rho}$ of $\sF_{\rho}$ is isomorphic to $F_{\rho}^{\rm ade}$. As a consequence of Theorem \ref{t:mainintro}, there exist compactifications of ${F}_{\rho}$ that fit into a diagram
	\[
	F_{\rho}^{\rm K} \to F_{\rho}^{\rm CY} \leftarrow F_{\rho}^{\rm KSBA}
	,\]
	where $F_{\rho}^{\CY}$ is an asymptotically good moduli space of $\sF_{\rho}^{\CY}$.
	
	Following the significant work \cite{AE23}, in \cite{AEH21} it was shown that the normalization of $F_{\rho}^{\rm KSBA}$ is a semi-toroidal compactification of $\bD_{\rho}/\Gamma_{\rho}$. 
	In the case of degree 2 K3 surfaces and, more generally, when the automorphism is an involution, the specific semi-toroidal compactification was explicitly described in \cite{AET23} and \cite{AE22}.
	As a consequence of the ampleness of the Hodge line bundle in Theorem \ref{t:mainintro}.4, we give the following description of $F_{\rho}^{\rm CY}$, which provides a modular meaning in the sense of asymptotically good moduli spaces for the Baily--Borel compactification.

	\begin{thm}\label{t:nonsymplecticautintro}
		The normalization of $F_{\rho}^{\CY}$ is isomorphic to the Baily--Borel compactification $(\bD_{\rho}/\Gamma_{\rho})^{\rm BB}$.
	\end{thm}

	\subsection{Sketch of proof}
	We now explain the proof of of Theorem \ref{t:mainintro}.
	Recall that $\cM^{\CY}$ denotes the seminormalization of the closure of the image of a finite type scheme $T$ in the moduli stack of boundary polarized CY pairs. 
	While $\cM^{\CY}$ is a locally finite type algebraic stack, its irreducible components can fail to be finite type. 
	Thus we cannot directly construct a moduli space for $\cM^{\CY}$.
	To add some finiteness, we consider the chain of  substacks
	\[
	\cM_1 \subset \cM_2 \subset \cM_3 \subset \cdots \subset \cM^{\CY}
	,\]
	where $\cM_m\subset \cM^{\CY}$ is the open substack parametrizing boundary polarized CY pairs $(X,D)$ with  Gorenstein index of $X$ is $\leq m$ at each point. 
	As a consequence of \cite{Kol85,Fuj17},   $\cM_m$ is finite type.

	To construct a good moduli space for $\cM_m$, we seek to apply \cite{AHLH23}, which states that $\cM_m$  admits a separated good moduli space 
	if and only if $\cM_m$ is S-complete and $\Theta$-reductive. 
	The latter notions involve extending a morphism $S\setminus 0  \to \cM_m$ from a certain  punctured  stacky surface to  a morphism $S \to \cM_m$ (see Definition \ref{d:scomplete+Thetareductive}).
	By \cite{BABWILD}, such a morphism will extend uniquely to a morphism $S\to \cM^{\CY}$. 
	To show that $S\to \cM^{\CY}$ factors through $\cM_m\hookrightarrow \cM^{\CY}$, 
	we prove a key new result on deformations of the Gorenstein index of boundary polarized CY surface pairs (Theorem \ref{thm:index-deform}). 
	The result relies on taking an index one cover  and then using that Gorenstein semi-log-terminal slc surface singularities are hypersurface singularities by \cite{KSB88}.
	In \cite{BABWILD}, this result was proven when the general fiber of $\cX\to T$ is $\bP^2$ by reducing to the case when $D$ is reduced and then using  the deformation theory of certain stacky curves. For arbitrary boundary polarized CY surface pairs, the reduction  step is not possible and so we need to analyze deformations of the surface.

	By the above argument, we deduce that $\cM_m$ is S-complete and $\Theta$-reductive for $m\gg0$ and hence verify the existence of a morphism  $\cM_m \to M_m$ to a separated good moduli space (Theorem \ref{t:existsGMindex}). 
	The inclusions of the substacks induce morphisms
	\[
	M_{m}\to M_{m+1} \to M_{m+2} \to \cdots.
	\]
	Using the properness of $M^{\rm KSBA}$ \cite{KX20}, we argue  
	that $M_m$ is proper for $m\gg0$ (Proposition \ref{p:properness}).
	The remaining steps are  to verify that  $M_m \to M_{m+1}$ is an isomorphism and the Hodge line bundle is ample on $M_m$ when $m\gg0$.
	Once this is complete, the main theorem will hold with $M^{\rm CY}:= M_m$ for $m\gg0$. 
	
	To proceed, we consider the source of a boundary polarized CY pair, which is the minimal lc center on a dlt modification (Definition \ref{d:source}).
	We say a pair is Type I, II, or III if the source is a surface pair, a curve pair, or a point, respectively. 
	This corresponds to the terminology used for degenerations of K3 surfaces.
	If the map $\cM_{m} \to M_m$  identifies two pairs, then their types and sources are preserved and so the  source and type of a point in $M_m$ is well defined.

	To show the stabilization of the moduli spaces, we use that the Type I and II loci of $\cM^{\CY}$ is bounded by \cite{BABWILD}.
	Thus the Type I and II loci of $M_m$  stabilizes.
	Next, we show that the Type III locus is discrete for $m\gg0$. 
	This relies on a delicate construction showing that any Type III pair admits an isotrivial degeneration via a test configuration 
	to a Type III pair whose normalization is toric.
	Using this result, we deduce that $M_m \to M_{m+1}$ is a bijection on closed points for $m\gg0$ (Theorem \ref{t:stabilization}).
	Since  each $M_m$ is seminormal, $M_m \to M_{m+1}$ is  an isomorphism for $m\gg0$.

	The last remaining step is to verify the ampleness of  the Hodge line bundle.
	As a consequence of  results of \cite{Amb05,Kaw97,Kol07} that rely on the theory of variations of Hodge structures, 
	the Hodge line bundle will be ample on $M_m$ if and only if the sources of the points on the moduli space have maximal variation.
	The variation of the source is maximal on the Type I locus tautologically and on the Type III locus as it is discrete.
	To verify the maximal variation of the source on the Type II locus,  we give a detailed description  of the geometry  of polystable Type II pairs in Section \ref{s:TypeII}.
	Once this is complete, we deduce that the Hodge line bundle is ample on $M^{\CY}$.
	The previous maximal variation results strengthen an analysis done in the $\bP^2$ case in \cite{BABWILD}.
	

	\subsection{Higher dimensions}
	In dimension at least three, there are examples where the set of S-equivalence classes of boundary polarized CY pairs do not admit an algebraic structure and so Theorem \ref{t:mainintro} will not extend verbatim to higher dimensions. These examples were mentioned in \cite{BABWILD} and will be computed in detail in a later paper. 
	
	Nevertheless, we still expect that there is projective scheme $M^{\CY}$ that parametrizes certain equivalence classes (but not S-equivalence classes) of boundary polarized CY pairs in higher dimensions. 
	In particular, we expect that there are  wall crossing morphisms 
	$
	M^{\K} \to M^{\CY} \leftarrow M^{\KSBA}
	$,
	which realize $M^{\CY}$ as the common ample model of the Hodge line bundle on the moduli spaces $M^{\K}$ and  $M^{\KSBA}$. 
	In addition, the boundary of $M^{\CY}$ should parametrize lower dimensional CY pairs that are sources of non-klt boundary polarized CY pairs. 
	This expectation is in line with conjectures of \cite{PS09,Laz24,Oda22} in the setting of boundary polarized CY pairs. 

	\subsection{Structure of the paper}
	The paper is structured as follows. 
	In Section \ref{s:prelim}, we discuss preliminaries on boundary polarized CY pairs including the notions of test configurations, sources, families, and moduli stacks of such pairs. 
	In Section \ref{s:Goreinsteinindex}, we construct good moduli spaces for boundary polarized CY surface pairs of bounded Gorenstein index. 
	In Section \ref{s:TypeII} and \ref{s:TypeIII}, we prove results on the S-equivalence classes of Type II and Type III boundary polarized CY surface pairs. 
	In Section \ref{s:asgm}, we use results from the previous sections to prove Theorem \ref{t:mainintro}.
	Finally, in Section \ref{s:examples}, we discuss explicit examples of the  main theorem and prove Theorems \ref{t:delpezzointro} and \ref{t:nonsymplecticautintro}.

	\subsection*{Acknowledgements} We thank Kenneth Ascher, Dori Bejleri, Kristin DeVleming, Giovanni Inchiostro, and Xiaowei Wang for their collaboration on  \cite{BABWILD}, which this paper builds on. 
	Additionally, we thank Sebastian Casalaina-Martin, Philip Engel, Changho Han, Yunfeng Jiang, Radu Laza, and Ananth Shankar for many useful conversations.  HB is partially supported by NSF Grant DMS-2200690. YL is partially supported by NSF CAREER Grant DMS-2237139
	and the Alfred P. Sloan Foundation.

	
	\section{Preliminaries}\label{s:prelim}
	Throughout, all schemes are defined over an algebraically closed field $\bk$  of characteristic $0$, which is our ground field. 
	A scheme $X$ is demi-normal if it is $S_2$ and its codimension 1 points are regular or nodes \cite[Definition 5.1]{Kol13}.
	
	\subsection{Pairs}
	A pair $(X,B)$ consists of a  demi-normal scheme $X$ that is essentially of finite type over a field $\bK\supset \bk$ and an effective $\bQ$-divisor $B$ on $X$ such that $\Supp(B)$ does not contain codimension 1 singular points of $X$ and $K_X+B$ is $\bQ$-Cartier. 
	Unless stated otherwise, we assume that $X$ is geometrically connected.
	For notions of singularities of pairs that assume $X$ is normal, including \emph{lc}, \emph{klt}, \emph{plt}, and \emph{dlt}, see \cite[Definition 2.8]{Kol13}. 
	A pair $(X,B)$ is called a surface (resp., curve) pair if $\dim X=2$ (resp., $\dim X=1$).

	The normalization of a pair $(X,B)$ is the possibly disconnected pair $(\oX, \oG+ \oB)$, where $\pi:\oX\to X$ is the normalization of $X$, $\oG \subset \oX$ is the conductor divisor on $\oX$, and $\oB$ is the divisorial part of $\pi^{-1}(B)$. 
	These satisfy 
	\[
	K_{\oX}+\oG+\oB= \pi^*(K_X+B)
	\]
	\cite[5.7]{Kol13}. 
	In this setting, there is a natural generically fixed point free involution $\tau:\oG^n \to \oG^n$ that fixes ${\rm Diff}_{\oG^n}(\oB)$
	and  the triple $(\oX,\oG+\oB,\tau)$ uniquely determines $(X,B)$ \cite[Section 5.1]{Kol13}.
	
	A pair $(X,B)$ is \emph{slc} if its normalization $(\oX,\oG+\oB)$ is lc. 
	A pair $(X,B)$ with $X$ projective is called \emph{log Fano} (resp., \emph{CY}, \emph{KSBA-stable}) if $(X,B)$ is slc and $-K_X-B$ is ample (resp., $K_{X}+B\sim_{\bQ}0$, $K_X+B$ is ample).
	
	\subsection{Families of pairs}
	
	To define a family of pairs, we need the following notion of a family of divisors \cite[Definition 4.68]{Kol23}.
	
	\begin{defn}
		Let $f:X\to T$ be a  flat finite type morphism with  fibers of pure dimension $n$. A subscheme $B\subset X$ is a \emph{relative Mumford divisor} if there is an open set $U\subset X$ such that 
		\begin{enumerate}
			\item $\codim_{X_t}(X_t \setminus U_t) \geq 2$ for every $t\in T$
			\item $D\vert_U$ is a relative Cartier divisor (i.e. $D\vert_U$ is a Cartier divisor on $U$ and does not contain irreducible components of fibers),  
			\item $D$ is the closure of $D\vert_U$, and 
			\item $X_t$ is smooth at the generic points of $D_t$ for every $t\in T$.
		\end{enumerate}
		When $T=\Spec(\bk)$, we simply call $D$ a \emph{Mumford divisor}.
	\end{defn}

	If $B\subset X$ is a relative Mumford divisor for $f:X\to T$ and $T'\to T$ is a morphism, then its \emph{divisorial pullback} $D_{T'}$ on  $X_{T'}:= X\times_T T'$ is the closure of the pullback of $D\vert_{U}$ to $U':= U\times_T T'$.
	Note that $D_t$ in (4) denotes the divisorial pullback  and does not always agree with the scheme theoretic fiber.

	\begin{defn}
		A \emph{family of slc pairs} $(X,B)\to T$ over a reduced Noetherian scheme $T$ is 
		\begin{enumerate}
			\item a flat finite type morphism $f:X\to T$ with fibers of pure dimension $n$  and 
			\item $B= \sum_{i}^l b_i B_i$, where each $B_i \subset X$ is a relative Mumford divisor and $b_i\in \bQ_{\geq0}$
		\end{enumerate}
		satisfying the following conditions:
		\begin{enumerate}
			\item[(3)] $K_{X/T}+B$ is $\bQ$-Cartier and
			\item[(4)] $(X_t, B_t)$ is an slc pair for every $t\in T$.
		\end{enumerate}
	\end{defn}
	
	The above notion agrees with a \emph{locally stable family} in \cite[Definition--Theorem 4.7.1]{Kol23}. 
	We will frequently use the following characterization of a family of slc pairs over a smooth curve.
	
	\begin{prop}[{\cite[Theorem 2.3]{Kol23}}]\label{p:familyslcpairsovercurve}
		If $X\to C$ is a flat proper morphism to a smooth curve and $B:= \sum b_i B_i$, where  each $B_i$ is a Mumford divisor on $X$ and $b_i \in \bQ_{\geq0}$, then the following are equivalent:
		\begin{enumerate}
			\item $(X,B)\to C$ is a family of slc pairs.
			\item $(X,B+X_c)$ is an slc pair for all closed points $c\in C$.
		\end{enumerate}
	\end{prop}
	
	\subsection{Marked pairs}
	A marking of a pair $(X,B)$ is a way of writing $B= \sum_{i=1}^r b_i B_i$, where the $B_i$ are Mumford divisors and  $b_i \in \bQ_{\geq 0}$. Note that 
	\[
	(\bA^1, \tfrac{1}{2} \{ x^2=0\}) \quad \text{ and } \quad (\bA^1, 1 \{x=0\})
	\]
	agree as pairs, but differ as marked pairs. 
	
	For the purpose of defining a moduli functor, it is important to keep track of a marking, e.g. the marking of the above pair determines the allowable deformations. For many arguments in this paper, the distinction between a pair and marked pair will not be important, 
	since, for a family of slc pairs $(X,B) \to T$ over a finite type normal scheme $T$, a new marking of the generic fiber of $B$ induces a new marking for $B$ \cite[8.5.1]{Kol23}. Thus the distinction will often be ignored when it is not important.

	\subsection{Boundary polarized CY pairs}
	
	\begin{defn}
		A \emph{boundary polarized CY} pair $(X,B+D)$ is an slc pair such that 
		\begin{enumerate}
			\item $B$ and $D$ are effective $\bQ$-divisors,
			\item $K_X+B+D\sim_{\bQ}0$, and 
			\item $D$ is an  ample $\bQ$-Cartier divisor.
		\end{enumerate}
	\end{defn}
	Note that  (3) is  equivalent to the condition that $(X,B)$ is log Fano.
	Hence we refer to $B$ as the \emph{log Fano boundary} and $D$ as the \emph{polarizing boundary} of $(X,B+D)$ and always write them in this order to avoid confusion.

	\begin{rem}
		Given a boundary polarized CY pair $(X,B+D)$, we can construct additional boundary polarized CY pairs as follows.
		\begin{enumerate}
			\item If $X$ is non-normal, then the normalization 
			$
			(\oX,\oG+\oB+\oD)
			$
			is a possibly disconnected boundary polarized CY pair, where $\oG+\oB$ is the log Fano boundary.
			This follows from the fact that $\pi:\oX\to X$ is finite,  $\pi^*(K_X+B) = K_{\oX}+\oG+\oB$, and $\pi^*(K_X+B+D) = K_{\oX}+\oG+\oB+\oD$. 
			
			\item Let $E \subset \lfloor \oG+ \oB \rfloor$ be a prime divisor
			and set 
			\[
			B_{E^n}={\rm Diff}_{E^n}(B-E)  \quad \text{ and }\quad D_{E^n}:=D\vert_{E^n}
			.
			\]
			Then $(E^n, B_{E^n} + D_{E^n})$ is boundary polarized CY pair, since $(K_{X}+B)\vert_{E^n} \sim_{\bQ} K_{E^n}+B_{E^n}$ and  $(K_{X}+B+D)\vert_{E^n} \sim_{\bQ} K_{E^n}+ B_{E^n}+D_{E^n}$ by \cite[(4.2.9)]{Kol13}.
			Note that $D$ is $\bQ$-Cartier and hence the definition for $D_{E^n}$ as the pullback of $D$ to $E^n$ is well defined.
		\end{enumerate}
	\end{rem}

	\begin{defn}\label{d:familybpcyreduced}
		A \emph{family of boundary polarized CY pairs} $(X,B+D)\to T$ over a reduced Noetherian scheme $T$ is a projective family of  slc pairs such that 
		\begin{enumerate}
			\item $D$ is $\bQ$-Cartier and relatively ample over $T$ and 
			\item $K_{X/T}+B+D \sim_{\bQ,T}0$.
		\end{enumerate}
		Additionally, we say that the \emph{CY index} of the family divides an integer $N$ if $N(K_{X/T} +B+D)\sim_T 0$.
	\end{defn}
	
	In order define a moduli stack of boundary polarized CY pairs, we will  define a family of boundary polarized CY pairs over arbitrary bases in Definition \ref{d:familyofbpcys}.
	While the later definition will be more complicated, the two definitions agree for families over  reduced Noetherian schemes by Proposition \ref{p:familybpcyequivalentdef}.
	
	\subsection{Source}\label{ss:source}
	We recall the definition of the source of an slc CY pair, which was first defined in \cite{Kol13} in a slightly less general setting.

	\subsubsection{Dlt modification} 
	If $(X,B)$ is an lc pair, then there exists a 
	\emph{dlt modifcation } 
	\[
	(Y,B_Y)\to (X,B)
	,\]
	which is the data of a pair $(Y,B_Y)$ and a proper birational morphism $f:Y\to X$ such that $K_Y+B_Y = f^*(K_X+B)$ and the pair $(Y,B_Y)$ is dlt \cite[Theorem 1.34]{Kol13}.
	If $S\subset Y$ is an lc center of $(Y,B_Y)$, then  $(S,B_S)$ is a dlt pair, where $B_S: = {\rm Diff}_S^*(B_Y)$ is defined via adjunction and  satisfies $K_S+ B_S \sim_{\bQ} (K_Y+B_Y)\vert_{S}$. 
	See \cite[Section 4.2]{Kol13} for details.
	
	\subsubsection{Source and regularity}
	
	\begin{defn}[Source]\label{d:source}
		Let $(X,B)$ be a  CY pair.
		\begin{enumerate}
			\item If $(X,B)$ is klt, then its source is ${\rm Src}(X,B):= (X,B)$.
			\item If $(X,B)$ is not klt, then let
			$(\oX,\oB):= \sqcup_{i=1}^r (\oX_i,\oB_i) $ denote its normalization and its decomposition into irreducible lc pairs. 
			For an arbitrary $1\leq i \leq r$, let 
			\[(Y,B_Y) \to (\oX_i,\oB_i)
			\] be a dlt modification and $S\subset Y$ be a minimal lc center. We set ${\rm Src}(X,B)$ equal to the crepant birational equivalence class of $(S,B_S)$. 
		\end{enumerate}
	\end{defn}
	
	Note that the source of a non-klt CY pair is the crepant birational equivalence class of a klt CY pair of lower dimension. As a consequence of \cite{Kol16}, the source of a  CY pair is independent of choice of dlt modification and minimal lc center; see  \cite[Lemma 8.2]{BABWILD} for details.
	
	\begin{defn}
		The \emph{coregularity} and \emph{regularity} of a CY pair $(X,B)$ are
		\[
		{\rm coreg}(X,B) = \dim {\rm Src}(X,B) \quad \text{ and } \quad {\rm reg}(X,B) :=   \dim X- \dim {\rm Src}(X,B)-1
		.\]
	\end{defn}
	
	Since we will usually work with surface pairs, we use the following terminology.

	\begin{defn}
		A CY surface pair $(X,B)$ is called 
		\begin{enumerate}
			\item Type I if $(X,B)$ is klt,
			\item Type II if ${\rm Src}(X,B)$ is  a curve pair, and
			\item Type III if ${\rm Src}(X,B)$ is a point. 
		\end{enumerate}
	\end{defn} 
	
	Note that Type I, II, and III correspond to coregularity 2, 1, and 0, respectively. 
	Additionally, when $(X,B)$ is a CY surface pair, the source is a single pair, since in Types II and III there is a single pair in the crepant birational equivalence class.

	\subsection{Moduli spaces}
	
	\subsubsection{Good moduli spaces}
	
	We recall the definition of a good moduli space in \cite{Alp13}.
	
	\begin{defn}
		A a \emph{good moduli space} $\phi:\cM \to M$ is a quasi-compact morphism from an algebraic stack to an algebraic space such that
		\begin{enumerate}
			\item $\phi_*$ is exact on quasi-coherent sheaves and 
			\item the natural morphism $\cO_M \to \phi_* \cO_{\cM}$ is an isomorphism.
		\end{enumerate}
	\end{defn}
	
	We will frequently use the following properties of good moduli spaces proven in \cite{Alp13}.
	
	\begin{prop}\label{p:propertiesofGMs}
		If $\phi:\cM\to M$ is a good moduli space, then the following hold:
		\begin{enumerate}
			\item If $\cM$ is locally Noetherian, then $\phi$ is universal among maps from $\cM$ to algebraic spaces. 
			\item For an algebraically closed field $\bK \supset \bk$ and $x,x'\in \cM(\bK)$, $\phi(x)=\phi(x')$ if and only if $\overline{\{x\}} \cap \overline{\{x'\}}\neq \emptyset$.
		\end{enumerate}
	\end{prop}
	
	\begin{defn}
		An open substack $\cU\subset \cM$ is saturated with respect to a good moduli space morphism $\phi:\cM\to M$ if $\phi^{-1}(\phi(\cU))=\cU$.  
	\end{defn} 
	
	\begin{prop}[{\cite[Remark 6.2]{Alp13}}]\label{p:saturatedopen}
		If $\phi:\cM\to M$ is a good moduli space morphism and $\cU\subset \cM$ is an open substack saturated with respect to $\phi$, then $U:=\phi(\cU)$ is open in $M$ and $\phi:\cU\to U $ is a good moduli space. 
	\end{prop}

	\subsubsection{Existence criterion}
	
	In order to construct good moduli spaces in this paper, we will use the following criterion.
	
	\begin{thm}[{\cite[Theorem A]{AHLH23}}]\label{t:AHLH}
		Let $\cM$ be a finite type algebraic stack with affine diagonal. Then the following are equivalent:
		\begin{enumerate}
			\item There exists a good moduli space morphism $\cM\to M$ to a separated algebraic space. 
			\item $\cM$ is S-complete and $\Theta$-reductive with respect to DVRs essentially of finite type over $\bk$.
		\end{enumerate}
	\end{thm}
	
	\begin{rem}
		We cannot apply this theorem directly to the moduli stack of boundary polarized CY pairs, since the stack  (and even its irreducible components) can fail to be finite type by Remark \ref{r:notfinitetype}.
	\end{rem}
	
	We now define  terminology in Theorem \ref{t:AHLH}.2. Let $R$ be a DVR over $\bk$ with uniformizer $\pi$. Let 
	\[
	\STR := [\Spec R[s,t]/(st-\pi)] \quad \text{and } \quad \Theta_R := [\Spec R[t]/\bG_m]
	,\]
	where $\bG_m$ acts on the above affine schemes with weights $1$ and $-1$ on $s$ and $t$, respectively, and trivially on $R$. 
	We a bit abusively write $0 \in \STR$ and $0 \in \Theta_R$ for the unique closed points.
	
	\begin{defn}\label{d:scomplete+Thetareductive}
		Let $\cM$ be an algebraic stack. 
		\begin{enumerate}
			\item The stack $\cM$ is \emph{S-complete} if any DVR $R$ and morphism  ${\STR}\setminus 0 \to \cM$ extends uniquely to a morphism $\STR \to \cM$.
			\item The stack $\cM$ is \emph{$\Theta$-reductive} if any DVR $R$ and morphism $\Theta_R\setminus 0\to \cM$ extends uniquely to a morphism $\Theta_R \to \cM$. 
		\end{enumerate}
		We say $\cM$ is \emph{S-complete} or \emph{$\Theta$-reductive} with respect to essentially of finite type DVRs if the respective statement in (1) or (2) holds for DVRs essentially of finite type over $\bk$.
	\end{defn}
	
	\subsection{Test configurations}
	
	\begin{defn}
		A \emph{test configuration} of a boundary polarized CY pair is the data of 
		\begin{enumerate}
			\item a family of boundary polarized CY pairs $(\cX,\cB+\cD)\to \bA^1$
			\item a $\bG_m$-action on $(\cX,\cB+\cD)$ extending the standard action on $\bA^1$, and
			\item an isomorphism $(\cX_1,\cB_1+\cD_1)\simeq (X,B+D)$. 
		\end{enumerate}
	\end{defn}
	By (2) and (3), there exists  a canonical $\bG_m$-equivariant isomorphism
	\[
	\cX\setminus \cX_0 \simeq X\times (\bA^1 \setminus \{0\})
	,
	\]
	where $\bG_m$ acts on the right as the product of the trivial action on $X$ and the standard action on $\bA^1\setminus \{0 \}$.
	By (2), the $\bG_m$-action fixes $\cX_0$ and so induces a $\bG_m$-action on $(\cX_0,\cB_0+\cD_0)$.

	\begin{rem}
		In the K-stability literature one often considers test configurations where the special fiber may have arbitrarily bad singularities. Since we are requiring that $(\cX_0,\cB_0+\cD_0)$ is a boundary polarized CY pair and, in particular, that the special fiber has slc singularities, these were referred to as \emph{weakly special test configurations} in \cite{BABWILD}.
		Since we only encounter test configurations of the above form in this paper, we  leave out the words weakly special.
	\end{rem}
	
	\begin{defn}
		We say that there exists a \emph{weakly special degeneration}
		\[
		(X,B+D)\rightsquigarrow (X_0,B_0+D_0)
		\]
		if there exists a test configuration $(\cX,\cB+\cD)$ of $(X,B+D)$ such that $(\cX_0,\cB_0+\cD_0)\simeq (X_0,B_0+D_0)$.
	\end{defn}

	We now define various properties of test configurations. 
	\begin{defn}
		Let $(\cX,\cB+\cD)$ be a test configuration of a boundary polarized CY pair $(X,B+D)$. 
		\begin{enumerate}
			\item The test configuration is a \emph{product} if there exists a $\bG_m$-equivariant isomorphism over $\bA^1$
			\[
			(\cX,\cB+\cD) \simeq (X\times \bA^1, B\times \bA^1+D\times \bA^1),
			\] 
			where  $\bG_m$ acts on $X\times \bA^1$ as the product of a $\bG_m$-action on $X$ and the standard $\bG_m$-action on $\bA^1$.
			
			\item The test configuration is \emph{trivial} if the equivariant isomorphism in (1) can be chosen so $\bG_m$ acts on $X\times \bA^1$ as the product of the trivial action on $X$ and the standard action on $\bA^1$.
			
			\item If an algebraic group $G$ acts on $(X,B+D)$, then the test configuration is \emph{$G$-equivariant} if the induced fiberwise $G$-action on $X\times(\bA^1\setminus 0)$ extends to a $G$-action on $\cX$. 
		\end{enumerate}
	\end{defn}
	
	\begin{remark}
		We mention a couple properties of equivariant test configurations.
		\begin{enumerate}
			\item  If $G$ is a connected algebraic group (e.g. $G=\bG_m$)  acting on $(X,B+D)$, then every test configuration of the pair is $G$-equivariant by \cite[Proposition 4.14 and Theorem 6.5]{BABWILD}.
			\item For a $G$-equivariant test configuration $(\cX,\cB+\cD)$ of $(X,B+D)$, the $G$ and $\bG_m$-action on $\cX$ commute, since they commute over $\cX\setminus \cX_0$. 
			Thus there is an induced $G\times \bG_m$-action on $(\cX_0,\cB_0+\cD_0)$.
		\end{enumerate}
	\end{remark}

	\begin{remark}
		If $(\cX,\cB+\cD)$ is a test configuration of $(X,B+D)$, then we can construct new test configurations in the following ways. 
		\begin{enumerate}
			\item The normalization $(\ocX,\ocG+\ocB+\ocD)$ of the pair $(\cX,\cB+\cD)$ is naturally a test configuration of $(\oX,\oG+\oB+\oD)$.  (The fact that it is a family of slc pairs follows from applying Proposition \ref{p:familyslcpairsovercurve} on $\cX$ and on $\ocX$.)
			
			\item If $E\subset \lfloor \oG+\oB\rfloor $ is a prime divisor, then $(\cE^n, \cB_{\cE^n}+\cD_{\cE^n})$ is a test configuration of $(E^n, B_{E^n}+D_{E^n})$, where 
			$\cE$ is the closure of $E\times (\bA^1\setminus 0)$ in $\ocX$, 
			\[
			\cB_{\cE^n} := {\rm Diff}_{\cE^n}(\ocG+\ocB-\cE)\quad \text{and } \quad \cD_{\cE^n} := \ocD\vert_{\cE^n}
			\]
		\end{enumerate}
	\end{remark}

	\subsubsection{Valuations}
	Let $(X,B)$ be a normal pair. A 
	\emph{divisorial valuation} on $X$ is the data of a valuation $v: K(X)^\times \to \bR$ such that $v= b\ord_E$, where $E\subset Y$ is prime divisor on a normal variety $Y$ with a proper birational morphism $\mu:Y \to X$ and $b \in \bR_{\geq 0}$.
	The \emph{log discrepancy} of $v= b \ord_E$ is 
	\[
	A_{X,B}(v) : = b (1 + {\rm coeff}_E(K_{Y}- \mu^*(K_X+B) ).
	\]

	Given a test configuration $(\cX,\cB+\cD)$ of an lc boundary polarized CY pair $(X,B+D)$, we can associate a set of divisorial valuations as follows.
	Since $\cX\setminus \cX_0$ is normal and $(\cX,\cB+\cD)\to \bA^1$ is a family of slc pairs, $\cX$ is normal using Proposition \ref{p:familyslcpairsovercurve}.
	Thus each irreducible 
	component $E\subset \cX_0$ induces a valuation 
	$\ord_E: K(\cX)^\times \to \bZ$
	defined by the order of vanishing along $E$. We write 
	\[
	v_E := r(\ord_E):K(X)^\times \to \bZ
	\]
	for its restriction to $K(X)^\times$ under the embedding $K(X)\subset K(X)(t)\simeq K(\cX)$.
	By \cite[Lemma 4.1]{BHJ17}, each $v_E$ is a $\bZ$-valued divisorial valuation on $X$. 
	
	\subsubsection{Filtrations}	
	Let $(X,B+D)$ be a boundary polarized CY pair and fix a positive integer $r$ such that $L:=-r(K_X+B)$ is a Cartier divisor. We set 
	\[
	R(X,L):= R: = \bigoplus_{m \in \bN} R_m:=\bigoplus_{m \in \bN} H^0(X,\cO_X(mL))
	.\]
	A \emph{filtration} $F^\bullet$ of $R$ is the data of vector subspaces $F^\la R_m\subset R_m$ for each $m \in \bN$ and $\la \in \bZ$ satisfying
	\begin{enumerate}
		\item $F^\la R_m \subset F^{\la'}R_m$ for $\la'\geq \la$;
		\item $F^\la R_m \cdot F^{\la'}R_{m'}\subset F^{\la+\la'}R_{m+m'}$;
		\item $F^{-\la} R_m=R_m$ and $F^\la R_m =0$ for $\la\gg0$.
	\end{enumerate}
	A filtration $F^\bullet$ of $R$ is \emph{finitely generated} if the graded $\bk[t]$-algebra
	${\rm Rees}(F) := \bigoplus_{m \in \bN} \bigoplus_{\la \in \bZ} F^\la R_m
	$ is finitely generated.
	Note that 
	\begin{equation}\label{e:Rees}
		{\rm Rees}(F) \otimes_{\bk[t]} \bk[t^{\pm1}] \simeq R[t^{\pm1}]
		\quad \text{ and } \quad  {\rm Rees}(F) / t {\rm Rees}(F) = \bigoplus_{m \in \bN} \bigoplus_{\la \in \bZ} {\rm gr}_F^{\la}R_m
		,
	\end{equation}
	where ${\rm gr}_F^\la R_m:= F^\la R_m / F^{\la+1}R_m$.
	
	If $(\cX, \cB+\cD)$ is a test configuration of $(X,B+D)$  and $\cL:=-r(K_{\cX/\bA^1}+\cB)$  is Cartier, then there is an induced  filtration $F^\bullet $ of $R$ defined by 
	\[
	F^\la R_m := \{ s\in R_m \, \vert\, t^{-\la} \overline{s}\in H^0(\cX,\cO_{\cX}(m\cL) )\}
	,\]
	where $\overline{s}$ is the unique $\bG_m$-invariant section of 
	$ H^0(\cX\setminus \cX_0,\cO_{\cX}(\cL)) $ such that $\overline{s}\vert_{\cX_1}= s$  and $t$ is the parameter for $\bA^1$. By the discussion in \cite{BHJ17}, the natural map 
	of graded $\bk[t]$-modules
	\begin{equation}\label{e:weightdecomponmcL}
		{\rm Rees}(F)
		\to
		H^0(\cX,\cO_{\cX}(m\cL))=
		\bigoplus_{\la \in \bZ} H^0(\cX,\cO_{\cX}(m\cL))_\la,
	\end{equation}
	where the subscript on the right denotes the weight $\la$-eigenspace,
	is an isomorphism
	and induces
	a $\bG_m$-equivariant  isomorphism
	$\cX \simeq \Proj( {\rm Rees}(F))$  over $\bA^1$.
	
	
	\begin{lemma}\cite[Lemma 4.7]{BABWILD}\label{l:tcfiltformula}
		If $(\cX,\cB+\cD)$ is a test configuration of an lc boundary polarized CY pair $(X,B+D)$ such that  $\cX$ is normal and $-r(K_{\cX/\bA^1}+\cD)$ is Cartier, then 
		\begin{enumerate}
			\item $
			F^\la R_m :=\bigcap_{E\subset \cX_0} \{s \in H^0(X, mL)\, \vert\, v_E(s) \geq \la + mr A_{X,B}(v_E)\}
			$ and 
			\item  $A_{X,B+D}(v_E)=0$ for each irreducible component $E\subset \cX_0$.
		\end{enumerate}
	\end{lemma}

	\subsubsection{Restriction}\label{sss:tcrestriction}
	
	Let $(X,B+D)$ be an  boundary polarized CY  pair and $E \subset \lfloor B \rfloor $ a prime divisor that is normal.
	Fix an integer $r>0$ such that $L:= -r(K_X+B)$ is Cartier. Write 
	\[
	R
	:=
	\bigoplus_{m \in \bN} R_m 
	:= 
	\bigoplus_{m \in \bN} H^0(X,mL)
	\quad \text{ and } \quad 
	R_E
	:=
	\bigoplus_{m \in \bN} R_{E,m} 
	:= 
	\bigoplus_{m \in \bN} H^0(E,mL\vert_E)
	.\]
	Adjunction induces  a canonical isomorphism 
	\[
	\omega_{X}^{[-mr]}(-mrB) \vert_E \simeq \omega_{E}^{[-mr]}(-mrB_E)
	\]
	for each integer $m$ and so  induces a graded $\bk$-algebra homomorphism $\phi:R \to R_{E}$, 
	whose kernel is the homogeneous ideal $I \subset R$ defining $E$. 
	Since $L$ is ample, the  map  $R_m\to R_{E,m}$ is surjective when $m$ is sufficiently  divisible.
	
	The next lemma relates the filtrations of the two test configurations. 
	In the statement, we denote by $F$ and $F_E$ for the filtrations of $R$ and $R_E$ induced by the above test configurations. 
	
	\begin{proposition}\label{p:tcrestriction}
		If $m \in \bN$ is sufficiently divisible and $\la \in \bZ$, then
		\[
		\im (F^\la R_m \to R_{E,m}) = F_E^\la R_{E,m}  
		.\]
	\end{proposition}
	
	\begin{proof}
		Fix  $m\in \bN$ sufficiently divisible so that $R_m \to R_{m,E}$ is surjective and $-mr(K_{\cX}+\cB)$ is a Cartier divisor.
		Adjunction induces a   canonical $\bG_m$-equivariant isomorphism 
		\[
		\omega_{\cX/\bA^1}^{[-mr]}(-mr\cB)\otimes \cO_{\cE} \to  \omega_{\cE/\bA^1}(-mr \cB_{\cE})
		\]
		and  a morphism
		\begin{equation}\label{e:cXtocE}
			H^0(\cX, \omega_{\cX/\bA^1}^{[-mr]}(-mr\cB)) \to 
			H^0(\cX, \omega_{\cE/\bA^1}^{[-mr]}(-mr\cB_\cE))
		\end{equation}
		of $\bZ$-graded $\bk[t]$-modules, where the $\bZ$-grading is induced by the weight decomposition on each module.
		Since $-r(K_{\cX/\bA^1}+ \cB)$ is ample over $\bA^1$, the above map is surjective after possibly replacing $m$ with a positive multiple. 
		
		Now fix $s_E \in R_{m,E}$. 
		By definition,  $s_E \in F_E^\la R_{E,m}$ if and only if 
		\[
		\overline{s_E} t^{-\la}\in H^0(\cE, \omega_{\cE/\bA^1} (-mr\cB_\cE))_\la.
		\] 
		By the surjectivity of \eqref{e:cXtocE} and the isomorphism\eqref{e:weightdecomponmcL}, the latter holds if and only if there exists $s\in F^\la R_m$ such that $\overline{s} t^{-\la}\vert_{\cE} = \overline{s_E} t^{-\la}$. 
		Since $(\overline{s} t^{-\la})\vert_{\cE} = \overline{s_E} t^{-\la}$ implies $s\vert_E = s_E$, the previous  condition holds if and only if there exists $s\in F^\la R_m$ such that $\phi(s)=s_E$.
	\end{proof}
	
	\subsubsection{Degeneration via boundary}
	
	The following proposition  constructs  test configurations induced by prime divisors on the boundary polarized CY pair.
	The construction will be used repeatedly in Section \ref{s:TypeIII}.

	\begin{proposition}\label{p:tcdivisorsonX}
		Let $(X,B+D)$ be an lc boundary polarized CY pair and 
		\[
		v_1 := c_1 \ord_{E_1}, ~\ldots,~ v_l:= c_l \ord_{E_l}
		\]
		be valuations such that the  $E_i$'s are distinct prime divisors contained in $\Supp( \lfloor B+D\rfloor )$ 
		and the $c_i$'s are  positive  integers satisfying
		$A_{X,B}(v_1 ) =\cdots = A_{X,B}(v_l)$.
		
		Then the filtration $F^\bullet$ of $R:=R(X,L)$ defined by
		\[
		F^\la R_m :=  \cap_{i=1}^l\{ s\in R_m \, \vert\, v_i(s) \geq \la +mr A_{X,B}(v_i) \}
		\]
		is finitely generated and induces a  test configuration 
		$(\cX,\cB+\cD)$ of $(X,B+D)$ satisfying:
		\begin{enumerate}
			\item The fiber $\cX_0$ has $l$ irreducible components.
			\item The  $\bG_m$-action is non-trivial on each irreducible component of $\cX_0$.
			\item The induced degeneration of $E:= E_1\cup \cdots \cup E_l$ is trivial. 
		\end{enumerate}
		
	\end{proposition}
	
	Statement (3) means that if we let $\cE$ denote the closure of $ E\times(\bA^1\setminus 0 )$  in $\cX$, then there is a $\bG_m$-equivariant isomorphism $\cE \simeq E\times \bA^1$, where $\bG_m$ acts on $E\times \bA^1$ as the product of  the trivial action on $E$ and the standard action on $\bA^1$.

	\begin{proof}
		Since $A_{X,B+D}(v_i)=0$ for each $i$, 
		${\rm Rees}(F) $ is a finitely generated $\bk[t]$-algebra by \cite[Proposition 4.9.i]{BABWILD}.
		Consider the $\bG_m$-equivariant morphism 
		\[
		\cX:= \Proj(  {\rm Rees}(F)) \to \bA^1 ,
		\]
		where the $\bG_m$-action on $\cX$ is induced by the $\bZ$-grading on the Rees algebra. 
		By \eqref{e:Rees},  there are $\bG_m$-equivariant morphisms
		\[
		\cX \setminus \cX_0 \simeq X\times(\bA^1\setminus 0) \quad \text{ and } \quad \cX_0 \simeq \Proj \Big( \bigoplus_{m \in \bN} \bigoplus_{\la \in \bZ} {\rm gr}_F^\la R_m \Big)
		.\]
		Write $\cB$ and $\cD$ for the closures of $B\times (\bA^1\setminus 0)$ and $D\times(\bA^1\setminus 0)$ in $\cX$.
		By \cite[Proposition 4.9.ii]{BABWILD} and its proof, $(\cX,\cB+\cD)$ is a test configuration of $(X,B+D)$, and there is a $\bG_m$-equivariant isomorphism $\omega_{\cX/\bA^1}^{[-mr]}(-mrB) \simeq \cO_{\cX}(m)$ for $m>0$ sufficiently divisible.
		
		To verify  (1), write $\cX_0 = \cX_0^1\cup \cdots \cup \cX_0^{k}$ for the decomposition of $\cX_0$ into irreducible components. 
		By \cite[Lemma 4.2]{BABWILD},  $k\leq l$ and, after possibly reordering the valuations,
		$v_i = r (\ord_{\cX_0^i})$ for all  $1\leq i \leq k$.
		We aim to show that $l=k$.
		By Lemma \ref{l:tcfiltformula},
		the filtration $F_{\cX}^\bullet$ of $R$  induced by  $(\cX,\cB+\cD)$
		satisfies
		\[
		F_{\cX}^\la R_m := \cap_{i=1}^{k} \{ s\in R_m \, \vert \, v_i(s) \geq \la + mr A_{X,B}(v_i) 
		\}
		.
		\]
		By the bijection between test configurations of polarized varieties and finitely generated filtrations of the section rings in \cite[Proposition 2.15]{BHJ17}, after possibly replacing $r$ with a multiple, the filtrations $F^\bullet$ and $F_{\cX}^\bullet$ are equal.
		Since $L$ is ample, there exists a positive integer $m$ and $s\in R_m$ such that $s$ vanishes on $E_1 \cup \cdots \cup E_k$ and does not vanish on $E_i$ for  $i>k$.
		Set 
		\[
		c:= A_{X,B}(v_1)=\cdots = A_{X,B}(v_l).
		\]
		If $k< l$, then $s\in F_{\cX}^{-mrc+1} R_m $ and $s\notin  F^{-mrc+1} R_m$, which  contradicts  $F^\bullet   = F_{\cX}^\bullet$. 
		Therefore $k=l$ and so (1) holds and $v_i = r(\ord_{\cX_0^i})$ for each $1\leq i \leq l$.
		
		To verify (2), we will describe the $\bG_m$-action on an equivariant birational model  of $\cX_0^i$.
		Let $w_i:=\ord_{\cX_0^i}$, which we view as a valuation of $K(X)(t)$.
		Since $r(w_i) = v_i$, \cite[Lemma 4.5]{BHJ17} implies that 
		\[
		w_i(f_0 + f_1 t+ \cdots + f_d t^d) = \min\{ v_i(f_j) + j \, \vert\, f_j \neq 0 \}
		,\]
		where each $f_j \in K(X) $ and  not all are zero.
		Let $U \subset X$ be the smooth locus of $X$ and  $\cJ:= \cI_{U \cap E_i} \subset \cO_{U}$.
		With this notation, the $m$-th valuation ideal of $w_i$ on $U\times \bA^1$ is
		\[
		\fa_m :=  \bigoplus_{j \leq  0} \cJ^{\lceil (m+j)/c_i\rceil } t^{-j} \subset  \bigoplus_{j\leq 0}\cO_{U} t^{-j}
		,\]
		where the unusual sign convention on the grading is chosen since $t$ has weight $-1$ with respect to the $\bG_m$-action on $U\times \bA^1$.
		Since the $\cO_{U}[t]$-algebra $\bigoplus_{m \geq 0} \fa_m$ is finitely generated,  we may consider the projective morphism
		\[
		Y := \Proj( \fa_0 \oplus \fa_1 \oplus \cdots ) \to U\times \bA^1
		\]
		By \cite[Theorem 1.4]{Blu21}, $Y$ is normal, the  exceptional locus of $Y \to U\times \bA^1$ is
		\begin{align*}
			Z&= \Proj( \fa_0/\fa_1  \oplus \fa_1/\fa_2 \oplus \cdots ) \\
			& \simeq \Proj (\fa_0/\fa_1 \oplus \fa_{c_i}/\fa_{c_i+1} \oplus \fa_{2c_i}/\fa_{2c_i+1}\oplus \cdots ),
		\end{align*}
		which is an irreducible prime divisor,
		and $\ord_{Z} = w_i$.
		Thus $\cX_0^i$ is the birational transform of $Z$ on $\cX$ and, hence, there is a $\bG_m$-equivariant birational map $Z \dashrightarrow \cX_0^i$. 
		Observe that
		\begin{align*}
			\fa_{mc_i} / \fa_{mc_i+1}  &= \cO_{U}/ \cJ \oplus  \cJ/\cJ^2 \oplus  \cdots \oplus \cJ^{m}/ \cJ^{m+1 } \\
			& \simeq  {\rm Sym}^{m } (\cO_U \oplus  \cJ/\cJ^2 ),
		\end{align*}
		where in the last expression $\cO_U$ and $\cJ/\cJ^2$ have weights $-c_i$ and $0$, respectively.
		Thus $Z$ is equivariantly birational to $E_i \times \bA^1$, where $\bG_m$ acts on $E_i \times \bA^1$ as the product of the $\bG_m$-action on $E_i$ and a non-trivial action on  $\bA^1$.
		Since the $\bG_m$-action on $E_i\times \bA^1$ is non-trivial, so is the $\bG_m$-action on $\cX_0^i$, which proves (2).
		
		It remains to prove (3). 
		Write $I:= \bigoplus_{m \geq 0} I_m \subset R$ for the homogeneous ideal defining $E$. 
		Write $\cE$ for the closure of $E\times (\bA^1\setminus 0)$ in $\cX$. 
		Observe that $\cE$ is defined by the homogeneous  ideal
		\[
		\cI:= \bigoplus_{m \in \bN} \bigoplus_{\la \in \bZ} (F^\la R_m \cap I_m) \subset {\rm Rees}(F)
		,\]
		since $\cI$ is the kernel of the composition
		\[
		{\rm Rees}(F) \to R[t^{\pm 1}] \to (R/I)[t^{\pm1}]
		.\]
		Since $F^\la R_m =R_m$ when $\la\leq -mrc$ and $F^{-mrc+1} R_m =I_m$, we see 
		\[
		F^\la R_m \cap I_m=
		\begin{cases}
			I_m & \text{ if } \la \leq -mrc+1 \\
			F^\la R_m & \text{ if } \la \geq -mrc+1. 
		\end{cases}
		\]
		Thus 
		\[
		{\rm Rees}(F)/ \cI  \simeq  \bigoplus_{\la \leq -mrc} (R_m/I_m) t^{-\la }
		\]
		and so  there is a $\bG_m$-equivariant isomorphism $\cE\simeq E\times \bG_m$.
	\end{proof}

	\begin{prop}\label{p:tcdivisorsonXextra}
		Keep  the notation in the setup and conclusions of Proposition \ref{p:tcdivisorsonX}.
		
		If there is a $\bG_m$-action on $(X,B+D)$ where each  $E_i$ is not in the $\bG_m$-fixed locus of $X$,  then the induced $\bG_m^2$-action on each irreducible component of $\cX_0$ has finite kernel.
	\end{prop}	
	
	\begin{proof}
		By Proposition \ref{p:tcdivisorsonX}, for each $1\leq i \leq l$, there is a $\bG_m^2$-equivariant birational map
		\[
		\cX_0^i \dashrightarrow E_i \times \bA^1
		,\]
		where  $\bG_m^2$ acts on the right as as the product of the $\bG_m$-action on $E_i$ and a non-trivial $\bG_m$-action on $\bA^1$. 
		Since $E_i$ not in the $\bG_m$-fixed locus of $X$, the $\bG_m$-action on $E_i$ is non-trivial. 
		Thus the kernel of the  $\bG_m^2$-action on $\cX_0^i$
		contains no non-trivial 1-parameter subgroups of $\bG_m^2$. 
		Since the connected non-trivial subgroups of $\bG_m^2$ are precisely the non-trivial 1-parameter subgroups, 
		we conclude that the $\bG_m^2$-action on $\cX_0^i$ has finite kernel.
	\end{proof}

	\subsection{Moduli stack}\label{ss:moduli-stack}
	In this section, we define the moduli stack of boundary polarized CY pairs and study its properties. 
	The  main statements appear in \cite{BABWILD} in a slightly less general setting.
	Throughout this section, we fix the data 
	\begin{itemize}
		\item a function $\chi: \bN\to \bZ$, which will be our Hilbert function,
		\item $N \in \bZ_{>0}$, which will be a multiple of the CY index, and
		\item ${\bf  a}= (a_1,\ldots a_l)  \in (\tfrac{1}{N}\bZ_{>0})^l$ and $c \in \tfrac{1}{N}\bZ_{>0}$, which will be our boundary coefficients.
	\end{itemize}
	
	\subsubsection{Moduli stack}
	

	
	\begin{defn}\label{d:familyofbpcys}
		A \emph{family of boundary polarized CY pairs} $(X, B + D)\to T$ over a Noetherian scheme $T$ with coefficients $(\bfa, c)$ and index dividing $N$ consists of the following data:
		\begin{enumerate}
			\item $X\to T$ is a flat projective morphism of schemes;
			\item  $B = \sum_{i=1}^{l} a_i B_i$ and $D = cD_1$, where $B_i$ and $D_1$ are relative K-flat Mumford divisors 
		\end{enumerate}
		satisfying the following conditions:
		\begin{enumerate}
			\setcounter{enumi}{2}
			\item $(X_t, B_t + D_t)$ is a boundary polarized CY pair for every $t\in T$,
			\item $\omega_{X/T}^{[N]}(N(B+D))\cong_{T} \cO_X$, and
			\item $\omega_{X/T}^{[m]}( mB  + qD_1)$ is flat over $T$ and commutes with base change for every $m,q\in \bZ$ with $m\bfa\in \bZ^l$.
		\end{enumerate}
	\end{defn}
	
	\begin{rem}\label{rem:moduli-stack-def}
		We now explain some of the terminology in Definition \ref{d:familyofbpcys}.
		\begin{enumerate} 
			\item[(i)] When $l=1$, the definition recovers \cite[Defintion 3.1]{BABWILD}.
			\item[(ii)] For the definition of a K-flat relative Mumford divisor, see  \cite[Section 7.1]{Kol23}. 
			
			\item[(iii)] We allow the case when ${\bf a} = (1)$ and $B= 1 \cdot \emptyset$, which corresponds to a family of boundary polarized CY pairs with no divisor in the log Fano boundary.
			
			\item[(iv)] To define the sheafs in (4) and (5), let $i:U\hookrightarrow X$ be the largest open subset on which fibers of $U \subset B$ are either smooth or nodal and the $B_i$'s and $D_1$ are relative Cartier divisors. 
			Then
			\[
			\omega_{X/T}^{[m]}(mB+ qD_1) := i_* \omega_{U/T}^{\otimes m}(m B_U + q D_1 )
			.\] 
			With this notation, condition (4) means that $\omega_{X/T}^{[N]}(N(B+D))\cong_{T} \cO_X \otimes L_T$, for some  line bundle $L_T$, which is the pullback of a line bundle from $T$.
			Condition (5) means that for any morphisms of schemes $T'\to T$, the natural map
			\[
			g'^* \omega_{X/B}^{[j]}(mB+qD_1) \to \omega_{X'/B'}^{[m]}(m B+q D_1)
			\]
			is an isomorphism, where $X' = X\times_T T'$ and $g' : X' \to X$ is the first projection. 
			This type of condition is referred as Koll\'ar's condition in the literature; see \cite[Chapter 9]{Kol23}.
			\item[(v)] If $T$ is reduced, then dropping the K-flatness in condition (2) and replacing (5) by the condition  that $D$ is $\bQ$-Cartier gives an equivalent definition. 
			As a consequence, Definition \ref{d:familyofbpcys} agrees with Definition \ref{d:familybpcyreduced} when $T$ is a reduced Noetherian scheme. 
		\end{enumerate}
	\end{rem}

	The next result shows that it is equivalent to use the simpler definition of a family of boundary polarized CY pairs when the base is a reduced Noetherian scheme.
	
	\begin{prop}\label{p:familybpcyequivalentdef}
		If $T$ is a reduced Noetherian scheme and $(X,\sum_{i=1}^l a_i B_i + c D_1) \to T$ is a family of boundary polarized CY pairs  with CY index dividing $N$ as defined in Definition \ref{d:familybpcyreduced}, then the family satisfies Definition \ref{d:familyofbpcys}.
	\end{prop}
	
	\begin{proof}
		This follows from the same argument as \cite[Lemma 3.3]{BABWILD}.    
	\end{proof}

	\begin{rem}[Non-noetherian]
		We can extend Definition \ref{d:familyofbpcys} to families over  arbitrary (possibly non-Noetherian) schemes 
		by the bootstrapping from the Noetherian case as in 
		\cite[Definition 3.4]{BABWILD}.
		This is needed to define the moduli stack below. 
	\end{rem}

	Using the previous definition, we now define the moduli stack of boundary polarized CY pairs. 
	When $l=1$, the definition agrees with \cite[Definition 3.4]{BABWILD}.
	
	\begin{defn}\label{def:moduli-stack}
		Let $\cM(\chi, N, \bfa, c)$ denote the category fibered in groupoids over $\mathrm{Sch}_{\bk}$ where:
		\begin{itemize}
			\item The objects are families of boundary polarized CY pairs 
			\[
			(X, B+D) \to T
			\]
			with coefficients $(\bfa, c)$, index dividing $N$, and $\chi(X_t, \omega_{X_t}^{[-m]}( -m B)) = \chi(m)$ for each $t\in T$ and $m\in \bN$ with $m\bfa\in \bZ^l$.
			\item The morphisms $[(X', B'+D')\to T']\to [(X, B+D) \to T]$ consist of morphisms of schemes $X'\to X$ and $T'\to T$ such that $X'\to X\times_T T'$ is an isomorphism and $B_i'$ and $D'_1$ are the divisorial pullbacks of $B_i$ and $D_1$.
		\end{itemize}
	\end{defn}
	
	
	
	\begin{thm}\label{thm:moduli-alg-stack}\label{t:algebraicstack}
		$\cM(\chi, N, \bfa, c)$ is an algebraic stack locally of finite type over $\bk$ with affine diagonal. 
	\end{thm}
	
	\begin{proof}
		The proof is very similar to \cite[Proof of Theorem 3.7]{BABWILD}, so we only provide a sketch. 
		
		By asymptotic Riemann-Roch, we may assume that $\chi(m) = \frac{V}{n!}m^n + O(m^{n-1})$ for some $V\in \bQ_{>0}$ and $n\in \bN$, as otherwise $\cM(\chi, N, \bfa, c)$ is empty.
		Let $\bfb=(b_1, \cdots, b_l)\in \bQ_{\geq 0}^{l}$ 
		and $k\in \bN$ such that $k\bfa\in \bZ^l$. Let $\cM(\chi, N, \bfa, c, \bfb, k)\subset \cM(\chi, N, \bfa, c)$ be the full subcategory consisting of families $[f:(X, B+D) \to T]$ such that $-k(K_{X_t}+B_t)$ is a very ample Cartier divisor and $\deg(B_{i,t}) = b_i$. Here the degree of a Weil divisor $G_t$ on $X_t$ is defined as $\deg(G_t):=((-K_{X_t}-B_t)^{n-1}\cdot G_t)$. Note that $\deg(D_{1,t}) = c^{-1}V$ as $V= (-K_{X_t}-B_t)^n = ((-K_{X_t}-B_t)^{n-1}\cdot cD_{1,t})$. We denote by $b_{l+1}:=c^{-1}V$. We claim that $\cM(\chi, N, \bfa, c, \bfb, k)$ is an algebraic stack of finite type over $\bk$ with affine diagonal.
		
		By \cite[Proposition 3.5]{BABWILD} we know that for any object $[f:(X, B+D) \to T]$ in $\cM(\chi, N, \bfa, c, \bfb, k)$, we have that $\omega_{X/T}^{[-k]}(-kB)$ is a $f$-very ample line bundle and  $f_*\omega_{X/T}^{[-k]}(-kB)$ is a rank $\chi(k)$ vector bundle on $T$. Moreover, we have an induced closed embedding $X\hookrightarrow \bP_T(f_*\omega_{X/T}^{[-k]}(-kB))$. Denote by $M:=\chi(k) - 1$. By \cite[Proof of Proposition 3.9]{BABWILD}, there exists a $\PGL_{M+1}$-invariant open subscheme $H_1\subset \Hilb_{\chi(k\cdot)}(\bP^M)$ that parametrizes closed subschemes $X\hookrightarrow\bP^M$ with Hilbert polynomial $\chi(k\cdot)$ such that $X$ is geometrically connected, reduced, equidimensional, deminormal, $H^i(X, \cO_X(1)) = 0$ for all $i>0$, and that the restriction map $H^0(\bP^M, \cO_{\bP^M}(1))\to H^0(X, \cO_X(1))$ is an isomorphism. Denote by $X_{H_1}\to H_1$ the universal family.
		
		Next, let 
		\[
		H_2' := \bigtimes_{i=1}^{l+1}{}_{H_1}\mathrm{KDiv}_{k^{n-1}b_i}(X_{H_1}/H_1),
		\]
		where we denote by $\mathrm{KDiv}_{k^{n-1}b_i}(X_{H_1}/H_1)$ the parameter space of K-flat relative Mumford divisors of degree $k^{n-1}b_i$. By \cite[Theorem 7.3]{Kol23} we know that $H_2'$ is a separated scheme of finite type over $\bk$. Let $B_{i, H_2'}$ with $1\leq i \leq l$ (resp. $D_{1, H_2'}$) be the universal family of K-flat relative Mumford divisors on $X_{H_2'}= X_{H_1}\times_{H_1} H_2'$ of degree $k^{n-1} b_i$ (resp. $k^{n-1} b_{l+1}$). From the construction, it is clear that $H_2'$ has a natural $\PGL_{M+1}$-action such that $H_2'\to H_1$ is $\PGL_{M+1}$-equivariant. Let $B_{H_2'}:=\sum_{i=1}^l B_{i, H_2'}$ and $D_{H_2'}:=c D_{1, H_2'}$.
		
		Next, by similar arguments to \cite[Proof of Proposition 3.9]{BABWILD}, there exists a $\PGL_{M+1}$-equivariant locally closed partial decomposition $H_5'\hookrightarrow H_2'$ such that the pull-back of the universal family $(X_{H_2'}, B_{H_2'}+ D_{H_2'})$ under $\phi:T\to H_2'$ gives rise to an object in $\cM(\chi, N, \bfa, c, \bfb, k) (T)$ if and only if $\phi$ factors through $H_5'$. Moreover, we have an isomorphism $[H_5'/\PGL_{M+1}]\cong \cM(\chi, N, \bfa, c, \bfb, k)$. Since $H_5'$ is a scheme of finite type over $\bk$, the claim is proved.
		
		Finally, by similar arguments to \cite[Proof of Theorem 3.7]{BABWILD}, we have that $\cM(\chi, N, \bfa, c)$ satisfies \'etale descent and has a covering by open substacks $\cM(\chi, N, \bfa, c, \bfb, k)$ which are algebraic stacks of finite type with affine diagonal. Thus $\cM(\chi, N, \bfa, c)$ is an algebraic stack locally of finite type with affine diagonal. The proof is finished.
	\end{proof}
	
	We now show that certain changes of the coefficient set induce  finite morphisms of  stacks.
	The result will be used to reduce various results to the case when $l=1$.
	
	\begin{prop}\label{prop:marking-change}
		Let $d, l'\in \bZ_{>0}$, $N'\in (dN) \bZ_{>0}$, $\bfa'\in \frac{1}{N'}\bZ_{>0}^{l'}$, and $c':= \frac{c}{d}$ (hence $c'\in \frac{1}{N'}\bZ_{>0}$). Let $\rmA = (\rmA_{i,j})$ be an $l'\times l$-matrix with non-negative integer entries  such that $\bfa = \bfa' \rmA$. 
		Then there is a finite representable morphism of algebraic stacks 
		\[
		\Phi:\cM(\chi, N, \bfa, c)\to \cM(\chi, N', \bfa', c')
		\]that maps $[(X,\sum_{i=1}^{l} a_i B_i + cD_1)\to T]$ to $[(X, \sum_{i=1}^{l'} a_i' B_i' + c' D_1')\to T]$ where $B_i':= \sum_{j=1}^{l} \rmA_{i,j} B_j$ and $D_1':=d D_1$.
	\end{prop}

	\begin{proof}
		We first show that $\Phi$ is a morphism of algebraic stacks. We denote by 
		\[
		B := \sum_{i=1}^{l} a_i B_i,\quad D:=cD_1, \quad B':=\sum_{i=1}^{l'} a_i' B_i' , \quad\textrm{and}\quad D' := c'D_1'.
		\]
		It suffices to show that if  $[(X, B+D)\to T]$ is a family of boundary polarized CY pairs, then so is $[(X, B' + D')\to T]$. By Remark \ref{rem:moduli-stack-def}(4), we may assume that $T$ is a Noetherian scheme. Since $\Phi$ does not change the unmarked pairs at all, we know that Definition \ref{d:familyofbpcys}(1)(3) hold. 
		Since the sum of finitely many K-flat divisors is still K-flat by \cite[7.4.5]{Kol23}, Definition \ref{d:familyofbpcys}(2) holds. It is clear that $N'(B'+D')= \frac{N'}{N} (N(B+D))$ as relative Mumford divisors with $\frac{N'}{N}\in \bZ_{>0}$, $m\bfa'\in \bZ^{l'}$ implies $m\bfa = m \bfa' \rmA\in \bZ^l$, and $qD_1' = (qd) D_1$. Thus Definition \ref{d:familyofbpcys}(4)(5) hold. This completes the proof that $\Phi$ is a morphism of algebraic stacks.
		
		It remains to  show that $\Phi$ is a finite representable morphism. Since every $B_i'$ and $D_1'$ are integer combinations of $B_j$'s and $D_1$, we know that $\Phi$ induces an injective group homomorphism $\Aut(X, B+D) \to \Aut(X, B'+D')$ for a point $[(X, B+D)]$ in $\cM(\chi, N, \bfa, c)$. Thus $\Phi$ is representable. Next, we show that $\Phi$ is of finite type. From the proof of Theorem \ref{thm:moduli-alg-stack}, there is an open covering of $\cM(\chi, N', \bfa', c')$ by algebraic stacks of finite type $\cM(\chi, N', \bfa', c', \bfb', k)$ where $\bfb'\in \bQ_{\geq 0}^{l'}$ and $k\in \bN$ such that $k\bfa'\in \bZ^{l'}$.  Suppose $[(X, B+D)]$ is a point in $\Phi^{-1}(\cM(\chi, N', \bfa', c', \bfb', k))$, then $(X, B+D)$ and $(X, B'+D')$ being isomorphic as unmarked pairs implies that $(X, \Supp(B+D))$ belongs to a bounded family. Since $B$ and $D$ have fixed coefficients and $B+D\sim_{\bQ} -K_X$, we know that the degree of each $B_i$ and $D_1$ are bounded from above which implies that $(X,B+D)$ belongs to a bounded family. Thus $\Phi^{-1}(\cM(\chi, N', \bfa', c', \bfb', k))$ is of finite type, and hence $\Phi$ is of finite type.
		
		Next, we show that $\Phi$ is quasi-finite and proper. The quasi-finiteness is clear because $\Supp(B) = \Supp(B')$ and $D_1'= dD_1$, and the degree of $B_i$ is bounded from above by $B+D\sim_{\bQ} -K_X$. We check the valuative criterion for properness.
		Let $\mu_R':\Spec R\to \cM(\chi, N', \bfa', c')$ be a morphism where $R$ is a DVR essentially of finite type over $\bk$ with fraction field $K$, such that there is a lifting morphism $\mu_K:\Spec K \to \cM(\chi, N, \bfa, c)$ of $\mu_K' = \mu_R'|_{\Spec K}$. Thus we have families of boundary polarized CY pairs $[(X_R, B_R'+D_R')\to \Spec R]$ and $[(X_K, B_K+D_K)\to \Spec K]$ by pullbacks of universal families under $\mu_R'$ and $\mu_K$ respectively. Suppose $B_K = \sum_{i=1}^l a_i B_{i,K}$ and $D_K = D_{1,K}$. Let $B_{i, R}$ (resp. $D_{1,R}$) be the divisorial scheme on $X_R$ as the closure of $B_{i, K}$ (resp. $D_{1,K}$). Hence $B_{i,R}$ and $D_{1,R}$ are  flat over $R$.
		Since $\Supp(B_{i,K})\subset \Supp(B_R')$ and $\Supp(D_{1,K})\subset \Supp(D_R')$, we know that $\Supp(B_{i,R})\subset \Supp(B_R')$ and $\Supp(D_{1,R})\subset \Supp(D_R')$.  Since $B_R'$ and $D_R'$ are $\bQ_{>0}$-combinations of relative Mumford divisor on $X_R/R$, we know that $B_{i, R}$ and $D_{1,R}$ do not contain codimension $1$ singular points in each fiber. This together with flatness  implies that $B_{i,R}$ and $D_{1,R}$ are relative Mumford divisors on $X_R/R$. Let $B_R:=\sum_{i=1}^l a_i B_{i, R}$ and $D_R:=cD_{1,R}$. By generic flatness of $B_R'$ and $D_R'$, we know that $(X_R, B_R + D_R) $ and $(X_R, B_R'+D_R')$ are isomorphic as unmarked pairs since they agree over $\Spec K$. By Remark \ref{rem:moduli-stack-def}(3) and \cite[Lemma 2.11]{BABWILD}, the family $[(X_R, B_R + D_R) \to \Spec R]$ gives an extension $\mu_R:\Spec R\to \cM(\chi, N, \bfa, c)$ of $\mu_K$ that lifts $\mu_R'$. This proves the existence part of the valuative criterion for properness. The uniqueness part follows from generic flatness of relative Mumford divisors.
		
		To summarize, we have shown that $\Phi$ is representable, of finite type, quasi-finite, and proper. Thus $\Phi$ is a finite morphism. The proof is finished.
	\end{proof}

	\begin{thm}\label{t:stackisScomp+Thetared}
		$\cM(\chi, N, \bfa, c)$ is S-complete, $\Theta$-reductive, and satisfies the existence part of the valuative criterion for properness, all with respect to DVRs essentially of finite type over $\bk$.
	\end{thm}
	
	\begin{proof}
		The theorem follows 
		immediately from combining  \cite[Theorems 3.7, 5.4, and 6.3]{BABWILD} with \cite[Lemma 2.11]{BABWILD}. Note that statements cited  make no assumption that the marking on the log Fano  boundary is by a single Mumford divisor and so they can be applied in our setting. 
		
		Alternatively, the result can be deduced from \cite[Theorem 1.1]{BABWILD}, which only states the result when the marking on the log Fano boundary is by a single divisor. 
		Indeed, let $\lcd(\bfa)\in \bZ_{>0}$ be the least common denominator of $a_1, \cdots, a_l$.
		By \cite[Theorem 1.1]{BABWILD} and Remark \ref{rem:moduli-stack-def}(1), we know that $\cM(\chi, N, \frac{1}{\lcd(\bfa)}, c)$ satisfies the statement. By Proposition \ref{prop:marking-change}, there is a finite morphism $\Phi: \cM(\chi, N, \bfa, c)\to \cM(\chi, N, \frac{1}{\lcd(\bfa)}, c)$ that maps $[(X,\sum_{i=1}^{l} a_i B_i + cD_1)\to T]$ to $[(X, \frac{1}{\lcd(\bfa)}B'_1 + cD_1)\to T]$ where $B'_1:= \sum_{i=1}^{l} (\lcd(\bfa) a_i) B_i$. Thus we have that  $\cM(\chi, N, \bfa, c)$ also satisfies the statement as a finite morphism is always S-complete, $\Theta$-reductive, and proper by \cite[Propositions 3.21(1) and 3.44(1)]{AHLH23}.
	\end{proof}
	
	\subsubsection{Boundedness}
	
	We now prove a  criterion for when  substacks  of the moduli stack of boundary polarized CY pairs  are of finite type.
	
	\begin{prop}\label{p:boundedness}
		A locally closed substack $\cZ \subset  \cM(\chi, N,{\bf a},c)$ is of finite type if and only if there exists a positive integer $r$ such that $r(K_X+B)$ is Cartier for every   $(X,B+D)\in \mathfrak{\cZ}(\bk)$.
	\end{prop}

	The reverse implication is a consequence of an effective very ampleness result for slc log Fano pairs in \cite{Fuj17}.
	
	\begin{proof}
		By the proof of Theorem \ref{thm:moduli-alg-stack}, $\cM(\chi, N,{\bf a},c)$ admits a cover by finite type open substacks $\cM(\chi,N,{\bf a},c,{\bf b},k)\subset \cM(\chi, N,{\bf a},c)$ parametrizing pairs $(X,B+D)$ in $\cM(\chi,N,{\bf a}, c)(\bk)$ 
		such that $-k(K_X+B)$ is a very ample Cartier divisor and $\deg(B_{i}):= (-K_X-B)^{\dim X-1} B_i=b_i$.
		Since $\cZ$ is finite type and, in particular, quasi-compact, $\cZ$ is contained in the union of finitely many of these substacks. 
		Thus there exists a positive integer $r$ such that $r(K_X+B)$ is Cartier for all $(X,B+D)\in\cZ(\bk)$.
		
		Conversely, assume the existence of such an integer $r$.
		After replacing $r$ with a multiple, we may assume that $N$ divides $r$.
		By \cite[Theorem 1.9]{Fuj17}, after replacing $r$ with a further multiple, we may assume that $-r(K_X+B)$ is very ample for all $(X,B+D)\in \cZ(\bk)$.
		
		We claim that there exists an integer $d$ such that  $  (-K_X-B)^{\dim X-1} \cdot (-K_X) \leq d$ for 
		all $(X,B+D) \in \cZ(\bk)$.
		Indeed, for each $(X,B+D)\in \cZ(\bk)$,  $L =-r(K_X+B)$ is a very ample Cartier divisor and has  Hilbert function $\chi(r\, \cdot \, )$.
		Thus, by properties of the Hilbert scheme, there exists a flat proper morphism of finite type schemes $Y\to Z$ and a line bundle $M$ on $Y$ such that  $(X, \cO_X(-r(K_X+B))$ is a fiber of $(Y,M)\to Z$ for all $(X,B+D)\in\cZ(\bk)$.
		By replacing $(Y,M)\to Z$ with a base change by  $Z' \hookrightarrow Z$, where $Z'$ is a disjoint union of finitely many closed subsets of $ Z$,  
		we may assume that $Z$ is smooth, the fibers of $Y\to M$ are demi-normal \cite[Corollary 10.42]{Kol23}, and $K_{Y/Z}\vert_{Y_z} = K_{Y_z}$ for  all $z\in Z$ \cite[Corollary 3.36]{Kol23}. 
		By \cite[Theorem 4.3.5]{Kol23}, the function   $Z\ni z \mapsto M_z^{n-1} \cdot (-K_{Y_z})$ is locally constant.
		Therefore there exists an integer $d$ satisfying the claim.
		
		Now, if $(X,B+D)\in\cZ(\bk)$, then 
		\[
		c_i\deg(B_i) \leq (-K_X-B)^{n-1}\cdot B
		\leq 
		(-K_X-B)^{n-1}\cdot (-K_X)\leq  d
		,\]
		where the second inequality uses that $-K_X-B\sim_{\bQ} D$, which is effective.
		Thus  $\deg(B_i) \in k^{-\dim X+1} \bZ \cap [0,d/c_i] $ and, hence,  can take finitely many possible values. 
		Therefore  $\cZ$ admits a cover by finitely many 
		of the substacks $\cZ\cap \cM(\chi,N,{\bf a},c,{\bf b},k)$, which are open in $\cZ$.
		Since each $\cZ\cap \cM(\chi,N,{\bf a},c,{\bf b},k)$ is of finite type, $\cZ$ is of finite type.
	\end{proof}

	\subsubsection{KSBA  substack}\label{ss:KSBA+K}
	We now discuss a substack defined using KSBA theory for the perturbed pair. 
	See \cite{Kol13} for a reference on KSBA moduli theory.

	\begin{defn}
		Let $\cM \subset \cM(\chi, N,{\bf a}, c)$ be a closed substack. 
		The \emph{KSBA} substack of $\cM$ is the the substack $\cM^{\rm KSBA}\subset \cM$ consisting of families $(X,B+D) \to T$ in $\cM$ such that, for each $t\in T$, $(X_t,B_t+(1+\varepsilon)D_t)$ is slc for $0<\varepsilon \ll1$.\footnote{The condition that  $(X_t,B_t+(1+\varepsilon)D_t)$  is slc is equivalent to the condition that $(X_t,B_t+(1+\varepsilon)D_t)$ is a KSBA-stable pair, since  $K_{X_t}+B_t+ (1+\varepsilon)D_t\sim_{\bQ} \varepsilon D_t$ is ample. } 
	\end{defn}
	
	The  following theorem is a consequence of  general KSBA theory combined with boundedness results in \cite{KX20,Bir22,Bir23} that allow the $\varepsilon$ coefficient.
	
	\begin{thm}\label{t:KSBA}
		\hfill
		\begin{enumerate}
			\item The stack $\cM^{\rm KSBA}$ is a finite type open substack of $\cM$. 
			\item There exists a coarse moduli space morphism $\cM^{\rm KSBA}\to M^{\rm KSBA}$ to a projective shceme.
		\end{enumerate}
	\end{thm}

	\begin{proof}
		For a pair $(X,B+D)$ in $\cM(\bK)$ for some field extension $\bk\subset \bK$, $(X,B+D)$ is slc and so the following are equivalent:
		\begin{enumerate}
			\item[(i)] $(X,B+(1+\varepsilon)D)$ is slc for all $0< \varepsilon \ll1$
			\item[(ii)]  $(X,B+(1+\varepsilon)D)$ is slc for some $0<\varepsilon <1$. 
		\end{enumerate}
		Thus the openness of slc singularities  in families of pairs \cite[Corollary 4.45]{Kol23} implies that $\cM^{\rm KSBA} \subset \cM$ is an open substack.
		Next, \cite[Corollary 1.8]{Bir22} implies that there exists a positive integer $r$ such that $-r(K_X+B)$ is a very ample Cartier divisor for every $(X,B+D)$ in $\cM^{\rm KSBA}(\bk)$.
		Therefore $\cM^{\rm KSBA}$ is finite type by Proposition \ref{p:boundedness}, which completes (1). 
		
		The  stack $\cM^{\rm KSBA}$ is separated by \cite[Proposition 2.50]{Kol23} and proper by \cite[Lemma 7]{KX20}.
		Furthermore, since the stabilizer groups of $\bk$-points in $\cM^{\rm KSBA}$ are discrete by \cite[Proposition 8.64]{Kol23}, $\cM^{\rm KSBA}$ is a Deligne--Mumford stack. 
		Thus \cite{KM97} implies the existence of a coarse moduli space morphism $\cM^{\rm KSBA}\to M^{\rm KSBA}$  to a proper algebraic space. Finally, $M^{\rm KSBA}$ is a projective scheme scheme as a consequence of \cite{KP17}.
	\end{proof}
	
	\subsubsection{K-moduli substack}
	Next, we discuss a substack defined using K-moduli for log Fano pairs.
	See \cite{Xu24} for a reference on K-moduli theory.
	
	\begin{defn}
		Let $\cM \subset \cM(\chi, N,{\bf a}, c)$ be a closed substack. 
		The \emph{K-moduli} substack of $\cM$ is the  substack $\cM^{\rm K}\subset \cM$  consisting of families $(X,B+D) \to T$ in $\cM$ such that, for each $t\in T$, $(X_t,B_t+(1-\varepsilon)D_t)$ is a K-semistable klt log Fano pair for $0<\varepsilon \ll1$. 
	\end{defn}
	
	The following result is a consequence of results in K-moduli theory, as well as boundedness results in  \cite{ADL19,Zho23} to allow the $\varepsilon$ coefficient.
	
	\begin{thm}\label{t:Kmoduli}
		\hfill
		\begin{enumerate}
			\item The stack $\cM^{\rm K}$ is a finite type open substack of $\cM$. 
			\item There exists a good moduli space morphism $\cM^{\rm K}\to M^{\rm K}$ to a projective scheme.
		\end{enumerate}
	\end{thm}
	
	In the following proof, the results in \cite{ADL19,Zho23} that we cite only imply the desired statements when $B=0$. 
	The same arguments extend to the $B\neq 0$ case with only minor changes. 
	
	\begin{proof}
		First, note that if $(X,B+D)$ is in $\cM(\bK)$ for some field extension $\bK\supset \bk$, then the following are equivalent:
		\begin{enumerate}
			\item[(i)] $(X,B+(1-\varepsilon)D)$ is K-semistable for some $0< \varepsilon < 1$; 
			\item[(ii)] $(X,B+(1-\varepsilon)D)$ is K-semistable for all $0< \varepsilon \ll  1$. 
		\end{enumerate}
		Indeed, since $(X,B+D)$ is an slc CY pair, the argument in \cite[Definition 2.4 and Proposition 2.13.1]{ADL19} shows that for every prime divisor $E$ over $X$, the function $\varepsilon\mapsto\beta_{(X, B+(1-\varepsilon)D)}(E)$ is linear for $\varepsilon\in [0,1]$ and non-negative at $\varepsilon=0$. Thus we have that (i) implies (ii), while the implication that (ii) implies (i) is trivial. 
		Therefore the openness of K-semistability in families of log Fano pairs \cite{BLX22,Xu20} (and also \cite{Zhu21} to remove the condition on the geometric fibers) 
		implies that  $\cM^{\rm K}$ is an open substack of $\cM$. 
		
		Next, \cite[Lemma 5.3 and Theorem 5.4]{Zho23} implies that there exists an integer $r$ such that  $-r(K_X+B)$ is a very ample Cartier divisor for every $(X,B+D)\in\cM(\bk)$ with $B=0$.
		Furthermore, the same argument as in \emph{loc. cit.}, but with \cite{Jia20} replaced by \cite[Theorem 7.25]{Xu24}, implies the case when $B\neq 0$. 
		Thus $\cM^{\rm K}$ is finite type by Proposition \ref{p:boundedness}.
		
		Finally, combining 
		\cite[Theorems 3.3 and 5.2]{ABHLX20}, \cite[Corollary 7.4]{BHLLX21}, and \cite[Theorem 1.2]{LXZ22} 
		with 
		Lemma \ref{l:logFanoext} shows that $\cM^{\rm K}$ is S-complete, $\Theta$-reductive, and satisfies the existence part of the valuative criterion for properness with respect to essentially of finite type DVRs over $\bk$.
		Therefore  \cite[Theorem A]{AHLH23} implies that there exists a good moduli space morphism $\cM^{\rm K} \to M^{\K}$ to a proper algebraic space. 
		The moduli space is projective as a consequence of \cite{XZ20,LXZ22}.  
	\end{proof}

	\begin{lemma}\label{l:logFanoext}
		Let $S$ be a  smooth scheme that is essentially of finite type over $\bk$ and $0\in S$ be a closed point such that $S^\circ := S \setminus \{0 \}$ is dense in $S$.
		Let
		$
		(X^\circ ,B^\circ+D^\circ)\to S^\circ
		$
		be a family of boundary polarized CY pairs.

		There exists $0<\varepsilon<1$ such that if $(X,B+(1-\varepsilon)D) \to S$ is an extension of $(X^\circ, B^\circ + (1-\varepsilon) D^\circ)\to S^\circ$ to a family of klt log Fano pairs, then 
		$(X,B+D)\to S$ is a family of boundary polarized CY pairs. 
	\end{lemma}
	
	In the lemma,  the statement that $(X,B+(1-\varepsilon)D)\to S$ is a ``family of  klt log Fano pairs'' means that it is projective family of slc  pairs such that $(X_s,B_s+(1-\varepsilon)D_s)$ is a klt log Fano pair for all $s\in S$.

	\begin{proof}
		After possibly shrinking $S$ in a neighborhood of $0$, there exists a reduced snc divisor $H=H_1 + \cdots + H_r$ on $S$ such that $H_1 \cap \cdots \cap H_r = \{0 \}$. 
		Now, if $(X,B+(1-\varepsilon)D) \to S$ is an extension of the original family to a  family of log Fano pairs, then 
		\[
		(X,B+(1-\varepsilon)D+ f^*(H_1+\cdots + H_r))
		\]
		is lc by \cite[Theorem 4.54]{Kol13}.
		By \cite[Theorem 1.1]{HMX}, there exists $c\in (0,1)$ that only depends on the dimension of $X$ and the coefficients of $B^\circ$ and $D^\circ$ such that if $0<\varepsilon<c$, then 
		\begin{equation}\label{e:XBDH}
			(X,B+D + f^*(H_1+\cdots + H_r))
		\end{equation}
		is lc. 
		Now  assume that $\varepsilon<c$ and so \eqref{e:XBDH} is lc.
		By \cite[Theorem 4.54]{Kol23}, $(X,B+D) \to S$ is a family of slc pairs. 
		Since $(K_{X/S}+B+D )\vert_{S^\circ} \sim_{\bQ,S^\circ}0$,  there exists a $\bQ$-divisor $P$ with support in $X_0$ satisfying
		\[
		K_{X/S}+B+D \sim_{S,\bQ} P
		.\]
		Since $X_0$ is irreducible by the klt assumption, $P\sim_{\bQ,S} 0$.
		Thus $K_{X/S}+B+D \sim_{\bQ,S}0$. 
		By the previous $\bQ$-linear equivalence and that
		$-K_{X/S}- B-(1-\varepsilon)D$  is ample over $S$, we see $D$ is  ample over $S$. Thus $(X,B+D)\to S$ is a family of boundary polarized CY pairs.
	\end{proof}
	
	\subsection{S-equivalence}
	
	The following equivalence relation will be useful for understanding what the points of the moduli space in Theorem \ref{t:mainintro}
	parametrize. 
	
	\begin{defn}[S-equivalence]
		Two boundary polarized CY pairs $(X,B+D)$ and $(X',B'+D')$ are \emph{S-equivalent} if there exist  weakly special degenerations of the pairs 
		\[
		(X,B+D) \rightsquigarrow (X_0,B_0+D_0) \leftsquigarrow (X',B'+D')
		\]
		to a common boundary polarized CY pair $(X_0,B_0+D_0)$.
	\end{defn}
	
	\begin{rem}\label{r:Sequiv}
		We describe a few properties of S-equivalence.
		\begin{enumerate}
			\item S-equivalence is an equivalence relation by \cite[Proposition 6.9]{BABWILD}.
			\item If two pairs $(X,B+D)$ and $(X',B'+D')$ are S-equivalent, then their sources and coregularity are equal \cite[Proposition 8.7]{BABWILD}.
			\item If $x:= [(X,B+D)]$ and $x':= [(X',B'+D')]$ are two $\bk$-points of the stack in Definition \ref{def:moduli-stack}, then $(X,B+D)$ and $(X',B'+D')$ are S-equivalent if and only if  $\overline{\{ x\}} \cap \overline{\{x'\}}\neq \emptyset$ \cite[Lemma 6.10]{BABWILD}.
		\end{enumerate}
	\end{rem}
	
	\subsection{Stability}
	
	We also define a notions of stability, which is  related to S-equivalence. 
	
	\begin{defn}[Stability]
		A boundary polarized CY pair $(X,B+D)$ is called
		\begin{enumerate}
			\item \emph{polystable} if every test configuration $(\cX,\cB+\cD)$ of $(X,B+D)$ 
			satisfies $(X,B+D) \simeq (\cX_0,\cB_0+\cD_0)$ and 
			\item \emph{stable} if every test configuration $(\cX,\cB+\cD)$ of $(X,B+D)$ is trivial. 
		\end{enumerate}
	\end{defn}
	
	\begin{prop}[{\cite[Theorem 6.15]{BABWILD}}]\label{p:stable=klt}
		A pair $(X,B+D)$ is stable if and only if $(X,B+D)$ is klt. 
	\end{prop}

	\begin{prop}[{\cite[Propositions 6.11 and 6.13]{BABWILD}}]\label{p:stable}
		Let   $x:= [(X,B+D)]$ and $x'=[(X',B'+D')]$ be two $\bk$-points of the moduli stack in Definition \ref{def:moduli-stack}.
		\begin{enumerate}
			\item The pair $(X,B+D)$ is polystable if and only if $x$ is a closed point.
			\item The pair $(X,B+D)$ is stable if and only if $x$ is a closed point and $\Aut(x)$ is discrete.
		\end{enumerate}
		Furthermore, if $(X,B+D)$ is stable and $\overline{\{x\}} \cap \overline{\{x'\}}\neq\emptyset$, then $x=x'$. 
	\end{prop}

	\section{Gorenstein index and moduli spaces}\label{s:Goreinsteinindex}
	
	In this section, we will construct a good moduli space for boundary polarized CY surface pairs with bounded Gorenstein index when the log Fano boundary has standard coefficients.
	The construction relies on Theorem \ref{thm:index-deform}, which is a result on the deformation of the Gorenstein index of slc pairs with standard coefficients.

	Throughout this section, we fix the following data:
	\begin{itemize}
		\item A moduli stack $\cM(\chi,N, {\bf a},c)$   in Definition \ref{def:moduli-stack} such that there exists $v>0$ satisfying $\chi(m  N) = v  m^2 + O(m)$ for $m\geq 0$ and the coefficients  $a_1,\ldots, a_l$ are in  $\mathbf{T}:= \{\tfrac{1}{2}, \tfrac{2}{3}, \tfrac{3}{4},\ldots \}\cup\big\{1\big\}$. 
		\item A finite type locally closed substacks
		$
		\cM^\circ \subset \cM(\chi,N, {\bf a},c)
		$.
	\end{itemize}
	We write $\cM$ for the stack theoretic closure  of $\cM^\circ$ in $\cM(\chi,N, {\bf a},c)$.
	Note that pairs $(X,B+D)$ in $\cM(\bk)$ are boundary polarized CY surface pairs and the marked coefficients of $B$ are  in $\bfT$.
	
	\subsection{Bounded Gorenstein index substack}\label{ss:indexboundsubstack}
	
	We now proceed to a define a finite type substack of $\cM$ by bounding the Gorenstein index of the underlying log Fano pairs.
	
	\begin{defn}
		Let $(X,B)$ be a marked slc pair. The \emph{Gorenstein index} of $K_X+B$ at a point $x\in X$, denoted by $\ind_x(K_X+B)$, is the smallest positive integer $m$ such that there exists an open neighborhood $x\in U\subset X$ satisfying that $mB|_U$ is a marked $\bZ$-divisor and $\omega_X^{[m]}(mB)|_U$ is locally free. 
	\end{defn}

	\begin{defn}
		For each integer $m\geq1$, let $\cM_m \subset \cM$ denote the substack consisting of families $f:(X,B+D)\to T$ in $\cM$ such that $\ind_{x}(K_{X_{t}}+B_{t})\leq m$ for all $x\in X$ and $t=f(x) \in T$.
	\end{defn}

	\begin{proposition}\label{p:cM_mfinitetype}
		For each integer $m\geq 1$,
		$\cM_m$ is an open substack of $\cM$ 
		and is a finite  type algebraic stack with affine diagonal.
	\end{proposition}
	
	Note that $\cM$ is in general not finite type even in very simple examples, while $\cM_m$ is always finite type by Proposition \ref{p:cM_mfinitetype}. See \cite[Example 9.4]{BABWILD} and Remark \ref{r:notfinitetype}.
	
	\begin{proof}
		To verify the openness, it suffices to show that if $(X,B+D)\to T$ is a family in $\cM$, then 
		\[
		U:= \{ t\in T\, \vert\, {\rm ind}_x(K_{X_t} +B_t) \leq m \text{ for all } x\in X_t\}
		\]
		is open in $T$. 
		By Remark  \ref{rem:moduli-stack-def}, we may assume that $T$ is Noetherian. 
		Now, for $1\leq j\leq m$, let $U_j\subset X$ denote the open locus where $jB$ is a marked $\bZ$-divisor and  $\omega_{X/T}^{[j]}(jB)$ is an invertible sheaf. 
		Now for $t\in T$ and $x\in X_t$, $j(K_{X_t}+B_t)$ is Cartier at $x$ if and only if $\omega_{X_t}^{[j]}(jB_t)$ is a line bundle.
		Since $\omega_{X/B}^{[j]}\vert_{X_b} \simeq \omega_{X_b}^{[j]}$, Nakayama's Lemma implies that the latter holds if and only if $x\in U_j$. Thus $T^\circ = T\setminus f( \cap_{j=1}^m Z_j)$, which is open.

		Since $\cM_m$ is an a locally closed substack of $\cM(\chi, N,{\bf a}, c)$, which is an algebraic stack with affine diagonal by Theorem \ref{t:algebraicstack}, $\cM_m$ satisfies the same properties. 
		In addition, Proposition \ref{p:boundedness}, which follows from \cite{Fuj17}, implies that $ \cM_m$ is of finite type. 
		Alternatively, since we are in dimension two, the necessary boundedness can be deduced from \cite{Kol85}.
	\end{proof}
	
	\begin{rem}
		In the special case when the $\bk$-points of  $\cM^\circ$ parametrize pairs $(\bP^2, \tfrac{3}{d} C)$, where $C\subset \bP^2$ is a smooth degree $d$ curve, \cite{BABWILD} construct a separated good moduli space $\cM_m \to M_m$ for each $m\geq 1$ by verifying that  $\cM_m$ is S-complete and $\Theta$-reductive. 
		The proof of the latter statement in \emph{loc. cit.} is to first show that for a pair $(X,B+D)$ in $\cM(\bk)$ the divisor $D$ can be replaced with a 1-complement \cite[Section 12.1]{BABWILD} and then uses the deformation theory of twisted curves \cite[Section 11.3]{BABWILD}. 
		
		Since  slc Fano surfaces do not always admit 1-complements \cite[Example 3.15]{Mor24}, the latter approach does not extend to our setting.
		Therefore, we will instead study deformations of the surface rather than curves on it. 
		This gives a more direct approach to proving that $\cM_m$ is S-complete and $\Theta$-reductive. 
	\end{rem}

	\subsection{Gorenstein index under deformation}
	
	
	\begin{thm} \label{thm:index-deform}
		Let $g:(X,B) \to S$ be a relative dimension $2$
		family of marked slc pairs over a smooth surface $S$ essentially of finite type over $\bk$.  Assume that $B$ is a marked $\bQ$-divisor with coefficients in the set $\bfT = \{ \tfrac{1}{2} , \tfrac{2}{3} , \tfrac{3}{4},\ldots\}\cup\{1 \}$. (We  allow the case when $B = 1\cdot \emptyset $.)
		
		If $0\in S$ is a closed point
		and $x_0\in X_0$ is a closed point that is not an lc center of $(X_0, B_0)$, then there exists a curve $x_0\in C\subset X$ not contained in $X_0$ such that
		\[
		\ind_{x_0}(K_{X_0}+B_0) = \ind_{x}(K_{X_s}+B_s)
		\]
		for all $x\in C$ and $s=g(x)$. 
	\end{thm}
	
	\begin{proof}
		Denote by $r$ the Gorenstein index of $K_{X_0}+B_0$ at $x_0$. 
		Since $(X,B) \to S$ is a family of slc pair of relative dimension $2$ with standard coefficients, by \cite[Theorem 1]{Kol18} (see also \cite{KK23} for a generalization to higher dimensions) we know that the sheaves $\omega_{X/S}^{[m]}(\lfloor mB\rfloor)$ are flat over $S$ and commute with base change for every $m\in \bZ$.\footnote{Note that the round down is taken with respect to the marking, i.e. if $B= \sum_i a_i B_i$, where the $B_i$ are relative Mumford divisors and $a_i \in \bQ_{\geq 0}$, then $\lfloor m B \rfloor := \sum_i \lfloor m a_i \rfloor B_i$. 
			} 
			In particular, for every $s\in S$ we have
			\begin{equation}\label{eq:kollar-condition}
				\omega_{X/S}^{[m]}(\lfloor mB\rfloor)\otimes \cO_{X_s}\cong \omega_{X_s}^{[m]}(\lfloor mB_s\rfloor).
			\end{equation}
			Thus $\omega_{X/S}^{[m]}(\lfloor mB\rfloor)$ is locally free at $x_0$ if and only if $\omega_{X_0}^{[m]}(\lfloor mB_0\rfloor)$ is locally free at $x_0$. Moreover, since $(X,B)\to S$ is a family of marked pairs, we know that the coefficient set of $B_0$ at $x_0$ is the same as the coefficient set of $B$ at $x_0$. In particular, we know that $mB_0$ is a $\bZ$-divisor at $x_0$ if and only if $mB$ is a $\bZ$-divisor at $x_0$. This together with \eqref{eq:kollar-condition} shows that 
			\[
			\ind_{x_0}(K_{X/S}+B) = \ind_{x_0} (K_{X_0}+B_0) = r.
			\]

			Since the problem is local, after shrinking $X$ around $x_0$ we may assume that $\omega_{X/S}^{[r]}(rB)$ is a trivial line bundle.
			Following \cite[Definition 2.49]{Kol13}, we take the index $1$ cover $\pi: (\tX,\tB) \to (X,B)$ of $K_{X/S}+B$ which is a Galois cover with Galois group $\bmu_r$. More precisely, we have 
			\[
			\tX = \Spec_{X} \cO_{X} \oplus \omega_{X/S}^{[1]}(\lfloor B \rfloor)\oplus \cdots \oplus \omega_{X/S}^{[r-1]}(\lfloor (r-1)B \rfloor)
			\]
			where the ring structure is induced by a fixed isomorphism $\gamma: \omega_{X/S}^{[r]}(rB) \xrightarrow{\cong} \cO_{X}$. Moreover, from \cite[Proposition 2.50(3)]{Kol13} we know that $\omega_{\tX/S}(\tB)$ is locally free at $x_0$, $K_{\tX/S} + \tB \sim_{\bQ} \pi^*(K_{X/S} +B)$,  and $\tB = \pi^* \lfloor B \rfloor $ is a reduced $\bZ$-divisor such that $(\tX, \tB)\to S$ is a locally stable family. Again by \cite[Theorem 1]{Kol18} we know that $(\tX_s, \tB_s)\to (X_s, B_s)$ is an index $1$ cover of $K_{X_s} + B_s$ for $s$ in a neighborhood of $0$ in $S$. 
			By properties of index $1$ covers (see \cite[2.48]{Kol13}), we know that the index of $K_{X/S}+B$ at $x$ equals $r$ if and only if the $\bmu_r$-action on $(\pi^{-1}(x))_{\rm red}$ is trivial.
			
			Since  $(X_0, B_0)$ is slc at $x_0$ whose minimal lc center is not a point, the same is true for $\tx_0\in (\tX_0, \tB_0)$ where $\tx_0$ is the unique preimage of $x_0$ under $\pi$. We claim that $\tx_0\in \tX_0$ is a hypersurface singularity. If $\tx_0\not\in \Supp(\tB_0)$, then  $\tx_0\in \tX_0$ 
			is semi-log-terminal in the sense of \cite[Definition 4.17]{KSB88}. By \cite[Theorem 4.23]{KSB88} we know that $\tx_0 \in \tX_0$ is a hypersurface singularity. If $\tx_0\in \Supp(\tB_0)$, then $\tx_0\in (\tX_0, \tB_0)$ is not klt as $\tB_0$ is reduced. Thus the minimal lc center of $(\tX_0, \tB_0)$ near $x_0$ is a curve, which implies that it is plt near $\tx_0$. Since $K_{\tX_0} + \tB_0$ is Cartier near $\tx_0$, by adjunction we know that $(\tB_0, {\rm Diff}_{\tB_0} (0) )$ is klt where ${\rm Diff}_{\tB_0} (0)$ is a $\bZ$-divisor. By \cite[Proposition 16.6]{Kol92} we know that ${\rm Diff}_{\tB_0} (0) =0$ and $\tB_0$ is a  smooth curve and Cartier in $\tX_0$ near $\tx_0$. This implies that $\tX_0$ is smooth at $\tx_0$. 
			Thus the claim is proved. From the claim we know that $\tx_0\in \tX$ is also a hypersurface singularity by smoothness of $S$.
			
			Next, we analyze the $\bmu_r$-action on $\tX$ near $\tx_0$. Denote by $(A,\fm) = (\cO_{\tX, \tx_0}, \fm_{\tx_0})$. From above arguments, we know that the $\bmu_r$-action on $A/\fm = k(\tx_0)$ is trivial. Let $w_1, w_2\in \cO_{S,0}$ be a regular system of parameters. Denote by $\tw_i=\pi^* w_i \in \fm$ for $i=1,2$. By the construction of the index $1$ cover, we know that each $\tw_i$ is $\bmu_r$-invariant. Since $\tx_0\in \tX_0$ is a hypersurface singularity of dimension $2$, we know that $\edim(A/(\tw_1,\tw_2), \fm/(\tw_1,\tw_2))\leq 3$. As a result, we can find $\bmu_r$-eigenvectors $\oz_1,\oz_2,\oz_3$ that span the $k(\tx_0)$-vector space $\fm/(\fm^2+(\tw_1,\tw_2))$. Since $\bmu_r$ is linearly reductive, there exists a $\bmu_r$-equivariant splitting of  the  quotient map $\fm\twoheadrightarrow \fm/(\fm^2+(\tw_1,\tw_2))$. In particular, there exists a lifting  $z_j\in \fm$ of $\oz_j$ such that $z_j$ is a $\bmu_r$-eigenfunction for $1\leq j\leq 3$. 
			
			Let $\tZ = V(z_1, z_2, z_3)\subset \tX$. Then $\tZ$ is $\bmu_r$-invariant, $\tx_0\in \tZ$, $\cO_{\tZ, \tx_0}\cong A/(z_1,z_2,z_3)$, and $\dim_{\tx_0} \tZ\geq \dim_{\tx_0} \tX - 3 = 1$. We claim that after possibly shrinking $\tX$ around $\tx_0$, the $\bmu_r$-action on $\tZ$ is trivial. Denote by $\fn:=\fm_{\tZ, \tx_0}$.  Then we have $\fn/\fn^2 \cong \fm/(\fm^2 + (z_1,z_2,z_3))$ which is a $k(\tx_0)$-vector space  spanned by (cosets of) $\tw_1$ and $\tw_2$. In particular, the $\bmu_r$-action on $\fn/\fn^2$ is trivial. Since there is a $\bmu_r$-equivariant surjective ring homomorphism $k(\tx_0) [[\fn/\fn^2]] \twoheadrightarrow \widehat{\cO_{\tZ, \tx_0}}$, we know that $\bmu_r$ acts trivially on the completion $\widehat{\cO_{\tZ, \tx_0}}$. Hence $\bmu_r$ acts trivially on the subring $\cO_{\tZ,\tx_0}$, which implies the claim. 
			
			Finally, let $Z= \pi(\tZ)$. Since $\tZ\cap X_0 = V(z_1, z_2, z_3, \tw_1, \tw_2)$ is supported at the point $\tx_0$, we know that $\tZ$ is not contained in $\tX_0$. Hence $x_0\in Z$, $\dim Z =\dim \tZ\geq 1$, and $Z$ is not contained in $X_0$. Since $\bmu_r$ acts trivially on $\tZ$, the index of $K_{X/S}+B$ at every point $x\in Z$ is equal to $r$. By \eqref{eq:kollar-condition}, we know that $K_{X_s} + B_s$ has index $r$ at every point $x\in Z$ where $ s= g(x)$. Thus we can find a curve $x_0\in C$ contained in $Z$ such that $C\not\subset X_0$ and $K_{X_s}+B_s$ has index $r$ at any $x\in C$. The proof is finished.
		\end{proof}

		\subsection{Existence of moduli space}

		\begin{thm}\label{t:existsGMindex}
			If $m\geq N$, then there exists a good moduli space morphism 
			$
			\phi_m:\cM_m \to M_m
			$
			to a separated finite type algebraic space. 
		\end{thm}
		
		To prove the result, we will show that $\cM_m$ is S-complete and $\Theta$-reductive with respect to DVRs essentially of finite type over $\bk$
		by combining Theorem \ref{t:stackisScomp+Thetared} with the following  consequence of Theorem \ref{thm:index-deform}. 
		
		\begin{prop}\label{p:cM_mextension}
			Let $S$ be a smooth surface essentially of finite type over $\bk$,  $0 \in S$ a closed point, and $S^\circ:= S\setminus \{0\}$.
			
			If $m \geq N$ and $(X,B+D)\to S$ is   a family in $\cM$ such that 
			its restriction  $(X^\circ, B^\circ+D^\circ)\to S^\circ$ is in $\cM_m$,  then $(X,B+D)\to S$ is in $\cM_m$. 
		\end{prop}
		
		\begin{proof}
			Assume to the contrary that $(X,B+D)\to S$ is not $\cM_m$. 
			Then there exists a closed point $x\in X_0$ such that $m_0 := {\rm ind}_{x_0}(K_{X_0}+B_{0}) >m$. 
			Since  $N(K_{X_0}+B_0+D_0) \sim 0$ and $m\geq N$, 
			$x_0$ must be in $\Supp(D_0)$ and so  $x_0$ is not an lc center of $(X_0,B_0)$. 
			By Theorem \ref{thm:index-deform}, there exists a curve $x_0\in C\subset X$ not contained in $X_0$ such that $\ind_{x} (K_{X_{s}} + B_{s}) = m_0$ for every $x\in C$ and $s\in g(s)$. Thus $(X^\circ,B^\circ+D^\circ)\to S^\circ$ is not in $\cM_m$, which is a contradiction.
		\end{proof}

		\begin{proof}[Proof of Theorem \ref{t:existsGMindex}]
			Let $R$ be a DVR essentially of finite type over $\bk$ with uniformizer $\pi$. Set
			\[
			S:= \Spec (R[s,t]/(st-\pi)) \quad \text{ or } \quad  S:= \Spec( R[t]),
			\]
			where $\bG_m$ acts on $S$ with weight $1$ and $-1$ on $s$ and $t$, respectively. Let $0 \in S$ denote the unique closed fixed point and $S^\circ = S\setminus \{0\}$. Thus $[S/\bG_m]$ equals $\STR$ or $\Theta_R$.
			
			Let $(X^\circ, B^\circ+D^\circ)\to S$ be a $\bG_m$-equivariant family in $\cM_m$.
			By Theorem \ref{t:stackisScomp+Thetared}, the family extends uniquely to a $\bG_m$-equivariant family $(X,B+D)\to S$ in $ \cM(\chi, N,{\bf a},c)$. 
			Since $\cM$ is a closed substack  and $S$ is reduced, $(X,B+D)\to S$ is in $\cM$.
			By Proposition \ref{p:cM_mextension}, $(X,B+D)\to S$ is in $\cM_m$. 
			Therefore $\cM_m$ is S-complete and $\Theta$-reductive with respect to DVRs essentially of finite type over $k$.
			Using additionally that $\cM_m$ is a finite type algebraic stack with affine diagonal by Proposition \ref{p:cM_mfinitetype},
			Theorem \ref{t:AHLH} now implies the existence of a good moduli space morphism $\cM_m\to M_m$ to  a separated finite type algebraic space.
		\end{proof}
		
		\subsection{Properties of moduli space}
		
		We now discuss certain properties of the moduli space constructed in Theorem \ref{t:existsGMindex}. In particular, we construct various subspaces, wall crossing maps, and prove a properness result.

		\subsubsection{Type I and II subspaces}

		\begin{defn}
			Let $\cM^{\rm I}\subset \cM$ (resp., $\cM^{\rm I+II}\subset \cM$) denote the substack consisting of families $(X,B+D)\to T$ in $\cM$ such that $(X_t, B_t+D_t)$ is Type I (resp., Type I or II) for all $t\in T$.  
		\end{defn}
		
		\begin{proposition}\label{p:TypeI+IIfinitetype}
			The subtstacks $\cM^{\rm I}$ and $\cM^{\rm I+II}$ are finite type open substacks of $\cM$. 
		\end{proposition}
		
		\begin{proof}
			By \cite[Proposition 8.8]{BABWILD}, $\cM^{\rm I} \subset \cM^{\rm I+II} \subset \cM$ are open substacks of $\cM$. 
			In the case when ${\bf a} = (a_1)$, 
			\cite[Theorem 8.15]{BABWILD} implies that $\cM^{\rm I +II}$ is finite type. (This result does not require that $\cM$ parametrizes pairs  of dimension 2 and that $a_1 \in \bfT$.)
			
			We will reduce to the above special case by altering the marking. 
			Let ${\rm lcd}(\bf a)$ be the least common denominator of $a_1,\ldots, a_\ell$ and ${\bf a'}= 1 / {\rm lcd}({\bf a})$.
			By Proposition \ref{prop:marking-change}, there is a finite representable morphism of algebraic stacks 
			\[
			\Phi: \cM(\chi, N,{\bf a},c) \longrightarrow \cM(\chi, N,{\bf a'}, c)
			\]
			that sends  a family $(X,B+D) \to T$ to $(X, a' ({\rm lcd}({\bf a}) B)+D)\to T$.
			Let $\cM' = \Phi(\cM)$ denote the stack theoretic image. 
			Since $\cM^\circ$ is finite type and dense in $\cM$, $\cM'$ is the closure of a finite type locally closed substack.
			Since $\Phi$ preserves the type of a pair, $\cM^{\rm I+II} = \Phi^{-1} (\cM'^{\rm I+II})$ and so there is an induced finite morphism 
			$
			\cM^{\rm I+II}\to \cM'^{\rm I+II}
			$.
			By \cite[Theorem 8.15]{BABWILD}, $\cM'^{\rm I+II}$ is finite type. 
			Thus $\cM^{\rm I+II}$ is finite type and so is $\cM^{\rm I}$. 
		\end{proof}

		\begin{proposition}\label{p:TypeI+IIsaturated}
			If $m\gg0$, then $\cM^{\rm I+II}\subset \cM_m$. Furthermore, the open substacks $\cM^{\rm I}$ and $\cM^{\rm I+II}$  are both saturated with respect to $\phi_m:\cM_m\to M_m$ and the maps
			\[
			\cM^{\rm I}\to \phi_m(\cM^{\rm I}) \quad \text{ and } \quad 
			\cM^{\rm I+II}\to \phi_m(\cM^{\rm I+II})
			\]
			are coarse and good moduli morphisms, respectively.
		\end{proposition}

		\begin{proof}
			Since $\cM^{\rm I+II}$ is quasi-compact  by Proposition \ref{p:TypeI+IIfinitetype} and $\cM_1 \subset \cM_2 \subset \cdots $ is an ascending chain of open substacks with union $\cM$,  $\cM^{\rm I+II}\subset \cM_m$ for $m\gg0$. 
			Next, note that, for any algebraically closed field $\bK\supset \bk$, if the map
			\[
			\cM_m(\bK) \to M_m(\bK)
			\]
			identifies two points $x=[(X,B+D)]$ and $x'=[(X',B'+D')]$ in $\cM_m(\bK)$, then $(X,B+D)$ and $(X',B'+D')$  are S-equivalent and, in particular, have the same type by Proposition \ref{p:propertiesofGMs} and Remark \ref{r:Sequiv}. 
			Thus $\cM^{\rm I}$ and $\cM^{\rm I+II}$ are saturated with respect to $\phi_m$ and so the maps
			\[
			\cM^{\rm I}\to \phi(\cM^{\rm I}) \quad \text{ and } \quad 
			\cM^{\rm I+II}\to \phi(\cM^{\rm I+II})
			\]
			are good moduli spaces by Proposition \ref{p:saturatedopen}. 
			Furthermore, since two Type I pairs are S-equivalent if and only if they are isomorphic by Proposition \ref{p:stable}, the composition 
			\[
			\cM^{\rm I}(\bK) \hookrightarrow \cM_m(\bK) \to M_m(\bK)
			\]
			is injective and so $\cM^{\rm I}\to \phi(\cM^{\rm I})$ is a bijection on geometric points. Therefore $\cM^{\rm I} \to M^{\rm I}$ is a coarse moduli space.
		\end{proof}

		\subsubsection{Wall crossing}	
		
		We now construct a wall crossing diagram with $M_m$ as the model between the the K- and KSBA-moduli spaces.
		
		\begin{proposition}\label{p:Mmwallcrossing}
			If $m\gg0$, then $\cM^{\rm KSBA} \cup \cM^{\rm K}\subset \cM_m$ and there exists a commutative diagram 
			\[
			\begin{tikzcd}
				\cM^{\rm K}\arrow[r,hook] \arrow[d]& \cM_m \arrow[d] & \cM^{\rm KSBA}\arrow[l,hook', swap]\arrow[d]\\
				M^{\K} \arrow[r]& M_{m} & M^{\rm KSBA} \arrow[l],
			\end{tikzcd}
			\]
			where the top row arrows are the open immersions, the bottom row arrows are proper morphisms, and the vertical maps are good moduli space morphisms.
			Furthermore, if the Type I substack $\cM^{\rm I}$ is dense in $\cM$, then the bottom row arrows are birational\footnote{Here, we are using the definition of a \emph{birational} morphism in \cite[\href{https://stacks.math.columbia.edu/tag/0ACV}{Tag 0ACV}]{stacks-project}, which does not require that the spaces are reduced or irreducible.
			} morphisms.
		\end{proposition}
		
		For certain choices of $\cM^\circ$, $\cM^{\K}$ or $\cM^{\rm KSBA}$ may be the empty stack. 
		When this is the case, the left or right vertical arrow is the trivial map between the empty stack and empty scheme.
		
		\begin{proof}
			Since $\cM^{\rm KSBA}$ and $\cM^{\rm K}$ are quasi-compact by Theorems \ref{t:KSBA}  and \ref{t:Kmoduli},  
			\[
			\cM^{\rm KSBA}\cup \cM^{\rm K}\subset \cM_m
			\]
			for $m\gg0$.
			By the universality of good moduli spaces, there exist unique morphisms in the bottom row that make the diagram commute. 
			Furthermore, since $M^{\rm K}$ and $M^{\rm KSBA}$ are proper and $M_m$ is separated, the morphisms in bottom row are  proper.

			To verify the final sentence of the proposition, we will show that the bottom row maps are isomorphsms on the Type I loci.
			First, note that $\cM^{\rm I}\subset \cM^{\KSBA}$ clearly holds. 
			In addition, $\cM^{\rm I} \subset \cM^{\K} $, since  
			if $(X,B+D)$ is a klt boundary polarized CY pair, then $(X,B+(1-\varepsilon) D)$ is K-stable for $0<\varepsilon\ll1$ by \cite[Theorem 2.10]{ADL21} and its  extension to the case when $B\neq 0$.
			By Proposition \ref{p:TypeI+IIsaturated},  $\cM^{\rm I}$ is a saturated open subset of $\cM_m$ for $m\gg0$.
			By the commutativity of the above diagram, $\cM^{\rm I}$ is also a saturated open subset of $\cM^\K$ and $\cM^{\KSBA}$.
			Hence, using Proposition \ref{p:saturatedopen} and the uniqueness of good moduli spaces, the natural maps
			\[
			\phi_{\K}( \cM^{\rm I})  \longrightarrow \phi_m(\cM^{\rm I}) \longleftarrow 
			\phi_{\KSBA}( \cM^{\rm I}) 
			\]
			are isomorphisms between open subsets of $M^{\rm K}$, $M_m$, and $M^{\KSBA}$.
			Since $\cM^{\rm I}$ is dense in $\cM$ by assumption, the three open subsets are dense. 
			Therefore the bottom row maps in the commutative diagram are birational.
		\end{proof}
		
		\subsubsection{Properness}

		\begin{proposition}\label{p:properness}
			If $m\gg0$, then the good moduli space $M_m$ is proper. 
		\end{proposition}
		
		To prove the result, we use the proper morphism $M^{\rm KSBA}\to M_m$ in Proposition \ref{p:Mmwallcrossing} and the  properness of the KSBA-moduli space.
		
		\begin{proof}
			We first prove the result in the special case when $\cM^{\rm KSBA}$ is dense in $\cM$.
			If $m\gg0$, then 
			$\cM^{\rm KSBA}\subset \cM_m$ and there is an induced map $M^{\rm KSBA} \to M_m$  by Proposition \ref{p:Mmwallcrossing}.
			Since $M^{\rm KSBA}$ is proper by Theorem \ref{t:KSBA} and $M_m$ is finite type and separated, the scheme theoretic image of $M^{\rm KSBA}$ in $M_m$ is a closed proper subspace.
			Since $\cM^{\rm KSBA}$ is dense in $\cM_m$ by assumption, the scheme theoretic image of $M^{\rm KSBA}$ in $M_m$ is dense. Therefore $M_m$ is proper.

			To deduce the full result from the above special case, we will  alter the marking on the polarizing boundary divisor.
			Since $\cM^\circ$ is finite type by assumption, there exists a positive integer $d$ such that $d D_1$ is very ample for all $(X,B+D)$ in $\cM^{\circ}(\bk)$.
			Set $c':= \tfrac{c}{d}$ and $N':= dN$. 
			By Proposition \ref{prop:marking-change}, there is a  finite representable morphism of stacks 
			\[
			\Phi:\cM(\chi, N, {\bf a}, c ) 
			\to 
			\cM(\chi, N', {\bf a}, c' ) 
			\]
			that sends a family $(X, \sum_{i=1}^l a_i B_i+cD_1)\to T$
			to $(X,\sum_{i=1}^l a_i B_i+ c'(d D_1)) \to T$.
			
			Let $\cM'^\circ \subset \cM(\chi, N',{\bf a}, c')$ denote the open substack parametrizing pairs $(X,B+D)$ such that  $(X,B+(1+\varepsilon)D)$ is slc for $0<\varepsilon\ll 1$ as in Section \ref{ss:KSBA+K}. By Theorem \ref{t:KSBA}, $\cM'^\circ$ is finite type. 
			Let $\cM'$ denote the stack theoretic closure of $\cM'^{\circ}$ in $\cM(\chi, N', {\bf a}, c' )$. 
			By construction, the  open substack $\cM'^{\rm KSBA} = \cM'^\circ$ and so $\cM'^{\KSBA}$ is dense in $\cM'$.
			For $m\geq N'$, write $\cM'_m\to M'_m$ for the good moduli space morphism.
			By the first paragraph,  $M'_m$ is proper for $m\gg0$ by the first paragraph.
			Thus $\cM'_m$ satisfies the valuative criterion for properness for $m\gg0$ by \cite[Theorem 5.4]{AHLH23}.
			
			We claim that  there is an inclusion of closed substacks $\Phi(\cM) \subset \cM'$. 
			Indeed, if $(X,B+cD_1)$ is in $\cM^\circ(\bk)$, then 
			\begin{enumerate}
				\item[(i)] $(X,B+ c' (dD_1))$ is in $\cM(\chi,N', {\bf a}, c')(\bk)$ and 
				\item[(ii)] $(X,B+ c' H)$ is in $\cM'^\circ(\bk)$ 
				for general $H \in |dD_1|$ by Bertini's Theorem. 
			\end{enumerate}
			Thus $(X,B+c' (dD_1))$ is in $\cM'(\bk)$ and so $\Phi(\cM)\subset \cM'$ as desired.
			Write $\Psi:\cM\to\cM'$ for the induced map.
			Since $\Psi^{-1}(\cM'_m) \cap \cM = \cM_m$,
			the morphism $\Psi\vert_{\cM_m}:\cM_m\to \cM'$  admits  a factorization 
			\[
			\cM_m \overset{\Psi_m}{\longrightarrow} \cM'_m  \lhook\joinrel\xrightarrow{} \cM'
			\]
			with $\Psi_m$ finite.
			Since $\Psi_m$ is proper  and $\cM'_m$ satisfies the valuative criterion for properness for $m\gg0$, $\cM_m$ also satisfies the the valuative criterion for properness for $m\gg0$. 
			Therefore $M_m$ is proper for $m\gg 0$ by \cite[Theorem 5.4]{AHLH23}.
		\end{proof}

		\section{Type II surface pairs}\label{s:TypeII}
		
		In this section we will describe the geometry of polystable Type II surface pairs and then prove Theorem \ref{t:TypeIIfieldofdef}. 
		The latter result will be used in Section \ref{s:asgm} to show that the Type II loci of our moduli space has maximally varying sources.
		
		\subsection{Regularity $0$ pairs}
		
		We first prove a result concerning regularity 0 pairs in arbitrary dimension. 
		Recall, a boundary polarized CY pair $(X,B+D)$ has regularity 0  if
		\[
		\dim ({\rm Src}(X,B+D)) = \dim X-1
		.\]
		Thus a boundary polarized CY surface pair has regularity 0 if and only if it is Type II.

		\begin{prop}\label{p:reg0}
			If $(X,B+D)$ is a regularity 0 boundary polarized CY pair, then there exists a weakly special degeneration 
			\[
			(X,B+D)\rightsquigarrow
			(X_0,B_0+D_0)
			\]
			such that $(X_0,B_0)$ is not klt and there exists a $\bG_m$-action on $(X_0,B_0+D_0)$ that is non-trivial on each irreducible component of $X_0$.
		\end{prop}

		The result immediately implies that if $(X,B+D)$ is a  regularity 0 boundary polarized CY pair that is polystable, then $(X,B)$ is not klt and there exists a $\bG_m$-action on $(X,B+D)$ that is non-trivial on each irreducible component of $X$.
		
		\begin{proof}
			First, we assume that $X$ is non-normal.
			Let $(\oX,\oG+\oB+\oD) := \sqcup_{i=1}^r (\oX_i, \oG_{i}+\oB_i+\oD_i)$ denote the normalization of the pair. 
			By Proposition \ref{p:tcdivisorsonX} applied to the divisorial valuations induced by the irreducible components of $\oG_i$, there exists a weakly special degeneration 
			\[
			(\oX,\oG+\oB+\oD)\rightsquigarrow(\oX_0,\oG_0+\oB_0+\oD_0)
			\]
			such that induced test configuration of each component of $\oG^n$ is trivial and the $\bG_m$-action on $\oX_0$ is non-trivial on each component.  
			Thus Lemma \ref{l:tcgluing} produces a weakly special degeneration 
			\[
			(X,B+D) \rightsquigarrow(X_0,B_0+D_0)
			\]
			such that the induced $\bG_m$-action on each component of $X_0$ is non-trivial. Since $X$ is non-normal, $X_0$ is non-normal and so $(X_0,B_0)$ is not klt. 
			
			Next, we assume that $X$ is normal and $(X,B)$ is not klt.
			By \cite[Theorem 4.8]{BABWILD} (see also \cite{CZ22b}), there exists a non-trivial test configuration that induces a  weakly degeneration of 
			\[
			(X,B+D) \rightsquigarrow (X_0,B_0+D_0)
			\]
			such that $X_0$ is irreducible. Since the test configuration is non-trivial, $(X_0,B_0+D_0)$ admits a non-trivial $\bG_m$-action. 
			Since $(X,B)$ is not klt, $(X_0,B_0)$ is not klt and so we are done in this case.
			
			It remains to consider the case when $(X,B)$ is klt. By the previous cases and Lemma \ref{l:connecting}, it suffices to a construct a weakly special degeneration to a pair $(X_0,B_0+D_0)$ such that $(X_0,B_0)$ is not klt.
			To proceed, choose a divisor $E$ over $X$ with $A_{X,B+D}(E)=0$. 
			Let $\cE$ be the divisor over $X_{\bA^1}:=X\times \bA^1$ such that $\ord_{\cE}$ is the quasi-monomial combination of $\ord_E$ and $\ord_{X\times 0}$ with weights $(1,1)$.
			Since $\cE$ is an lc center of 
			$(X_{\bA^1},B_{\bA^1}+D_{\bA^1}+X\times 0)$, \cite[Corollary 1.38]{Kol13} produces a proper birational morphism of normal varieties 
			$\cX\to X_{\bA^1}$ such that $\cX$ is $\bQ$-factorial and $\cE\subset \cX$ is the sole exceptional divisor. 
			By choosing the log resolution in the proof of \emph{loc. cit.} to be $\bG_m$-equivariant, 
			we may assume that $\cX \to X_{\bA^1}$ is $\bG_m$-equivariant, since the steps of the MMP in the proof will  be $\bG_m$-equivariant as $\bG_m$ is a connected group.
			
			Let $\cB$ and $\cD$ denote the strict transforms of $B_{\bA^1}$ and $D_{\bA^1}$ on $\cX$.
			Since 
			\[
			(\cX, \cB+\cD+\cX_0)\to (X_{\bA^1},B_{\bA^1}+D_{\bA^1}+X\times \{0\})
			\]
			is crepant, $(\cX, \cB+\cD+\cX_0)$ is lc and CY over $\bA^1$. 
			In addition, for $0<c\ll1$, the crepant pullback of  $(X_{\bA^1},B_{\bA^1} + (1-c) (D_{\bA^1} +X\times 0))$ to $\cX$ is  a klt pair with big and nef anti-log canonical divisor over $\bA^1$. 
			Thus $\cX$ is Fano type over $\bA^1$. 
			By \cite[Corollary 1.3.1]{BCHM10},
			we can run a $-(K_{\cX}+\cB)$ MMP over $\bA^1$ and obtain a composition of divisorial contractions and flips
			\[
			\cX\dashrightarrow \cX^1 \dashrightarrow \cdots \dashrightarrow \cX^r:=\cX^{\rm m}
			\]
			such that $-K_{\cX^{\rm m}} - \cB^{\rm m}$ is big and semiample over $\bA^1$, where $\cB^{\rm m}$ is the birational transform of $\cB$.
			Since
			\[
			-K_{\cX}-\cB \sim_{\bQ, \bA^1} \cD+\cX_0 \sim \cD
			\]
			and $-K_{\cX}-\cB$ is nef over $\A^1\setminus 0$,
			each map $\cX^i \dashrightarrow \cX^{i+1}$ only contracts curves in $\Supp(\cD_i)\cap \cX^i_0$. 
			Thus the steps of the MMP are flips and so $\cX_0^{\rm m}$ is a union of two prime divisors.
			
			Now $-K_{\cX^{\rm m}} - \cB^{\rm m}$ is big and semiample over $\bA^1$. Let $\cX^{\rm m}\to \cX'$ denote its ample model.
			Thus $-K_{\cX'}- \cB'$ is ample over $\bA^1$. 
			Additionally, since $(\cX,\cB+\cD+\cX_0)$ is an lc CY pair over $\bA^1$ and $\cX\dashrightarrow \cX'$ is a birational contraction,  $(\cX',\cB'+\cD'+\cX'_0)$
			is  an  lc CY pair over $\bA^1$. 
			Therefore, using Proposition \ref{p:familyslcpairsovercurve}, we deduce that $(\cX',\cB'+\cD')\to \bA^1$ is  $\bG_m$-equivariant family of boundary polarized CY pairs and, hence, a test configuration of $(X,B+D)$. 
			Finally, note that $(\cX^{\rm m},\cB^{\rm m}+\cX^{\rm m}_0)$ is not plt, since $\cX^{\rm m}_0$ is the union of two intersecting prime divisors. 
			Since discrepancies do not change when taking the anti-log canonical model using \cite[Lemma 3.38]{KM97}, $(\cX',\cB'+\cX'_0)$ is not plt. Therefore $(\cX'_0,\cB'_0)$ is not klt and the proof is finished.
		\end{proof}

		\subsection{Seifert Bundles}
		We briefly recall the definition of a Seifert $\bG_m$-bundle and its compactication.
		The concept will be used to understand the geometry of polystable Type II pairs. 
		
		\begin{defn}[\cite{Kol04}]
			Let $T$ be a normal variety.
			A \emph{Seifert $\bG_m$-bundle} over $T$ is  the data of a normal variety $Y$ with a $\bG_m$-action and a morphism $Y\to T$ such that: 
			\begin{enumerate}
				\item $f:Y\to T$ is affine and $\bG_m$-equivariant with respect to the trivial $\bG_m$-action on $T$, 
				\item For every $t\in T$, the $\bG_m$-action on the reduced fiber $Y_t^{\rm red}:= {\rm red}( f^{-1}(t))$ is $\bG_m$-equivariantly isomorphic to the natural left $\bG_m$-action on $\bG_m/ \bmu_{m(t)}$ for some $m (t) \in \bZ_{>0}$. 
				\item The integer $m(t) = 1$ for  $t$ in a dense open subset of $T$.
			\end{enumerate}
		\end{defn}

		If $L$ is a $\bQ$-Cartier $\bQ$-divisor on a normal variety $T$, then 
		\[
		Y_L := {\bf Spec}_T  \bigoplus_{i \in \bZ} \cO_{T}\left(\lfloor i L \rfloor \right)
		,\]
		is a Seifert $\bG_m$-bundle, where the $\bG_m$-action on $Y_{L}$ is induced by the $\bZ$-grading and $\cO_T\left(\lfloor i L \rfloor \right)$ is the weight $i$-eigenspace.
		Furthermore, if $Y\to T$ is a Seifert $\bG_m$-bundle, then there exists a $\bQ$-Cartier $\bQ$-divisor $L$ (that is unique up to $\bZ$-linear equivalence) such that $Y$ is isomorphic to $Y_L$ as Seifert $\bG_m$-bundles; see \cite[Theorem 7]{Kol04}.
		
		If $Y\to T$ is a Seifert $\bG_m$-bundle over a normal variety with $Y \simeq Y_L$, then it has a \emph{partial compactification} and \emph{compactification}  
		\[
		\overline{Y}_L^a :=
		{\bf Spec}_T \bigoplus_{i\in \bN} \cO_{T}\left(\lfloor iL \rfloor \right) 
		\quad \text{ and } \quad
		\overline{Y}_L :={\bf Proj}_T \bigoplus_{m\in \bN} \bigoplus_{i=0}^m \cO_{T}\left(\lfloor iL \rfloor \right) 
		,
		\]
		where the Proj is taken with respect to the $m$-grading and both schemes admit a $\bG_m$-action induced by the $i$-grading. 
		There are natural open embedding $Y_L \hookrightarrow \overline{Y}^a_L\hookrightarrow \overline{Y}_L$.
		The complements 
		$T_\infty :=\overline{Y}_L \setminus \overline{Y}^a_L$
		and
		$T_0 := \overline{Y}^a_L\setminus Y_L$
		are two disjoint divisors  both isomorphic to $T$.
		See \cite[Section 2.6]{BABWILD}  or \cite[14]{Kol04} for further details. 
		
		\begin{prop}\label{p:toricsingularity}
			Let $T$ be a smooth curve and $L$ a $\bQ$-divisor on $T$, 
			Let $X:= \overline{Y}_L^a$ with projection $g: X\to T$ and $C\subset X$ denote the zero section.
			If $p\in C(\bk)$, then 
			$(X, C+ g^{-1}(g(p))) $ is \'etale locally isomorphic to a toric surface pair at $p$. 
		\end{prop}
		
		\begin{proof}
			This follows from the computation in \cite[Proposition 3.8]{FZ03}, which shows that, $(X, C+ \pi^{-1}(\pi(p))) $ is Zariski locally isomorphic to a toric pair when $T \simeq \bA^1$.
		\end{proof}

		\subsection{Geometry of Type II surface pairs}
		We now analyze the geometry of Type II polystable surface pairs. 
		In particular, we will show that their normalizations are crepant birational to  Seifert $\bG_m$-bundles in an  explicit way.

		\begin{prop}\label{p:TypeIISeifert}
			Let $(X,B+D)$ be a  Type II  boundary polarized CY surface pair with an effective $\bG_m$-action such that $X$ is normal and $(X,B)$ is not klt.
			
			Then there exist two lc CY pairs $(Y,D_Y)$ and $(Z,D_Z)$ with  $\bG_m$-actions and  crepant birational morphisms
			\[
			\begin{tikzcd}
				&(Y,D_Y)\arrow[ld,"f", swap] \arrow[rd,"g"]&\\
				(X,B+D)& & (Z,D_Z)
			\end{tikzcd}
			\]	
			that are  $\bG_m$-equivariant and satisfy:
			\begin{enumerate}
				\item[(1)] The divisor $\lfloor D_Y\rfloor=E_1+E_2$, where $E_1$ and $E_2$ are prime divisors with 
				$
				\Exc(f) \subset  E_1 \cup E_2
				$.
				\item[(2)] The variety $Z$ is a compactified Seifert $\bG_m$-bundle with zero and infinity section  $F_1 := g_*E_1$ and $F_2 :=g_* E_2$.
			\end{enumerate}
		\end{prop}
		
		\begin{proof}
			The $\bG_m$-action on $(X,B+D)$ and its inverse
			induce two distinct product test configurations of $(X,B+D)$. 
			By \cite[Theorem 4.8]{BABWILD}, the product test configurations induce two distinct prime divisors $E_1$ and $E_2$ over $(X,B+D)$ with log discrepancy zero.
			Since $(X,B+D)$ is Type II, \cite[Lemma 8.9]{BABWILD} implies that there are no other prime divisors over $(X,B+D)$ with log discrepancy 0.

			By \cite[Corollary 1.38]{Kol13}, there exists a crepant birational morphism  
			\[
			f:(Y,D_Y)\to (X,B+D)
			\]
			such that (1) holds and $Y$ is $\bQ$-factorial. 
			Additionally, by modifying the proof in \emph{loc. cit.} to choose the log resolution to be $\bG_m$-equivariant, we may construct $f:Y\to X$ so that $\bG_m$ acts on $(Y,D_Y)$ and $f$ is $\bG_m$-equivariant. 
			
			To construct $Z$, fix $0<\varepsilon< 1$ and set 
			$
			D'_Y:=D_Y -\varepsilon E_1-\varepsilon E_2
			$.
			Since $K_{Y}+D'_Y$ is not pseudoeffective, 
			the $(K_Y+D'_Y)$-MMP process terminates with morphisms
			\[
			Y \overset{g}{\to} Z \overset{\pi}{\to} T
			,\]
			where $g$ is a composition of divisorial contractions and $\pi$ is a Mori fiber space. 
			Since $\bG_m$ is a connected algebraic group, every step of the above MMP is $\bG_m$-equivariant. 
			In particular, $(Z,g_*D'_Y)$ admits a $\bG_m$-action and the morphisms $g$ and $\pi$ are $\bG_m$-equivariant.
			Therefore $(Z, D_Z := g_* D_Y)$ admits a $\bG_m$-action as well.
			Since $K_Y+D_Y \sim_{\bQ}0$, $K_{Z}+D_{Z}\sim_{\bQ}0$ and  $(Y,D_Y)\to (Z,D_Z)$ is crepant. 
			Now, note that the $(K_Y+D'_Y)$-MMP is the same as the $(K_Y+D_Y- E_1-E_2)$-MMP, since 
			\[
			K_Y+D_Y-E_1 - E_2 \sim_{\bQ} \varepsilon (K_Y+D'_Y)
			.
			\] 
			Thus \cite[Proof of Proposition 4.37]{Kol13} shows that $T$ is a curve and $(Z,D_Z)$ is a standard $\bP^1$-link, which means that 
			$(Z,D_Z)$ is a $\bQ$-factorial plt CY pair,
			$\pi|_{F_i}: F_i \to T$ are both isomorphisms, 
			and the reduced fibers of $Z\to T$ are isomorphic to $\bP^1$.

			We will now show that the $\bG_m$-action fixes $T$. 
			Since $\bG_m$ acts on $(Z,D_Z)$ and the $F_i$ are $\bG_m$-invariant, 
			there is an induced
			$\bG_m$-action on 
			\[
			(F_i, D_{F_i}:= {\rm Diff}_{F_i}(D_Z-F_i))
			.\] 
			Since $(F_i, D_{F_i})$ is a klt CY pair by adjunction, 
			either $F_i$ is an elliptic curve or $F_i \simeq \bP^1$ and $\Supp(\Gamma_{F_i})$ contains at least three points. 
			In either case, the $\bG_m$-action on $F_i$ must be trivial. 
			Using that $\pi\vert_{F_i}: F_i \to T$ is surjective and $\bG_m$-equivariant, the $\bG_m$-action on $T$ must also be trivial. 
			Since $(Z,D_Z) \to T$ is a standard $\bP^1$-link with an effective $\bG_m$-action fixing $T$, 
			\cite[Proposition 2.25]{BABWILD} implies that (2) holds. 
		\end{proof}

		\begin{prop}\label{p:TypeIIblowup}
			Keep the same setup and notation as in Proposition \ref{p:TypeIISeifert}. 
			There exists a sequence of  birational morphisms
			\[
			Z_r \overset{h_{r-1}}{\to} Z_{r-1} \to \cdots \to Z_1 \to Z_0 =Z
			\]
			satisfying the following conditions:
			\begin{enumerate}
				\item The morphism $Z_1\to Z_0$ is the minimal resolution of the surface $Z_0=Z$.
				\item 
				For $j\geq 1$,  $ Z_{j+1} \to Z_{j} $ is the blowup of a point at the intersection of two distinct curves in $\Supp(D_{Z_j})$, 
				where $D_{Z_j}$ is the $\bQ$-divisor defined inductively by 
				\[
				K_{Z_{j}} + D_{Z_j} = h_j^*( K_{Z_{j-1}} +D_{Z_{j-1}})
				.\]
				\item The induced birational map $Z_r \dasharrow  Y$ is a morphism.
			\end{enumerate}
		\end{prop}

		\begin{proof}
			Let $h_0:Z_1\to Z_0=Z$ be the minimal resolution of $Z$.
			We claim that $h_0$ is a log resolution of $(Z,D_Z)$.
			By the proof of Proposition \ref{p:TypeIISeifert},  $\pi:Z \to T $ is a compacticatified Seifert $\bG_m$-bundle with zero and infinity section $F_1$ and $F_2$. 
			Thus there exists a $\bQ$-divisor $L$ on $T$ and a $\bG_m$-equivariant isomorphism $Z \simeq  \overline{Y}_{L}$.
			Since $D_{Z}$ is $\bG_m$-invariant, $\Supp(D_Z)$ is 
			the union of $F_1\cup F_2$ and, a possibly empty, sum of fibers.  
			Now, fix a point $z\in Z$ in the non log smooth locus of $(Z,\Supp(D_Z))$ and set $t= \pi(z)$.
			Since the non log smooth locus is discrete and $\bG_m$-invariant, $z\in F_i$ for some $i$.
			By Proposition \ref{p:toricsingularity},
			$(Z,F_1+F_2+\pi(\pi^{-1}(z))$ is \'etale locally isomorphic to a toric surface pair at $z$.
			Therefore $Z_1 \to Z$ is a log resolution of $(Z,F_1+F_2+\pi(\pi^{-1}(z))$  at $z$ and so is also a log resolution of $(Z,D_Z)$ at $z$.
			Hence the claim holds.
			
			Next, we construct the remaining morphisms.
			By \cite[Lemma 2.22]{Kol13} applied to the curves contracted by $Y\to Z$, there exists finitely many point blowups 
			\[
			Z_r \to Z_{r-1} \to  \cdots  \to Z_2 \to Z_1
			\]
			such that 
			each $Z_{j+1} \to Z_{j}$ is the blowup at the center of $\ord_C$ on $Z_{j}$
			for some curve $C\subset Y$ contracted by $Y\to Z$
			and the curves contracted by $Y\to Z$ are not contracted by $Y \dashrightarrow Z_r$. 
			Thus  $Z_r\dashrightarrow Y$ is defined at codimension 1 points and, hence, is a morphism.

			It remains to show that this blowup process satisfies condition (2) for each fixed $j\geq 1$. 
			Write $C\subset Y$ for the curve contracted by $Y\to Z$ mentioned above 
			and   $p \in Z_j$ for its center.
			Observe that
			\[
			A_{ Z_{j},D_{Z_j}} (\ord_C)= A_{Z,D_Z} (\ord_C)= A_{Y,D_Y}(\ord_C) \leq 1
			,\]
			where the equalities use the pairs are crepant and the inequality uses that $C\subset Y$ and $D_Y$ is effective.
			By the above inequality and the fact that 
			$D_{Z_j}$ has snc support and coefficients $\leq 1$,
			\cite[Proposition 2.7]{Kol13} implies that 
			\begin{itemize}
				\item[(i)] $p$ is a point at the intersection of two  curves in
				$\Supp(D_{Z_j})$ or 
				\item[(ii)] $p$ is contained in $\Supp( \lfloor D_{Z_j}\rfloor)$. 
			\end{itemize}
			It remains to show that (i) always holds. 
			To simplify notation, let 
			\[
			\rho:= \pi \circ g :Y\to T
			\]
			and $t:= \rho(C)$.
			We will proceed to show that if (i) does not hold, then the ampleness of $D$ will be violated, which will be a contradiction.

			\medskip 
			
			\noindent  \emph{Claim}: If (i) does not hold, then $\rho^{-1} ( t)$ is a chain of rational curves $G_1 \cup \cdots \cup G_s$ with $s\geq 2$ and each component satisfies $G_i\not\subset \Supp(D_Y)$. 
			\medskip 
			
			\noindent \emph{Proof of the Claim.}
			Write $F_{ij}$ for the birational transform of $F_i$ on $Z_{j}$. 
			Since $(Z_j,D_{Z_j})\to (Z,D_Z)$ is crepant, $\Supp(\lfloor D_{Z_j}\rfloor)= F_{1j} \cup F_{2j}$. Since $(Z_1,\Supp(D_{Z_1}))$ and the $h_i$ are point blowups, $(Z_j, \Supp(D_{Z_j})$ is also log smooth. 
			Now, if (i) does not hold, then 
			\[
			p \notin \Supp( {\rm Diff}_{F_{ij}}(D_{Z_j}-F_{ij}))
			\]
			for $i=1,2$.
			Since $(Z_j, D_{Z_j}) \to (Z,D_Z)$ is crepant birational, 
			the isomorphism $F_{ij}\to F_{i}$ induces an isomorphism 
			\[
			(F_{ij}, {\rm Diff}_{F_{ij}}(D_{Z_j}-F_{ij}) )
			\overset{\sim}{\longrightarrow}
			(F_i,  {\rm Diff}_{F_i}(D_{Z}-F_i) )
			.\]
			Thus, using the formula for the different \cite[Corollary 3.45]{Kol13}, 
			the image of $p$ on $Z$ is a smooth point not contained in $\Supp(D_{Z}- F_{1}-F_{2})$. 
			Therefore each curve $G \subset \rho^{-1}(t)$ satisfies 
			\[
			1\leq A_{Z,D_Z}(G) = A_{Y,D_{Y}}(G) \leq 1
			,\]
			where the first inequality holds by \cite[Proposition 2.7]{Kol13} and the previous sentence. 
			Thus $A_{Y,D_Y}(G)=1$ and so  $G \not \subset \Supp(D_Y)$.
			
			It remains to show that $\rho^{-1}(t)$ is a chain of at least two rational curves. 
			First, since $Y\to Z$ contracts $C$, $\rho^{-1}(t)$ must contain at least two irreducible components. 
			Using that the fiber of $Z_r \to T$ over $t$ is a chain of rational curves and that $Z_r \to Y$ only contracts curves in the fibers, we conclude that $\rho^{-1}(t)$ is a chain of rational curves.
			\qed
			\medskip 
			
			Now assume that (i) does not hold and write $\rho^{-1}(t)=G_1 \cup \cdots \cup G_s$  with $s\geq 2$ and $G_i \not\subset \Supp(D_Y)$ for the chain of rational curves that exists by the above claim. 
			Since $(X,B)$ is not klt,
			after possibly switching $E_1$ and $E_2$,
			we may assume that $E_1$ is an lc center of 
			$(X,B)$.
			Note that  $E_1 \not\subset \Supp(f^*D)$, since $(X,B+D)$ is also lc. 
			After reindexing the chain of curves, we may assume 
			$
			G_1\cap E_2 = \emptyset
			$.
			Now observe that
			\[
			0<D\cdot f(G_1)  = f^*D \cdot G_1 =0
			.\]
			Indeed, the first inequality holds, since  $D$ is ample and $f$ does not contract $G_1$ by Proposition \ref{p:TypeIISeifert}.1.
			The  equality is by the projection formula.
			The final equality holds, since
			\[
			\Supp(f^*D) \cap G_1  \subset \Supp(D_Y-E_1) \cap G_1= E_2\cap G_1=\emptyset
			.\]
			Here we are using the fact that $\Supp(D_Y-E_1-E_2)$ is contained in a finite union of fibers of $\rho$ that are disjoint from $\rho^{-1}(t)$ by the claim.
			Therefore we have reached a contradiction and so (i) must always hold as desired.
		\end{proof}

		\subsection{Fields of definition}
		
		We now prove a result on the field of definition for polystable Type II boundary polarized  CY surface pairs. 
		
		\begin{thm}\label{t:TypeIIfieldofdef}
			Let $(X,B+D)$ be a polystable Type II boundary polarized CY surface pair defined over $\bk$
			and  $(E,D_E):= {\rm Src}(X,B+D)$.
			
			If  there exists a pair $(E',D_{E'})$ over an algebraically closed subfield $\bk'\subset \bk$ such that 
			\[
			(E,D_E)\simeq (E',D_{E'}) \times_{\bk'} \bk,
			\]
			then there exists a pair $(X',B'+D')$ defined over $\bk'$ 
			such that 
			\[
			(X,B+D) \simeq (X',B'+D') \times_{\bk'}\bk.
			\] 
		\end{thm}

		
		\begin{proof}
			Since $(X,B+D)$ is polystable, Proposition \ref{p:reg0} implies that $(X,B)$ is not klt and that there exists a $\bG_m$-action on $(X,B+D)$ that is non-trivial on each component.
			We now break the proof up into two cases.

			First, we assume  that $X$ is normal.
			Since the $\bG_m$-action on $X$ is non-trivial, the subgroup $H:=\ker( \bG_m \to \Aut(X))$ is finite and so $\bG_m / H \simeq \bG_m$.
			Thus, by replacing $\bG_m$ with $\bG_m/H$, we may assume the $\bG_m$-action on $X$ is effective. 
			Thus the hypotheses of
			Propositions \ref{p:TypeIISeifert} and Proposition \ref{p:TypeIIblowup} are satisfied and we can use the results and notation from the propositions including the  crepant birational morphisms 
			\[
			(X,B+D) \overset{f}{\leftarrow} (Y,D_Y) \overset{g}{\rightarrow} (Z,D_Z)
			\]
			such that $Z$ is a compactified Seifert $\bG_m$-bundle with zero and infinity section $F_1$ and $F_2$. 
			Since the above pairs are crepant birational, the isomorphism $E_i \to F_i$ induces an isomorphism
			\[
			(E,D_E) = (E_{i}, D_{E_i}:= {\rm Diff}_{E_i}(D_Y-E_i)) \simeq (F_i, D_{F_i}:={\rm Diff}_{F_i}(E_Z-F_i))
			\]
			by \cite[Proposition 4.6]{Kol13}. 
			Thus $Z$ is a compactified Seifert $\bG_m$-bundle over $(E,D_E)$ with respect to some ample $\bQ$-divisor $L$.
			
			We claim that after possibly replacing $L$ with a linearly equivalent divisor or changing the isomorphism $(E,D_E)\simeq (E',D_{E'}) \times_{\bk'} \bk$, $L$ is supported at $\bk'$ points of $E$.
			Indeed, since $(E,D_E)$ is a klt CY pair, either (i) $E$ is a genus 1 curve and $D_E=0$ or (ii) $E\simeq \bP^1$ and $D_E\neq 0$. 
			By \cite[Proposition 2.20]{BABWILD},
			\[
			\Supp (L - \lfloor L \rfloor ) \subset \Supp( {\rm Diff}_{E_i}(0)) \subset \Supp (D_E)
			.\]
			In case (i), we get $L = \lfloor L \rfloor$.
			By properties of divisors on elliptic curves, after replacing $L$ with a linearly equivalent divisor, we may assume $L = d \cdot p$ for some $p \in E$ and $d:= \deg(L)$.
			After twisting the isomorphism $E\simeq E' \times_{\bk'} \bk$ with a translation of $E$, we may further assume that $p$ is a $\bk'$-point. 
			In case (ii), we get $L - \lfloor L \rfloor \subset \Supp(D_E)$. 
			Hence, $L- \lfloor L \rfloor $ is supported at $\bk'$-points. 
			Since $E\simeq \bP^1$ in this case, after replacing $L$ with a linearly equivalent divisor, we may additionally assume that $\lfloor L \rfloor$ is supported at $\bk'$.
			Therefore the claim holds. 
			
			By the claim, there exists an ample $\bQ$-divisor $L'$ on $E'$ such that 
			\[
			(E',L')\times_{\bk'} \bk \simeq (E,L).
			\]
			Let $Z'$ be the compactified Seifert $\bG_m$-bundle over $E'$ with respect to $L'$. 
			By construction,  
			\[
			Z' \times_{\bk'} \bk\simeq Z
			.\] 
			Since $\Supp (D_Z)$ is $\bG_m$-invariant, every curves $C\subset \Supp(D_Z)$ is either one of the two $\G_m$-invariant sections of $Z\to E$ or  a fiber $Z\to E$. 
			In case the latter case, $C\cap E_i \subset \Supp( D_{E_i})$ and hence $C$ is the fiber of $Z\to E$ over a $\bk'$-point.
			Thus there exists a $\bQ$-divisor $D_{Z'}$ on $Z'$ such that 
			\[
			(Z',D_{Z'}) \times_{\bk'} \bk\simeq (Z,D_Z)
			\]
			Next, let $Z'_1 \to Z'$ be the  minimal resolution of $Z'$. 
			Since the blowups of $Z_i \to Z_{i-1}$ occur at $\bk'$ points by Proposition \ref{p:TypeIIblowup}.2, we can choose a sequence of point blowups
			\[
			Z'_r  \to \cdots \to Z'_{1} 
			\]
			that base changes to 
			\[
			Z_r \to \cdots \to Z_1
			.\]
			
			We will now construct $X'$.
			Let $h$ denote the composition 
			\[
			\begin{tikzcd}
				Z_r\arrow[r] \arrow[rr,bend left,"h"]& Y\arrow[r,"f"] & X
			\end{tikzcd}
			\]
			and define $B_{Z_r}$ by 
			$
			K_{Z_r}+B_{Z_r}= h^*(K_X+B)
			$.
			Since
			\[
			\Supp(B_{Z_r})\subset \Supp(D_{Z_r}) \cup \Exc(h)
			,\]
			there exists a $\bQ$-divisor 
			$B_{Z'_r}$ such that 
			\[
			(Z_r, B_{Z_r}) \simeq (Z_r, B_{Z'_r}) \times_{\bk'} \bk .
			\]
			Since 
			$-K_{Z_r}- B_{Z_r}$ is semiample and $h:Z_r\to X$ is the morphism to the ample model,
			$-K_{Z'_r} - B_{Z'_r}$ is semiample
			and
			the morphism to its ample model 
			\[
			Z'_r \overset{h'}{\longrightarrow} X'
			\]
			satisfies $h\simeq h \times_{\bk'}\bk$.
			Finally, if we set 
			$
			B':= h'_* B_{Z'_r} $ and $ D':= h'_* D_{Z'_r} -B_{Z'_r}$, then
			\[
			(X,B+D) \simeq (X',B'+D') \times_{\bk'} \bk
			\] 
			as desired.

			It remains to prove the result when $(X,B+D)$ is not normal. 
			In this case, let
			\[
			(\overline{X},\overline{G}+ \overline{B}+\overline{D}) :=\sqcup_{i=1}^r (\overline{X}_i,\overline{G}_i+\overline{B}_i+ \overline{D}_i) 
			\]
			denote the pairs normalization.
			Since $(X,B+D)$ is Type II, $\overline{X}$ has at most two connected components and the conductor divisor $\overline{G}$ normal \cite[Proposition 8.10]{BABWILD}.
			Write 
			\[
			(\overline{G}, D_{\overline{G}}) : = \sqcup_{i=1}^r (\overline{G}_i,D_{\overline{G}_i}),
			\]
			where $D_{\overline{G}_i}:= {\rm Diff}_{G_i}(\overline{B}_i+\overline{D}_i)$,
			and  $\tau: (\overline{G},D_{\overline{G}})\to (\overline{G},D_{\overline{G}})$ for the induced involution.
			By adjunction, $(\overline{G}_i, D_{\overline{G}_i}) $ is a klt CY pair and
			$(\overline{G}_i, {\rm Diff}_{\overline{G}_i}(\overline{B}_i))$ is a klt log Fano pair.
			Therefore $\overline{G}_i\simeq \bP^1$ and $\Supp( \overline{D}_{\overline{G}_i}  )$ contains at least three points. 
			Note that the $\bG_m$-action on $(X,B+D)$ induces a $\bG_m$-action on $(\oX,\oG+\oB+\oD)$ that is non-trivial on each component. 
			
			Since $X$ is non-normal, each component of $(\overline{X},\overline{G}+\overline{B})$ is not klt.
			Thus the previous paragraphs of the proof  imply that there exists a boundary polarized CY pair $(\overline{X}',\overline{G}'+\overline{B}' +\overline{D}')$   defined over $\bk'$ such that 
			\[
			(\overline{X},\overline{G}+ \overline{B}+\overline{D}) \simeq (\overline{X}',\overline{G}' +\overline{B}'+\overline{D}')\times_{\bk'} \bk
			\]
			In particular,
			\[
			(\overline{G}, D_{\overline{G}}) \simeq (\overline{G}', D_{\overline{G}'} ) \times_{\bk'} \bk
			,\]
			where $D_{\overline{G}'} := {\rm Diff}_{\overline{G}'}(B'+D')$,
			and so the points of $\Supp(\overline{D}_{\overline{G}_i})$ are all $\bk'$ points. 
			Since the data of an automorphism of $\bP^1$ is equivalent to the data of where  three points get sent, 
			there exists an involution $\tau':  (\overline{G}', D_{\overline{G}'} )\to (\overline{G}', D_{\overline{G}'})$ 
			whose base change  to $\bk$ is $\tau$.
			The gluing relation $R({\tau}')\rightrightarrows \oX$ as in \cite[Definition 5.31]{Kol13} has finite equivalence classes, since its base change to $\bk$ has the same property (alternatively, this can be seen more explicitly as we are in dimension two). 
			Thus the geometric quotient of the pair exists by \cite[Corollary 5.33]{Kol13} and is a boundary polarized CY pair $(X',B'+D')$  and satisfies $(X,B+D) \simeq (X',B'+D') \times_{\bk'} \bk$.
		\end{proof}

		\section{Type III surface pairs}\label{s:TypeIII}
		
		In this section, we  prove that every Type III boundary polarized CY surface pair admits a weakly special degeneration to a pair whose normalization is toric (Theorem \ref{t:TypeIIItoricdegen}).
		This result will be used in Section \ref{s:asgm} verify that the Type III locus of the CY moduli space is discrete.
		
		To prove the result, we will first study the geometry of normal boundary polarized CY and their weakly special degenerations induced by  curves in their log Fano boundary.
		We will then prove the toric degeneration result  by degenerating the normalization of the pair and then carefully gluing the degenerations.
		
		\subsection{Toric pairs}
		A \emph{normal toric variety} $X$ is a normal variety $X$ with the data of a $\bT:= \bG_m^{r}$-action on $X$ and a $\bT$-equivariant open embedding $\bT \hookrightarrow X$.
		A pair $(X,B)$ is \emph{toric} if $X$ is a normal toric  variety and  $B$ is $\T$-invariant.
		
		If a boundary polarized CY pair $(X,B+D)$ is a toric pair, then $B+D$ equals the reduced toric boundary, which we denote by $\Delta_{\bT}$. 
		Indeed, since $(X,B+D)$ is lc and $B+D$ is $\bT$-invariant, $B+D \leq \Delta_{\bT}$. 
		Since $K_{X}+\Delta_{\bT} \sim 0$ always holds on a toric variety, we conclude that $B+D = \Delta_{\bT}$. 
		
		\begin{prop}\label{p:torusactionfinitekernel->toric}
			If a normal boundary polarized CY pair $(X,B+D)$ admits a $\bT:= \bG_m^{\dim X}$-action with finite kernel, then $(X,B+D)$ can be endowed with the structure of a toric pair.
		\end{prop}
		
		\begin{proof} 
			Let $H$ denote the kernel of $\bT\to \Aut(X)$. Since $H$ is a  normal subgroup of $\bT$ and finite by assumption, $G:= \bT/ H$ is a connected algebraic group of dimension $\dim X$ and the induced $G$-action on $X$ is effective. Since $\bT$ is diagonalizable, $G$ is diagonalizable by \cite[Theorem 12.9.c]{Milne}. Since $G$ is a connected and diagonalizable,  $G$ is an algebraic torus by \cite[12.5 and 12.8]{Milne}. 
			Thus by replacing  $\bT$ with $G$, we may assume that that $\bT$-action is effective.
			
			We  claim that there exists a decomposition of $X$ into finitely many locall closed sets on which the stabilizer groups are  constant. 
			To proceed,  fix a positive integer $r$ such that $-r(K_X+B)$ is a very ample Cartier divisor.
			Since $\cO_X(-r(K_{X}+B))$ admits a canonical $\bT$-linearization, there is an induced $\bT$-action on 
			$V:= H^0(X, -r(K_X+B))$
			and so  there is a direct sum decomposition
			$V= \oplus_{ \alpha \in M} V_{\alpha}$, where $M:= \Hom (\bT, \bG_m)$ and  $t\in \bT $ acts on $V_\alpha$ via multiplication by $\alpha(t)$.
			Thus there exists a basis $(s_{0}, \ldots, s_N)$ for $V$  and $\alpha_0,\ldots, \alpha_N \in V$ such that  $s_i \in V_{ \alpha_i}$. 
			The basis induces a $\bT$-equivariant embedding 
			$
			X\hookrightarrow \bP^N
			$
			where $\bT$ acts on $\bP^N$ by 
			\[
			t\cdot [x_0:\ldots : x_n] = [\alpha_0(t) x_0 : \cdots \alpha_N(t)  x_N ]
			.\]
			For a subset $I \subset \{0,\ldots, N\}$, consider the locally closed subset $Z_{I}:= \{ x_i = 0 \, \vert\, i \in I\} \cap \{ x_i \neq 0 \, \vert\, i \notin I\} \subset \bP^N$.
			Note that the stabilizer of $x\in Z_I$ depend only on $I$. 
			Now $X =  \cup_I X_I$, where  $X_I:=X\cap Z_I$ is  a locally closed subset of $X$ such that the stabilizer group of $x\in X_I$ depends only on $I$. 
			Therefore the claim holds. 
			
			By the claim, there exists an open set $U \subset X$ on which the stabilizer group is constant.
			Since the $\bT$-action on $X$ has finite kernel, the stabilizer group must be trivial for all $x\in U$.
			Fix an element $e \in U(\bk)$. 
			The $\bT$-equvariant morphism  $j:\bT \to  X$  defined by $t\mapsto t \cdot e$ is a locally closed embedding.
			Since $\dim X = \dim \bT = \dim   j(\bT)$,  $j$ is an open embedding. 
			With the data of this embedding,  $X$ is a normal toric variety and  so $(X,B+D)$ is a toric pair. 
		\end{proof}


		\subsection{Normal surface pairs}\label{ss:normalsurfacepairs}
		
		\subsubsection{Classification}
		
		We now study the structure of curves appearing with coefficient one in the boundary of a  boundary polarized CY surface pair.

		\begin{proposition}\label{p:normallogFanoboundary}
			If $(X,B+D)$ is an lc boundary polarized CY surface pair, then either
			
			\begin{enumerate}
				\item  $\lfloor B \rfloor =0$,
				\item $\lfloor B \rfloor = C$, where $C\simeq \bP^1$, or
				\item $\lfloor B \rfloor = C_1 \cup C_2$, where each $C_i \simeq \bP^1$ and  the curves meet at a single node. 
			\end{enumerate}
		\end{proposition}
		
		\begin{proof}
			If (1) does not occur, then $C:=\lfloor B \rfloor$ is a curve with at worst nodal singularities by \cite[Theorem 2.31]{Kol13}.
			Since $X$ is a normal surface, the non-klt locus  of $(X,B)$ is the union of $C$ and a finite collection of points.
			Using that  $-K_X-B$ is ample,  \cite[Theorem 17.4]{Kol92} implies that the non-klt locus of $(X,B)$ is connected and  so $C$ must be connected. 
			Finally, using that $(C, {\rm Diff}_{C}(B-C ))$ is an slc log Fano pair by adjunction, we conclude that either (2) or (3) is satisfied.
		\end{proof}

		\begin{proposition}\label{p:existsE}
			Let $(X,B+D)$ be an lc boundary polarized CY surface pair with a non-trivial $\bG_m$-action. 
			If $C\subset \lfloor  B \rfloor $ is an irreducible curve 
			whose points are $\bG_m$-fixed and 
			\[
			(C,  {\rm Diff}_{C}( B+D-C )  )
			\]
			is not klt at a point $p \in C$, then there exists a curve $E\subset \lfloor B+D-C\rfloor$ such that $p \in E \cap C$ and $E$ is not contained in the $\bG_m$-fixed locus of $X$. 
		\end{proposition}
		
		\begin{proof}
			After replacing the $\bG_m$-action by the action by $\bG_m/H$, where $H:= \ker(\bG_m \to \Aut(X))$, we may assume that the $\bG_m$-action is effective.
			After possibly replacing the $\bG_m$-action with its inverse, we may assume that $C$ is in the attracting locus.
			Since $X$ is a normal variety with a $\bG_m$-action, a classical theorem of Sumihiro implies that $X$ is covered by $\bG_m$-invariant affine open sets. 
			Thus there exists a $\bG_m$-invariant affine open set $p\in U \subset X$.
			Since $U$ is normal and admits a non-trivial $\bG_m$-action fixing the points of $C\cap U $, \cite[Theorem 3.2]{FZ03} implies that there exists a $\bQ$-divisor $L$ on $ C\cap U$ and $\bG_m$-equivariant isomorphism
			$
			U \simeq \overline{Y}^a_L
			$.
			Since  the natural map $\overline{Y}^a_L\to C \cap U$ is $\bG_m$-equivariant and has reduced fibers isomorphic to $\bA^1$, there is a unique $\bG_m$-invariant curve $E\subset X$ that is not equal to $C$ and contains $p$.
			By Proposition \ref{p:toricsingularity}, 
			$(X,C+E)$ is lc at $p$.
			Now, let $a:= {\rm coeff}_E(B+D)$.
			Since   $\Supp(B+D)$ is $\bG_m$-invariant, $B+D$ and $C+aE$ agree in a neighborhood of $p$. 
			Thus $(C,{\rm Diff}_C(aE))$ is not klt at $p$ by our assumption. 
			By inversion of adjunction,  $(X,C+aE)$ is not plt at $p$
			Therefore $a=1$ and so $E\subset \lfloor B+D\rfloor$.
		\end{proof}
		
		\subsubsection{Degenerations}
		
		We now analyze certain  weakly special degenerations of normal boundary polarized CY surface pairs.
		Since these results will  be used in Section \ref{ss:nonnormalpairs} to construct degenerations of non-normal boundary polarized CY pairs via gluing, it will be important to analyze the induced degenerations of  the curves with coefficient one in the boundary. 
		
		First, we analyze a degeneration induced by two curves in the log Fano boundary of a pair.
		
		\begin{proposition}\label{p:degentwocurves}
			If $(X,B+D)$ is an lc boundary polarized CY surface pair such that  $ \lfloor B \rfloor$ is the union of two rational curves $C_1 \cup C_2$, then there exists a weakly special degeneration
			\[
			(X,B+D) \rightsquigarrow (X_0,B_0+D_0)
			\]
			with the following properties:
			\begin{enumerate}
				\item The scheme $X_0$ is a union of two normal surfaces glued along a smooth rational curve.
				\item The  $\bG_m$-action on $X_0$ is non-trivial on each component.
				\item The induced test configurations of $(C_i,B_{C_i}+D_{C_i})$ is trivial  for $i=1,2$ . 
			\end{enumerate}
		\end{proposition}
		\begin{proof}
			Set $v_i = \ord_{C_i}$ for $i=1,2$
			and note that $A_{X,B}(v_i) = 0$. 
			By Proposition \ref{p:tcdivisorsonX} applied to $v_1$ and $v_2$, there exists a weakly special degeneration
			\[
			(X,B+D) \rightsquigarrow (X_0,B_0+D_0)
			\]
			such that $X_0$ is a union of two irreducible surfaces and that (2) and (3) both hold.  
			Write 
			\[
			(\oX_0,\oG_0 + \oB_0+ \oD_0) := \sqcup_{i=1}^2  (\oX_0^i,\oG_0^i + \oB_0^i+ \oD_0^i) 
			\]
			for the normalization of $(X_0,B_0+D_0)$ and its decomposition into irreducible pairs.
			Since  $X_0$ is connected and has two irreducible components, each $\oG_0^i$ must contain at least one  curve.
			Also note that the degeneration of $C_1\cup C_2$ in $X_0$ is the image of a union of two curves in $\lfloor \oB_0 \rfloor$.
			Since  $\lfloor \oG_0^i+\oB^i_0\rfloor $  is a union of one or two smooth rational curves by
			Proposition \ref{p:normallogFanoboundary}, 
			we conclude that each $\oG_0^i$ is a smooth rational curve and so (1) holds.
		\end{proof}
		
		Next, we analyze a degeneration induced by a curve intersecting the log Fano boundary.

		\begin{proposition}\label{p:degenbyE}
			Let $(X,B+D)$ be an lc boundary polarized CY surface pair with a non-trivial $\bG_m$-action, 
			$C\subset \lfloor B\rfloor$  an irreducible curve 
			whose points are $\bG_m$-fixed, and $E\subset \lfloor B+D-C \rfloor $  an irreducible curve that intersects $C$. 
			Then $E$ induces a weakly special degeneration
			\[
			(X,B+D) \rightsquigarrow (X_0,B_0+D_0)
			\]
			with the following properties:
			\begin{enumerate}
				\item The normalization of $(X_0,B_0+D_0)$ is an irreducible toric pair. 
				\item The induced degeneration of $E$ is trivial.
				\item The induced test configuration of $(C,B_C+D_C)$  is induced by the valuation $m\ord_p$, where $p:= C \cap E$  and $m$ is a positive integer. 
			\end{enumerate}
			Furthermore, if $E\subset \lfloor B \rfloor$, then $X_0$ is normal. 
		\end{proposition}
		
		\begin{proof}
			Fix an integer $r>0$ such that $L:=-r(K_X+B)$ is a Cartier divisor. 
			Let $F^\bullet$ denote the filtration of $R:= R(X,L)$ 
			defined by 
			\[
			F^\la R_m := \{ s\in R_m \, \vert\, \ord_E(s) \geq \la + mr A_{X,B}(\ord_E)\}
			.\]
			By
			Propositions \ref{p:tcdivisorsonX}, the filtration is finitely generated and induces a test configuration $(\cX,\cB+\cD)$ of $(X,B+D)$ such that  $\cX_0$ irreducible and (2) holds.
			Note that $E$ is not in the $\bG_m$-fixed locus of $X$ by the  proof of Proposition \ref{p:existsE}.
			Thus Propositions \ref{p:tcdivisorsonXextra} and \ref{p:torusactionfinitekernel->toric}
			imply that $(X_0,B_0+D_0):= (\cX_0,\cB_0+\cD_0)$ satisfies (1).
			
			Next, write $(\oX_0,\oG_0+\oB_0+\oD_0)$ for the normalization of $(X_0,B_0+D_0)$.
			Write $C_0 \subset \lfloor B_0 \rfloor $ for the degeneration of $C$ in $X_0$ and $\oC_0\subset \oX_0$ for its preimage under $\oX_0\to X_0$.
			We claim that $C_0$ is irreducible. 
			To see this, note that $\bG_m$-action on $X$ induces a fiberwise $\bG_m$-action on $\cX$ that fixes $C\times (\bA^1\setminus 0)$.
			Hence this $\bG_m$-action on $(X_0,B_0+D_0)$ fixes $C_0$ and so the induced $\bG_m$-action on $\oX_0$ fixes $\oC_0$. 
			Since $\oC_0 \subset \lfloor \oB_0+\oG_0 \rfloor$,  $\oC_0$ is connected by Proposition \ref{p:normallogFanoboundary}.
			Using that the $\bG_m$-fixed locus of a normal surface with a $\bG_m$-action is necessarily normal (see e.g. \cite[1.1 and 1.2]{FZ03}), $\oC_0$ must be normal. 
			Therefore $\oC_0$ is irreducible and so $C_0$ is irreducible. 
			
			Next, write $(\cC,\cB_{\cC}+\cD_{\cC})$ for the induced test configuration of $(C,B_C+D_C)$
			and $F_C^\bullet$ for the induced filtration of $R_C:= R(C,L\vert_C)$.
			Set $c:= A_{X,B}(E)$
			and note that
			\[
			F^{-mrc}R_m = R_m\quad \text{ and } \quad {\rm Bs} (F^{-mrc+1}R_m) = E
			\]
			for $m>0$ sufficiently divisible by the formula for $F^\bullet$.
			By Proposition \ref{p:tcrestriction}, we see
			\[
			F^{-mrc} R_{C,m}  = R_{C,m} \quad \text{and } \quad {\rm Bs}(F^{-mrc+1}R_{C_m} )= \{p\}
			\] 
			for all $m>0$ sufficiently divisible. 
			Using that $\cC_0$ is irreducible, we conclude using  Lemma \ref{l:tcfiltformula} 
			that $v_{\cC_0}= b  \ord_{p}$ for some integer $b>0$.
			Therefore (3) holds.
			
			Finally, assume that $E\subset \lfloor B\rfloor $. 
			Since $\lfloor B \rfloor $ contains at least two  irreduicble curves, $\lfloor B_0\rfloor$ contains  at least two irreducible curves and so does $\lfloor \oB_0 \rfloor$. 
			Using that  $ \oB_0+\oG_0$ is the log Fano boundary of $(\oX_0,\oB_0+\oG_0+\oD_0)$,  $\oG_0= 0$ by  Proposition \ref{p:normallogFanoboundary}.
			Therefore $X_0$ is normal.
		\end{proof}
		
		Combining the previous two propositions, we produce a useful two step degeneration.
		
		\begin{proposition}\label{p:twostep}
			If $(X,B+D)$ is an lc boundary polarized CY surface pair such that $\lfloor B \rfloor$ is a union of two rational curves intersecting at a point, then there exists two weakly special degenerations 
			\[
			(X,B+D) \rightsquigarrow (X_0,B_0+D_0) \rightsquigarrow (X'_0,B'_0+D'_0)
			\]
			and a positive integer $m$
			satisfying the following conditions: 
			\begin{enumerate}
				\item The first degeneration induces the trivial test configuration on the components $\lfloor B\rfloor$. 
				\item The second degeneration induces a non-trivial test configuration on $(C_0,B_{C_0}+D_{C_0})$ for each curve  $C_0\subset \lfloor B_0 \rfloor $ given by  $m \ord_{p_0}$, where $p_0 $ is the nodal point of $\lfloor B_0 \rfloor$.
				\item Each irreducible irreducible component of $(X'_0,B'_0+D'_0)$ are toric pairs.
			\end{enumerate}
		\end{proposition}
		
		\begin{proof}
			Write $\lfloor B\rfloor = C^1\cup C^2$. 
			Proposition \ref{p:degentwocurves} produces a weakly special degeneration 
			\[
			(X,B+D) \rightsquigarrow(X_0,B_0+D_0)
			\]
			such that the induced test configuration of each $C^i$ is trivial,  $X_0$ is a union of two normal surfaces glued along a smooth rational curve, and the $\bG_m$-action on each component is non-trivial.
			Write 
			\[
			(\oX_0,\oG_0 + \oB_0+ \oD_0) := \sqcup_{i=1}^2  (\oX_0^i,\oG_0^i + \oB_0^i+ \oD_0^i) 
			\]
			for the normalization of $(X_0,B_0+D_0)$ and $\oC^i_{0} \subset \lfloor \oB_0^i\rfloor$ for the preimage of the degeneration of $C^i$ on the normalization. 
			Since the $\bG_m$-action on $\oC^i_{0}$ is trivial and $\oG_0^i$ intersects $\oC^i_{0}$,
			Proposition \ref{p:degenbyE} applied with $E=\oG_0^i$ produces a weakly special degeneration 
			\begin{equation}\label{e:oXi}
				(\oX_0^i,\oG_0^i + \oB_0^i+ \oD_0^i) 
				\rightsquigarrow
				(\oX'^i_0,\oG'^i_0+\oB'^i_0+\oD'^i_0)
			\end{equation}
			such that the degeneration is a toric pair,the induced test configuration  of $\oG^i_{0}$ is trivial, and the  induced test configuration of $\oC^i_0$ is induced by $m_i \ord_{p_o}$, where $p_i = \oC^i_0 \cap \oG^i_0$.
			After possibly scaling the two test configuration, we may assume that $m:=m_1=m_2$.
			Now Lemma \ref{l:tcgluing} (which is stated in the next section) produces a weakly special degeneration
			\[
			(X, B_0+D_0)\rightsquigarrow(X'_0,B'_0+D'_0)
			\]
			satisfying (2) and (3).
		\end{proof}
		
		We analyze a final way to degenerate a boundary polarized CY pair. This time, we use two curves intersecting the log Fano boundary to induce the degeneration.
		
		\begin{proposition}\label{p:degeninvolution}
			Let $(X,B+D)$ be a boundary polarized CY pair admitting a non-trivial $\bG_m$-action that fixes the points of a curve $C\subset\lfloor B\rfloor$.
			If
			\[
			\lfloor B_C+D_C\rfloor = p_1+p_2
			\]
			for some points $p_1,p_2\in E$ and there is an involution $\tau: (C,B_C+D_C)$ that interchanges $p_1$ and $p_2$, then there exists a weakly special degeneration 
			\[
			(X,B+D) \rightsquigarrow
			(X_0,B_0+D_0)
			\]
			satisfying the following: 
			\begin{enumerate}
				\item The normalization of $(X_0,B_0+D_0)$ is a disjoint union of two toric pairs.
				
				\item The induced test configuration of $(C,B_C+D_C)$ is $\tau$-equivariant.
			\end{enumerate}
		\end{proposition}
		
		By an involution $\tau$ of $(C,B_C+D_C)$, we mean that $\tau$ is an involution of $C$  satisfying $\tau^*(B_C)=B_C$ and $\tau^*(D_C)=D_C$.
		
		\begin{proof}
			By Proposition \ref{p:existsE}, there exist  curves $E_1$ and $E_2$ in $\Supp(\lfloor B+D\rfloor) $ such that 
			$E_i$ is not in the $\bG_m$-fixed locus of $X$ and $p_i \in E_i$.
			Since   $\lfloor B_C+D_C\rfloor=p_1+p_2$  and $(C,B_C+D_C)$ is an lc CY curve pair, $C\simeq \bP^1$ and
			$
			B_C+D_C = p_1+p_2
			$.
			By our assumption on $\tau$, there exists $a\in (0,1]$ such that
			\[
			B_C = (1-a) (p_1+p_2) \quad \text{ and } \quad D_C= a (p_1+p_2)
			.\]
			Since $(C,B_C)$ is klt at $p_i$, $(X,B+D)$ is plt at $p_i$ and so $a_i:=A_{X,B}(\ord_{E_i})>0$. 
			
			To construct the test configuration,
			fix an integer $r>0$ such that $L:= -r(K_X+B)$ is a Cartier divisor.
			Fix an  integer $c>0$ such that $c_1:= c/a_i$ and $c_2:= c/ a_i$ are integers.
			Now set 
			\[
			v_1 = c_1 \ord_{E_1} \quad \text{ and } \quad v_2:=c_2 \ord_{E_2}
			,\]
			which satisfy 
			$
			A_{X,B}(v_1) = c = A_{X,B}(v_2)
			$.
			Proposition \ref{p:tcdivisorsonX}
			implies that the  filtration $F^\bullet$ of $R:=R(X,L)$ defined by 
			\[
			F^\la R_m :=  \cap_{i=1}^2\{ s\in R_m \, \vert\, v_i(s) \geq \la +mr c   \}
			\]
			is finitely generated and induces a  test configuration $(\cX,\cB+\cD)$  of $(X,B+D)$.
			Furthermore, $(X_0,B_0+D_0):=(\cX_0,\cB_0+\cD_0)$ satisfies (1) by Propositions \ref{p:tcdivisorsonXextra} and \ref{p:torusactionfinitekernel->toric}.
			
			To verify (2), 
			let $(\cC, \cB_\cC+\cD_\cC)$ denote the induced test configuration of $(C,B_C+D_C)$
			and write $F_C^\bullet$ for the corresponding filtration of $R_C := R(C,L\vert_C)$. 
			By the formula for $F^\bullet$,
			\[
			F^{-mrc}R_m = R_m \quad \text{ and } {\rm Bs}(F^{-mrc+1}R_m) = E_1\cup E_2.
			\]
			for $m>0$ sufficiently divisible.
			Thus, using Proposition \ref{p:tcrestriction}, we see that
			\begin{equation}\label{e:filtFC}
				F_C^{-mrc} R_{C,m}= R_{C,m} \quad \text{ and } \quad 
				{\rm Bs} (F_C^{-mrc+1} R_{C,m}) = \{ p_1, p_2\}
			\end{equation}
			for $m>0$ sufficiently divisble. 
			Now, note that $\cC_0$ has at most two irreducible, since $(\cC_0,\cB_{\cC_0})$ is an  slc log Fano curve pair. 
			Using Lemma  \ref{l:tcfiltformula} and \eqref{e:filtFC}, we see  that 
			\[
			F^\la_C R_{C,m} := \cap_{i=1}^{2} \{ s\in R_{C,m} \, \vert\, w_i(s) \geq \la+m rc\}
			,\]
			where  
			$w_i:= (c/a) \ord_{p_i}$.
			Since $\tau$-interchanges $w_1$ and $w_2$, the filtration $F_C^\bullet$ of $R_C$ is $\tau$-invariant. 
			Using that $\cC = \Proj ( {\rm Rees}(F_C))$, we conclude that  the fiberwise involution on $\cC\setminus \cC_0$ extends to an involution of $\cC$. Therefore (2) holds.
		\end{proof}

		\subsection{Not necessarily normal pairs}\label{ss:nonnormalpairs}
		
		\subsubsection{Classification}
		In order to produce weakly special degenerations of non-normal boundary polarized CY surface pairs, we need to analyze the geometry of their conductor divisors.

		\begin{proposition}\label{p:slcclassification}
			Let $(X,B+D)$ be a boundary polarized CY surface pair and write 
			\[
			(\oX,\oG+\oB+\oD):= \sqcup_{i=1}^r(\oX_i,\oG_i+\oB_i+\oD_i)
			\]
			for its normalization and decomposition into irreducible components. 
			For each $i$, either
			\begin{center}
				\textrm{(1) $ \oG_i = \emptyset$,\quad  (2) $ \oG_i = \bP^1$,\quad or (3) $\oG_i=\bP^1 \cup \bP^1$,}
			\end{center}
			where in (3) the curves meet at a node. 
			Furthermore, after possible reordering the components, one of the following holds:
			\begin{enumerate}
				\item[(A)] $X$ is normal.
				\item[(B)] $X$ is non-normal, irreducible, and $\oG$ satisfies (2). 
				\item[(C)] $X$ is a union of two normal surfaces with $\oG_1$ and $\oG_2$ both satisfying (2). 
				\item[(D)] $X$ is a union of $r\geq 1$ normal surfaces $\oX_1,\ldots, \oX_r$ such that each $\oG_i$ satisfies (3) and each $\oX_{i}$ is glued along a rational curve to $\oX_{i+1 \text{ mod } r} $.
				\item[(E)] $\oX$ is a union of $r\geq 1$ surfaces $\oX_1,\ldots, \oX_r$ such that
				$\oX_i$ is glued to  $\oX_{i+1}$ along a rational curve for $1\leq i\leq r-1$
				and $\oG_i$ satisfies (3) for $1< i < r$, while $\oG_1$ and $\oG_r$ satisfy either (2) or (3).
			\end{enumerate}
			Additionally, in case (E) with $r\leq 2$, we require that at least one of the $\oG_i$ satisfies (3) so that we are not in in cases (B) or (C).
		\end{proposition}
		
		See \cite[Theorem 5.5]{Hac04} for a similar statement when $B=0$.
		
		\begin{proof}
			Since $\oG_i +\oB_i$ is the log Fano boundary of  $(\oX_i,\oG_i+\oB_i+\oD_i)$, 
			Proposition \ref{p:normallogFanoboundary} immediately implies the first statement.
			The second statement uses that $X$ is the quotient of $\oX$ by a finite equivalence relation given by a generically fixed point free involution $\tau:\oG^n\to \oG^n$. 
			Note that the involution either identifies a curve  in $ \oG$ with a different curve or induces a self gluing of the curve.
			Therefore we get the above cases. (Note that, in case (E), when $\oG_1$ or $\oG_r$ satisfies (3), then one of the two curve in the support must be self glued.)
		\end{proof}
		
		\subsubsection{Degeneration}
		We are now ready to prove the main result of this section.
		
		\begin{theorem}\label{t:TypeIIItoricdegen}
			If $(X,B+D)$ is a Type III boundary polarized CY surface pair, then there exists a weakly special degeneration 
			\[
			(X,B+D) \rightsquigarrow(X_0,B_0+D_0)
			\]
			such that the normalization of $(X_0,B_0+D_0)$ is a union of toric pairs. 
		\end{theorem}

		Throughout the proof, we will repeatedly use the following two  lemmas sometimes without mention.
		The first lemma allows us to reduce  proving the theorem to constructing  a sequence of weakly special degenerations to a pair whose normalization is toric.
		The second lemma allows us to ``glue'' test configurations of the normalization.

		\begin{lem}\label{l:connecting}
			If there exist two  weakly special degenerations of boundary polarized CY pairs
			\[
			(X,B+D) \rightsquigarrow(X_0,B_0+D_0)
			\rightsquigarrow(X'_0,B'_0+D'_0)
			,\]	
			then there is a weakly special degeneration 
			\[
			(X,B+D)\rightsquigarrow(X'_0,B'_0+D'_0)
			.\]
		\end{lem}
		
		\begin{proof}
			The statement follows from the argument in \cite[Proof of Lemma 6.10]{BABWILD}.
		\end{proof}

		\begin{lem}\label{l:tcgluing}
			Let $(X,B+D)$ be  a boundary polarized CY pair and  $(\ocX,\ocG+\ocB+\ocD)$ be a test configuration of the normalization $(\oX,\oG+\oB+\oD)$.
			
			If the induced test configuration of $(\oG^n, B_{\oG^n} + D_{\oG^n})$ is equivariant with respect to the involution of $\oG^n$, then there exists a test configuration
			$
			(\cX,\cB+\cD)
			$
			of $(X,B+D)$ whose normalization is $(\ocX,\ocG+\ocB+\ocD)$.
		\end{lem}
		
		\begin{proof}
			The statement follows from  \cite[Proposition 2.12]{BABWILD}.
		\end{proof}

		\begin{proof}[Proof of Theorem \ref{t:TypeIIItoricdegen}]
			We will prove the result by analyzing Cases (A)--(E) of the possible geometry of $(X,B+D)$ listed in Proposition \ref{p:slcclassification}.
			In some cases, we reduce the result to a later cases via  applying a weakly special degeneration.	
			\medskip 
			
			\noindent \emph{Case (A)}.
			Since $X$ is normal and $(X,B+D)$ Type III, \cite[Proposition 8.17]{BABWILD} implies that there exists a non-trivial weakly special degeneration 
			\[
			(X,B+D)\rightsquigarrow (X_0,B_0+D_0)
			\]
			such $X_0$ is irreducible and $\lfloor B_0+D_0\rfloor = B_0+D_0$.
			by the non-triviality of the degeneration, there is a $\bG_m$-action on $\lfloor B_0+D_0 \rfloor$
			If $X_0$ is non-normal, then we can use Lemma \ref{l:tcgluing} to reduce to prove the statement in cases (B)--(E). 
			If $X_0$ is normal, then we claim that there is a curve $E\subset \lfloor B_0+D_0\rfloor $ not in the $\bG_m$-fixed locus. 
			To see the claim, fix a curve $C\subset \lfloor B_0+D_0\rfloor$.
			Since $(X,B+D)$ is Type III, $(C,{\rm Diff}_C(B+D-C))$ is not klt. 
			Thus Proposition \ref{p:existsE} implies that either $C=E$ or a curve $E$ intersecting $C$ satisfies the claim. 
			Now Proposition \ref{p:tcdivisorsonX}  and  \ref{p:tcdivisorsonXextra}
			applied to $v_1 = \ord_E$ produces a weakly special degeneration of 
			\[
			(X_0,B_0+D_0) \rightsquigarrow (X'_0,B'_0+D'_0)
			\]
			such that $X'_0$ is irreducible and the $\bG_m^2$ acts on $X'_0$ has finite kernel. Therefore the normalization of $(X'_0,B'_0+D'_0)$ is toric by Proposition \ref{p:torusactionfinitekernel->toric}.
			\medskip 
			
			Before  proceeding to the next case, we make a reduction.
			We claim that if $(X,B+D)$ satisfies (B) or (C) of Proposition \ref{p:slcclassification}, then it admits a weakly special degeneration to a pair 
			\[
			(X,B+D)\rightsquigarrow (X_0,B_0+D_0)
			\]
			such that one of the following holds:
			\begin{enumerate}
				\item  The pair $(X_0,B_0+D_0)$ satisfies  (B) or (C) and admits an  $\bG_m$-action that is non-trivial on each component and fixes the conductor divisor.
				\item The pair $(X_0,B_0+D_0)$ satisfies (D) or (E). 
			\end{enumerate} 
			Assuming the claim, to prove the theorem it suffices to verify cases (B) and (C) when (1) holds and then prove cases (D) and (E).
			
			To see the latter claim holds, let
			$
			(\oX,\oG+\oB+\oD):= \sqcup_{i=1}^r 
			(\oX_i,\oG_i+\oB_i+\oD_i)
			$
			denote the normalization of  $(X,B+D)$.
			By Proposition \ref{p:tcdivisorsonX} applied to the valuation $\ord_{\overline{G}_i}$ on $\oX_i$ and Lemma \ref{l:tcgluing}, there exists a weakly  special degeneration 
			$
			(X,B+D) \rightsquigarrow(X_0,B_0+D_0)
			$
			such that the $\bG_m$-action on $X_0$ acts non-trivially on each component and fixes the degeneration of $\overline{G}$. 
			Thus, if $(X_0,B_0+D_0)$ satisfies (B) or (C), then (1) holds. 
			If not, then (2) holds and the proof of the claim is complete.
			\medskip 
			
			\noindent \emph{Case (B)}. By the above discussion, we may  assume that (1) holds.
			Now consider the involution $\tau$ on  the lc CY curve pair $(\oG, B_{\oG}+D_{\oG})$. 
			By the Type III assumption, $(\oG, B_{\oG}+D_{\oG})$ is not klt and so  $\lfloor B_{\oG}+D_{\oG} \rfloor $ is the sum of two points or a single point. 
			If there are two points and $\tau$ interchanges them, then Proposition \ref{p:degeninvolution} and Lemma \ref{l:tcgluing} imply that there exists  a weakly special degeneration 
			\[
			(X,B+D)\rightsquigarrow(X_0,B_0+D_0)
			\]
			such that the normalization of the later pair is toric. 
			
			If not, then $\tau$ fixes the points of $\lfloor B_{\oG}+D_{\oG} \rfloor$. 
			Choose an arbitrary  point $p \in \Supp( \lfloor B_{\oG}+D_{\oG}\rfloor )$. By Proposition \ref{p:existsE}, there exists a curve $E\subset \lfloor \oB+\oG \rfloor$ such that $p \in E\cap \oG$.
			By  Proposition \ref{p:degenbyE} and Lemma \ref{l:tcgluing}, $\ord_E$ induces a weakly special degeneration
			\[
			(\oX,\oG+\oB+\oD)\rightsquigarrow(\oX_0,\oG_0+\oB_0+\oD_0)
			\]
			such that the latter pairs normalization is toric.
			
			We claim that induced test configuration of $(\oG,B_{\oG}+D_{\oG})$ is equivariant with respect to the involution $\tau:\oG\to \oG$. 
			Indeed,  the  filtration of the section ring of $(\oG,B_{\oG}+D_{\oG})$ induced by the test configuration is given by the formula in Lemma \ref{l:tcfiltformula}, where the valuations are of the form  $v_1= c_1 \ord_{p_1},\ldots,v_r= c_r \ord_{p_r}$ and the $p_i$ are lc places of $(\oG,\oB_{\oG}+D_{\oG})$.
			Thus $\tau$ fixes each $v_i$ and hence fixes the filtration. 
			Since the test configuration is the $\Proj$ of the Rees algebra of the filtration, the test configuration is $\tau$-equivariant.
			Thus  Lemma \ref{l:tcgluing} produces a weakly special degeneration 
			\[
			(X,B+D) \rightsquigarrow (X_0,B_0+D_0)
			\]
			to a pair whose normalizaiton is toric.
			\medskip 
			
			\noindent \emph{Case (C)}. As above, we may assume that (1) holds.
			Since $(X,B+D)$ is Type III, $(\oG_i, B_{\oG_i}+D_{\oG_i})$ is an lc CY curve pair that is not klt. 
			Thus we may choose points 
			\[
			p_1 \in \lfloor B_{\oG_1}+ D_{\oG_1}\rfloor
			\quad \text{ and } \quad
			p_2 \in \lfloor B_{\oG_2}+ D_{\oG_2}\rfloor
			\]
			such that $p_1 = \tau (p_2)$, where $\tau$ is the involution of $\oG$.
			By Proposition \ref{p:existsE}, there exist curves $E_i\subset \lfloor \oB_i +\oD_i \rfloor $ such that $p_i \in E_i \cap \oG_i$. 
			Then Proposition \ref{p:degenbyE} produces a weakly special degeneration
			\[
			(\oX_i,\oG_i+\oB_i+\oD_i) \rightsquigarrow (\oX_{i0},\oG_{i0}+\oB_{i0}+\oD_{i0})
			\]
			for $i=1,2$ 
			such that the normalization of the pair on the right is toric and the induced test configurations of $\oG_i$ are induced by a  $m_i\ord_{p_i}$ for some positive integer $m_i$. 
			After scaling these two test configurations, we may assume that $m_1=m_2$.
			Thus Lemma \ref{l:tcgluing} produces  a weakly special degeneration 
			\[
			(X,B+D)\rightsquigarrow(X_0,B_0+D_0),
			\]
			where the normalization of the latter pair is toric.
			\medskip 
			
			\noindent \emph{Cases (D) and (E)}.
			We claim that  after possibly replacing the pair with a weakly special degeneration, we may assume that $\lfloor \oG_i+\oB_i\rfloor$ is a union of two  curves for all  $1\leq i \leq r$. 
			Indeed, the statement holds in Case (D). If we are in Case (E), then the statement holds for $1<i < r$. 
			If the statement does not hold for $i=1$, let $C_1\subset \oG_1$ be the curve that gets glued to a curve  $C_2 \subset \oG_2$. 
			Now  consider  the weakly special degenerations  
			\[
			(\oX_i , \oG_i +\oB_i+\oD_i ) \rightsquigarrow (\oX_{i0} , \oG_{i0} +\oB_{i0}+\oD_{i0} ) 
			\]
			produced by Proposition \ref{p:tcdivisorsonX} applied to $\ord_{C_1}$ for $i=1$ and the trivial test configuration for $1<i\leq r$.
			By Proposition \ref{p:tcdivisorsonX},  the induced $\bG_m$-action on $\oX_{10}$ is non-trivial and induces the trivial test configuration on $C_{1}$.  
			By Lemma \ref{l:tcgluing},  the above weakly special degenerations glue to give a weakly special degeneration
			\[
			(X,B+D)\rightsquigarrow(X_0,B_0+D_0)
			.\]
			So after replacing $(X,B+D)$ with $(X_0,B_0+D_0)$, we may assume that $(\oX_1,\oB_1+\oG_1+\oD_1)$ admits a $\bG_m$-action that fixes $C_1$. 
			Since $\lfloor \oG_{2} +\oB_{2}\rfloor$ is the union of two smooth  curves meeting at a point 
			\[
			\lfloor {\rm Diff}_{C_{2}}(\oG_{2}+\oB_{2}-C_{2})\rfloor \neq \emptyset 
			.
			\]
			Using that the involution $\tau:\oG^n \to \oG^n$ sends ${\rm Diff}_{\oG^n}( \oB)$ to itself,
			we see 
			\[
			\lfloor {\rm Diff}_{C_{1}}(\oG_{1}+\oB_{1}-C_{1})\rfloor \neq 0
			\]
			and so Proposition \ref{p:existsE} implies that there exists a curve $E \subset \lfloor G_1+B_1-C_1\rfloor$. Thus $\lfloor G_1+B_1\rfloor $ is a union of two curves as desired. Arguing similarly for $i=r$ proves the claim.

			Now, using that  each $\lfloor \oG_i + \oB_i\rfloor$ is a union of two   curves  for all $1\leq i \leq r$,
			Proposition \ref{p:twostep} implies that each component of the normalization of $(X,B+D)$ admits a sequence of weakly special degenerations
			\[
			(\oX_i,\oG_i+\oB_i+\oD_i) 
			\rightsquigarrow
			(\oX_{i0},\oG_{i0}+\oB_{i0}+\oD_{i0})
			\rightsquigarrow
			(\oX'_{i0},\oG'_{i0}+\oB'_{i0}+\oD'_{i0})
			\]
			such that 
			\begin{enumerate}
				\item[(i)] the normalization of the last pair is toric, 
				\item[(ii)] the induced test configuration on each component of   $\lfloor \oG_i+\oB_i \rfloor $ is trivial, and 
				\item[(iii)] if  we write  $\lfloor \oG_{i0}+\oB_{i0} \rfloor=C_{i0}^1\cup C_{i0}^2$, then the   induced test configuration of $C_{i0}^j$ is given by  $m_i \ord_{p_{i0}}$, where   $m_{i} $ is a positive integer   and $p_{i0} = C_{i0}^1 \cap C_{i0}^2$.
			\end{enumerate}
			After scaling the test configuration, we may assume that the $m_i$ are equal.
			By Lemma \ref{l:tcgluing},  the sequence of weakly special degenerations of the normalization descend to give  weakly special degenerations
			\[
			(X,B+D)
			\rightsquigarrow
			(X_0,B_0+D_0)
			\rightsquigarrow
			(X'_0,B'_0+D'_0)
			\]
			such that the normalization of $(X'_0,B'_0+D'_0)$ is toric. 
			Thus the result holds in this case by Lemma \ref{l:connecting}.
		\end{proof}
		
		\subsection{Fields of definition}
		Using the previous toric degeneration result, we prove a statement that will be used in Section \ref{s:asgm} to show that the Type III locus in the CY moduli space is discrete.
		
		\begin{corollary}\label{c:TypeIIIvariation}
			If $\bk'\subset \bk$ is an extension of algebraically closed fields and $(X,B+D)$ is a Type III boundary polarized CY surface pair defined over $\bk$, then there exists a weakly special degeneration 
			\[
			(X,B+D) \rightsquigarrow (X_0,B_0+D_0)
			\]
			such that $(X_0,B_0+D_0)$ is the base change of a boundary polarized CY pair defined over $\bk'$.
		\end{corollary}
		
		To prove the result, we will use Theorem \ref{t:TypeIIItoricdegen} and that a  toric pair over $\bk$ is the base change of a toric pair defined over $\bk'$.
		
		\begin{proof}
			By Theorem \ref{t:TypeIIItoricdegen}, there exists a weakly special degeneration of $(X,B+D)$ to a pair $(X_0,B_0+D_0)$ such that the normalization of $(X_0,B_0+D_0)$
			is a union of toric pairs
			\[
			(\oX_0,\oG_0+\oB_0+\oD_0):= \sqcup_{i=1}^r (\oX^i_0,\oG^i_0+\oB^i_0+\oD^i_0)
			.\]
			Thus there exists a $\bT:= \bG_m^2$-action on  $(\oX^i_0,\oG^i_0+\oB^i_0+\oD^i_0)$ and an origin ${\bf 1}_i \in \oX^i_0\setminus \Supp(\oG^i_0+\oB^i_0+\oD^i_0)$ such that the map $\bT \to \bT \cdot {\bf 1}_i$ defined by $t\mapsto t \cdot {\bf 1}_i$ is an isomorphism.
			Note that the structure of $(\oX^i_0,\oG^i_0+\oB^i_0+\oD^i_0)$ as a toric pair is not unique, since we can act on $\oX^i_0$ by an element in $g_i\in \bT(\bk)$ to produce a new origin $g_i\cdot {\bf 1}_i$.

			Now write $\tau: \oG_0^n\to \oG_0^n$ for the  generically fixed point free involution on the normalization of the conductor divisor and write
			\[
			({\oG^{i,n}_0}, B_{\oG^{i,n}_0}+ D_{\oG^{i,n}_0}) := \bigsqcup_{j=1}^r (C_{ij}, B_{ij}+D_{ij})
			,\]
			where each $C_{ij}$ is connected. 
			Since the pairs are on the right are toric CY pairs, there are isomorphisms
			\[
			(C_{ij}, B_{ij}+D_{ij}) \simeq (\bP^1_{\bk}, \{0\}+ \{\infty\})
			\]
			as pairs. (Here, we are  ignore the decomposition of the boundary into the sum of $B_{ij}$ and $D_{ij}$).
			On each pair, there is an induced origin ${\bf 1}_{ij}\in C_{ij}$ defined by ${\bf 1}_{ij}:= \lim_{t\to 0} \la(t)\cdot {\bf 1}_{ij}$, where $\la:\bG_m \to \bT$ is an arbitrary 1-PS such that the limit lies in $C_{ij}\setminus \Supp(B_{ij}+D_{ij})$.
			
			We seek to choose new origins so that the involution will be defined over $\bk'$.
			To proceed, note that the
			generically fixed point free involutions  $\bP^1_{\bk}\to \bP^1_{\bk}$ sending $\{0,\infty\}$ to $\{0,\infty\}$
			are of the form:
			\begin{enumerate}
				\item[(i)] $ [x:y]\mapsto [x:-y]$ or
				\item[(ii)] $[x:y]\mapsto [y:cx]$ for some $c\in \bk^\times$. 
			\end{enumerate}
			The first involution fixes both $0$ and $\infty$ and sends $[1:1]$ to $[1:-1]$. 
			The second involution interchanges $0$ and $\infty$ and fixes both $[1: \pm\sqrt{c}]$.
			By the above classification, 
			we may choose points $p_{ij} \in C_{ij} \setminus \Supp(B_{ij}+D_{ij})$ satisfying the following: if  $\tau(C_{ij}) = C_{kl}$, then $\tau(p_{ij})=p_{kl}$ when  we are not in the case case when $C_{ij} = C_{kl}$ and  $\tau$ fixes both points in $\Supp(B_{ij}+D_{ij})$.
			Using that $\oG^i_0 \subset \oX^i_0$ is either a toric invariant curve  or a  union of two intersecting toric invariant curves on $\oX^i_0$, 
			there exists a new origin ${\bf 1}_i \in \oX^i_0$ such that ${\bf 1}_{ij}= p_{ij}$ for each $j$.
			The new choice of origins endows the connected components of 
			$(\oX_0,\oG_0+\oB_0+\oD_0)$ with the structure of a toric pair. 
			Additionally, the involution 
			\[
			\tau:(\oG^n,B_{\oG^n}+D_{\oG^n})\to (\oG^n,B_{\oG^n}+D_{\oG^n})
			\]
			sends the toric boundary and  origins to points defined over $\bk'$.
			Here, we are using that ${\bf 1}_{ij}$ is a $\bk'$-point and that $\tau({\bf 1}_{ij})$ equals either ${\bf 1}_{ij}$ or $-{\bf 1}_{ij}$.
			Thus there exists a
			not necessarily connected boundary polarized CY surface pair $(\oX'_0,\oG'_0+\oB'_0+\oD'_0)$ defined over $\bk'$ and an involution
			$\tau': \oG'^n_0 \to {\oG'}^n_0$ whose base change is our original toric pair and involution. 
			The gluing relation $R({\tau}')\rightrightarrows \oX'_0$ as in \cite[Definition 5.31]{Kol13} has finite equivalence classes, since its base change to $\bk$ has the same property (alternatively, this can be seen more explicitly as we are in dimension two). 
			Thus the geometric quotient of the pair exists by \cite[Corollary 5.33]{Kol13} and is a boundary polarized CY pair $(X'_0,B'_0+D'_0)$  over $\bk'$ whose base change to $\bk$ is $(X_0,B_0+D_0)$.
		\end{proof}

		\section{Asymptotically good moduli spaces}\label{s:asgm}
		
		In this section, we construct an asymptotically good moduli space parametrizing boundary polarized CY surface pairs and describe its properties.  Theorem \ref{t:mainintro} will  be deduced as a consequence of these results.
		\medskip
		
		\subsection{Statement of main result}
		
		Throughout Section \ref{s:asgm}, we fix a finite type locally closed substack 
		$\cM^\circ \subset \cM(\chi,N, {\bf a},c)$   of the moduli stack in Definition \ref{def:moduli-stack} such that there exists $v>0$ such that $\chi(m  N) = v  m^2 + O(m)$ for $m\geq 0$ and the coefficients  $a_1,\ldots, a_l$ are in  $\mathbf{T}:= \{\tfrac{1}{2}, \tfrac{2}{3}, \tfrac{3}{4},\ldots \}\cup\big\{1\big\}$. 
		We will use the following notation:
		\begin{itemize}
			\item Let $\cM$ denote the stack theoretic closure of $\cM^\circ$ in $\cM(\chi,N,{\bf a}, c)$ and $\cM^{\CY}$ denote its seminormalization.
			\item  Let $\cM_m\subset \cM$ denote the open substack of $\cM$ parametrizing pairs $(X,B+D)$ in $\cM$ such that $\ind_x(K_X+B) \leq m$ for each $x\in X$.  See Section \ref{ss:indexboundsubstack}.
			There is a chain of inclusions
			\[
			\cM_1 \subset \cM_2 \subset \cM_3 \subset \cdots \subset \cM
			.\]
			\item 
			Let $\cM^{\K}$ and $\cM^{\KSBA}$ denote the open substacks of $\cM$ defined in Section \ref{ss:KSBA+K}. 
			We write 
			$\cM^{\K} \to M^{\K}$ and $\cM^{\KSBA}\to M^{\rm KSBA}$ for their good moduli space morphisms.   
			
		\end{itemize}
		By Theorem \ref{t:existsGMindex}, if $m\geq N$, then there exists a good moduli space $\cM_m\to M_m$.
		By the universality of good moduli space morphisms, there exists a commutative diagram of morphisms, where
		\begin{equation}\label{e:diagMmMm+1}
			\begin{tikzcd}
				\cM_m \arrow[r,hook]\arrow[d,"\phi_m"]& \cM_{m+1}\arrow[r,hook]\arrow[d,"\phi_{m+1}"] & \cM_{m+2} \arrow[r,hook]\arrow[d,"\phi_{m+2}"]& ...\\
				M_m \arrow[r,"\psi_m"] & M_{m+1}\arrow[r,"\psi_{m+1}"]& M_{m+2} \arrow[r,"\psi_{m+2}"]& ...
			\end{tikzcd}
		\end{equation}
		where the  inclusion in the top row  are the natural open embeddings and the vertical arrows are good moduli space morphisms.
		The following theorem is the main result of this paper.
		
		\begin{theorem}\label{t:mainCY}
			There exists a projective seminormal scheme $M^{\rm CY}$ and an asymptotically good moduli space morphism $\phi:\cM^{\CY} \to M^{\rm CY} $
			such that:
			\begin{enumerate}
				\item The morphism $\phi$ is universal among maps from $\cM^{\CY}$ to algebraic spaces.
				\item The map $\cM^{\CY}(\bk) \to M^{\CY}(\bk)$ is surjective and identifies two pairs if and only if they are S-equivalent.
				\item  There exists a commutative diagram
				\[
				\begin{tikzcd}
					\cM^{\rm K,\sn}\arrow[r,hook] \arrow[d]& \cM^{\CY} \arrow[d] &\cM^{\rm KSBA, \sn}\arrow[l,hook', swap]\arrow[d]\\
					M^{\rm K,\sn} \arrow[r]& M^{\rm CY} & M^{\rm KSBA,\sn} \arrow[l]
				\end{tikzcd}
				\]
				where the top row arrows are open immersions, the vertical arrows are (asymptotically) good moduli space morphisms, and the bottom row arrows are projective morphisms.
				In addition, if The Type I loci  is dense in $\cM$, then the bottom row arrows are  birational.

				\item The  Hodge line bundle on $M^{\rm CY}$ is ample.
			\end{enumerate}
		\end{theorem}
		
		In the above theorem, the superscript $\rm sn$ denotes the seminormalization of a scheme or algebraic stack. 
		For information on seminormalization in these settings, see  \cite[Section 10.8]{Kol23} and \cite[Section 13.2]{BABWILD}, respectively.
		
		The setup and notation in Theorem \ref{t:mainCY} differs slightly from that in Theorem \ref{t:mainintro}.
		In the above theorem, we allow the moduli space to parametrize pairs with non-trivial log Fano boundary as long as the marked coefficients are in $\bfT$. 
		In addition, $\cM^{\K}$
		and $\cM^{\KSBA}$ refer to  the open substacks in the seminormalization of $\cM^{\CY}$ in Theorem  \ref{t:mainintro}, rather than open substacks of $\cM$. There is a similar discrepancy between $M^{\K}$ and $M^{\KSBA}$.

		\subsection{Stabilization}
		
		
		\begin{theorem}[Stabilization]\label{t:stabilization}
			If $m\gg0$, then the map
			$
			M_m(\bk)\to M_{m+1}(\bk)
			$
			is a bijection. 
		\end{theorem}
		
		In order to prove the result, we first analyze the loci determined by type. 
		Let $\cM^{\rm I+II}\subset \cM$ denote the Type I+II substack. 
		Recall, if $m\gg0$, then $\cM^{\rm I+II}$ is a saturated open substack of $\cM_m$ by Proposition \ref{p:TypeI+IIsaturated} and hence 
		\[
		M_m^{\rm I+II}:=\phi_m(\cM^{\rm I+II}) \subset M_m 
		\]
		is an open subspace.
		We set $M_m^{\rm III}:= M_{m}\setminus M_m^{\rm I+II}$.
		
		\begin{proposition}\label{p:surj+isomI+II}
			If $m\gg0$, then $M_m \to M_{m+1}$ is surjective and
			$
			M_m^{\rm I+II}\to M_{m+1}^{\rm I+II}
			$
			is an isomorphism.
		\end{proposition}
		
		\begin{proof}
			Note that $\cM^\circ$ is finite type. 
			Thus, if $m\gg0$, then  $\cM^\circ \subset \cM_m$  and so $\cM= \overline{\cM_m}$, which implies $\cM_{m}$ is dense in $\cM_{m+1}$.
			For such $m$,   $\cM_m$ is dense in $\cM_{m+1}$ and so $\psi_{m}(M_m)$ is dense in $M_{m+1}$.
			In addition, if $m\gg0$, $M_m$ is proper by Proposition \ref{p:properness}
			and so  $\psi_m(M_m)$ is closed in $M_{m+1}$.
			Therefore, if $m\gg0$, then $M_m\to M_{m+1}$ surjective.
			
			Next, we analyze the Type I and II locus.
			By Proposition \ref{p:TypeI+IIsaturated}, if $m\gg0$, then $\cM^{\rm I+II}$ is a saturated open substack of $\cM_m$ and there is a commutative diagram 
			\[
			\begin{tikzcd}
				\cM^{\rm I+II}\arrow[r,"="]\arrow[d] & \cM^{\rm I+II} \arrow[d]\\
				M_{m}^{\rm I+II} \arrow[r]& M_{m+1}^{\rm I+II}
			\end{tikzcd}
			,\]
			where the vertical arrows are good moduli space morphisms. 
			By the  uniqueness and universality of good moduli space morphisms, the  arrow in the bottom row is an isomorphism.
		\end{proof}
		
		\begin{proposition}\label{p:TypeIIIlocusdiscrete}
			If $m\gg0$, then $M_{m}^{\rm III}$ is a finite set of closed points.
		\end{proposition}
		
		\begin{proof}
			Let $\psi_{m,l}: M_m \to M_l$ denote the composition $\psi_{l-1} \circ \cdots \circ \psi_m$ when $l> m$. 
			By Proposition \ref{p:surj+isomI+II}, there exists an integer $m$ such that $\psi_{m,l}(M_{m}^{\rm III})= M_{l}^{\rm III}$ for all $l\geq m$. 
			Thus, it suffices to prove that $\psi_{m,l}(M_{m}^{\rm III})$ is a finite set of closed points for $l\gg m$. 
			
			To prove the latter statement, fix an irreducible closed subset $Z\subset M_{m}^{\rm III}$.
			Let $\bK:= \overline{K(Z)}$ denote the algebraic closure of its fraction field 
			and $\overline{\eta}:\Spec(\bK)\to M_m$ denote the corresponding geometric point.
			Since $\bK$ is algebraically closed, $\cM_m(\bK)\to M_m(\bK)$ is surjective
			and so
			there exists a Type III pair $(X,B+D)$ in $\cM_m(\bK)$ 
			that maps to $\overline{\eta}\in M_m(\bK)$. 
			By Corollary \ref{c:TypeIIIvariation}, there exists a weakly special degeneration 
			\[
			(X,B+D)\rightsquigarrow (X_0,B_0+D_0)
			\]
			to a pair  $(X_0,B_0+D_0)$ that is the base change of a boundary polarized CY pair $(X^{\bk}_0,B^{\bk}_0+D^{\bk}_0)$ defined over $\bk$.
			Fix $l\geq m$ so that $(X^{\bk}_0,B^{\bk}_0+D^{\bk}_0)$ is in $\cM_l(\bk)$. 
			Now observe that 
			\[
			\psi_{m,l}(\overline{\eta}) = \phi_l([(X,B+D)]) = \phi_l([X_0,B_0+D_0)])
			\]
			and the latter $\bK$-point of $M_m$ factors through a $\bk$-point. 
			Thus $\psi_{m,l}(Z)$ is a closed point.
			Repeating the above argument for each irreducible component of $M_m^{\rm III}$ shows that $\psi_{m,l}(M_m^{\rm III})$ is a finite set of closed points when $l\gg m$.
		\end{proof}
		
		Theorem \ref{t:stabilization} is now a simple consequence of the the above results. 
		
		\begin{proof}[Proof of Theorem \ref{t:stabilization}]
			By Propositions \ref{p:surj+isomI+II}  and \ref{p:TypeIIIlocusdiscrete}, 
			there exists an integer $m_0$ satisfying: if $m\geq m_0$, then
			$
			\psi_m:M_{m}(\bk) \to M_{m+1}(\bk)
			$
			is surjective, the induced map 
			$
			M_m^{\rm I+II}\to M_m^{\rm I+II}
			$
			is an isomorphism, and $M_m^{\rm III}(\bk)$ is a finite set. 
			Thus, if $ m\geq m_0$,
			then 
			\[
			M_m^{\rm III}(\bk) \to M_{m+1}^{\rm III}(\bk) \to M_{m+2}^{\rm III}(\bk)\to \cdots
			\] 
			is a sequence of surjective  maps  between finite sets. Therefore the latter maps are all bijections for sufficiently large $m$
			and so  $M_m(\bk)\to M_{m}(\bk)$ is also a bijection for sufficiently large $m$.
		\end{proof}

		\subsection{Variation of source}\label{ss:vationsource}
		We now prove  a variation result, which  is the key ingredient in verifying the ampleness of the Hodge line bundle on our moduli space.
		Its proof relies on the results in Sections \ref{s:TypeII} and \ref{s:TypeIII} on Type II and III boundary polarized CY surface pairs.

		\begin{proposition}\label{p:variationsource}
			If $m\gg0$, then, for each $x \in M_m(\bk)$, the set of points with the same source
			\[
			W_x:=\{y \in M_m(\bk) \, \vert\, {\rm Src}(x) \simeq {\rm Src}(y) \}  
			\]
			is finite.
		\end{proposition}
		
		In the proposition, the notation ${\rm Src}(x)$ for a point $x\in M_m(\bk)$ denotes the source of a boundary polarized CY pair that maps to $y$ under the map $\cM_m(\bk)\to M_m(\bk)$.
		Note that ${\rm Src}(y)$ is well defined, since if two pairs in $\cM_m(\bk)$ map to $y$, then they are S-equivalent by  Proposition \ref{p:propertiesofGMs} and Remark \ref{r:Sequiv}.3 and so  they have the same source by Remark \ref{r:Sequiv}.1.

		\begin{proof}
			If $m\gg0$, then Proposition \ref{p:TypeI+IIsaturated} and Theorem \ref{t:stabilization}  imply that $\cM^{\rm I+II}\subset \cM_m$ and  that $M_m\to M_l$ is a  bijection on $\bk$-points for all $l\geq m$.
			If $x \in M_m(\bk)$ is in the Type I locus,  then $W_x$ is finite, since the source of a Type I pair $(X,B+D)$ is  the underlying klt CY pair and it admits finitely many possible markings in our moduli problem. 
			If $x\in M_m(\bk)$ is a Type III pair, then $W_x$ is finite by Proposition \ref{p:TypeIIIlocusdiscrete}.

			It remains to analyze the case when $x\in M_m(\bk)$ is in the Type II locus.
			Let $W:= \overline{W_x}\subset M_m$ denote the closure.
			Fix an irreducible component $Z\subset W$.
			By \cite[Theorem 14.11]{BABWILD}, there exists a smooth quasi-projective variety $T$ and commutative diagram
			\[
			\begin{tikzcd}
				T \arrow[r] \arrow[d]& \cM_m\arrow[d] \\
				Z \arrow[r] & M_m
			\end{tikzcd}
			\]
			such that $T\to Z$ is a dominant generically finite morphism and $T\to \cM_m$ maps closed points to closed points.
			The morphism $T\to \cM_m$ induces a family of boundary polarized CY pairs
			$(X,B+D)\to T$.
			Since the Type II locus  of $M_m$ is constructible, after possibly shrinking $T$, we may assume that all fibers of $(X,B+D)\to T$ are  Type II.  
			Since $\cM^{\rm I+II}\subset \cM_m$, the condition $T\to \cM_m$ maps closed points to closed points  implies that the fibers of $(X,B+D) \to T$ are polystable.
			
			We aim to show that $\dim Z=0$.
			Let $\eta \in T$ denote the generic point. 
			By Lemma \ref{l:srcfield} below, the source of $(X_{\overline{\eta}},B_{\overline{\eta}}+D_{\overline{\eta}})$ is the base change of a CY pair defined over $\bk$. 
			Thus Theorem \ref{t:TypeIIfieldofdef} implies that $(X_{\overline{\eta}},B_{\overline{\eta}}+D_{\overline{\eta}})$ is  the base change of a pair defined over $\bk$. The latter pair over $\bk$ maps to  a closed point $p \in M_m$.
			Thus the composition $T \to Z\hookrightarrow  M_m$ maps $\eta$ to $p$.
			Since $T\to Z$ is dominant, $Z= \{p\}$.
			Arguing the same way for each irreducible component of $W$ shows that $W$ is finite  and hence so $W_x$ is also finite. 
		\end{proof}
		
		\begin{lemma}\label{l:srcfield}
			Let $(X,B+D)\to T$ be a family of Type II boundary polarized CY surface pairs over a normal variety $T$.
			If $x \in T(\bk)$ and
			\[
			\{ t\in T(\bk) \, \vert\, {\rm Src}(X_t, B_t+D_t) \simeq {\rm Src}(X_x,B_x+D_x)\}
			\]
			is dense in $T$,
			then the source of the geometric generic fiber of the family is the base change of ${\rm Src}(X_x,B_x+D_x)$.
		\end{lemma}
		
		\begin{proof}
			Throughout the proof, we may replace $T$ with a dense open subset or a generically finite cover, since this will not change the hypotheses or the geometric generic fiber.
			Thus, by replacing $T$ with an open set, we may assume that $T$ is smooth. 
			By replacing $(X,B+D)$ with an irreducible component of its normalization, we may assume that $\oX$ is normal.
			Next, let 
			\[
			(Y,\Delta)\to (X,B+D)
			\]
			be a dlt modification. 
			By replacing  $T$ with a dense open set, we may assume that 
			$
			(Y_t, \Delta_t)\to (X_t,B_t+D_t)
			$ is a dlt modification for all $t\in T$. 
			By replacing $T$ with a generically finite cover, we may assume that there exists a prime divisor $C\subset \Supp(\lfloor \Delta \rfloor )$
			such that $C\to T$ is dominant and has irreducible fibers. 
			Let $\Gamma:= {\rm Diff}_C(\Delta-C)$ and note that $(C, \Gamma)$ is an lc pair by adjunction.
			By replacing $T$ with a dense open subset, we may assume $(C,\Gamma)\to T$ is  a family of lc pairs. 
			Since the fibers of $(X,B+D)\to T$ are Type II surface pairs, $(C_t,\Gamma_t)$ is the source of $(X_t,B_t+D_t)$ for $t\in T$ and so each $(C_t,\Gamma_t)$ is a klt CY curve pair. 
			Thus either 
			\begin{enumerate}
				\item[(a)]  $C\to T$ is a family of genus 1 curves and $\Gamma=0$ or
				\item[(b)] $C\to T$ is a family of genus 0 curves and each $\Supp(\Gamma_t)$ contains at least three points.
			\end{enumerate}
			In case (a), after replacing $T$ with a generically finite cover, $C\to T$ admits a section. Hence we get a family of elliptic curves.
			In case (b), after replacing $T$ with a generically finite cover, we can write
			$\Supp(\Gamma)= \Gamma_1+\cdots + \Gamma_r$, where the $\Gamma_i$ are prime divisors and $C\to T$ admits sections with  images $\Gamma_1,\ldots, \Gamma_r$. 
			Thus we get a family of genus 0 curves with $r\geq 3$ marked points. 
			This give a map
			\[
			\pi:T\to M^{\rm curve}
			\]
			to the coarse moduli space of elliptic curves or genus zero curves with $r\geq 3$ marked points. 
			By our assumption on the sources,  $\pi(x)$ is dense in $\pi(T)$ and so $\{\pi(x)\}=\pi(T) $. 
			Thus the geometric generic fiber of $(C,\Gamma)\to T$ is the base change of $(C_x,\Gamma_x)$.
		\end{proof}

		\subsection{Hodge Line Bundle}
		We now proceed to prove the ampleness of the Hodge line bundle on our moduli space.
		Before proceeding, we recall its definition in various settings.  
		
		\begin{definition}[Hodge Line Bundle]
			\hfill
			
			\begin{enumerate}
				\item If $f:(X,B+D)\to T$ is a family of boundary polarized CY pairs in $\cM(\chi, N,{\bf a},c)$, then the \emph{Hodge line bundle} of $f$ is the sheaf
				\[
				\la_{{\rm Hodge}, f,N}: = f_* \omega_{X/T}^{[N]}(N(B+D))
				.\]
				By \cite[Proposition 14.7.1]{BABWILD}, the above sheaf  is the unique line bundle on $T$ such that $\omega_{X/S}^{[N]}(N(B+D))=f^* \la_{{\rm Hodge},f,N}$.
				
				\item The line bundle in (1) is functorial by \cite[Proposition 14.7.2]{BABWILD} and so induces a  line bundle on $\cM$ that we denote by $\la_{{\rm Hodge},N}$.

				\item Since $\cM_m\to M_m$ is a good moduli space morphism, 
				\cite[14.9 and 14.10.1]{BABWILD} implies that, for any sufficiently divisible positive integer $l$, the line bundle $\la_{{\rm Hodge},N}^{\otimes l}$ restricted to  $\cM_m$ descends to a line bundle $L_{ {\rm Hodge},Nl}$. 
				We set $L_{\rm Hodge}:= L_{ {\rm Hodge}, Nl}^{\otimes 1/l}$, which is a $\bQ$-line bundle, and a bit abusively refer to it as the \emph{Hodge line bundle} on $M_m$.
			\end{enumerate}	
		\end{definition}
		
		\begin{rem}
			The Hodge line bundle is related to various other important functorial line bundles that appear when studying families of Calabi--Yau varieties. 
			
			\begin{enumerate}
				\item  The Hodge line bundle is closely related to the moduli divisor of an lc trivial fibration in \cite{Kaw97,Amb05}.
				See \cite[Proposition 14.7.3]{BABWILD} for a  precise statement. 
				
				\item The Hodge line bundle agrees up to a positive constant with the associated CM line bundle of the polarized family of pairs;  see \cite[pg. 17]{ADL19} for details.
			\end{enumerate}
		\end{rem}

		\begin{theorem}[Ampleness]\label{t:amplehodge}
			If $m\gg0$, then the Hodge line bundle  $L_{\rm Hodge}$  on $M_m$ is ample.
		\end{theorem}

		The proof relies on a positivity result for the Hodge line bundle on a family of boundary polarized CY pairs with maximally varying source  in \cite[Theorem 14.13]{BABWILD}, which is a consequence of positivity results for the moduli divisor of an lc trivial fibration in \cite{Kaw97,Amb05,Kol07}.

		\begin{proof}[Proof of Theorem \ref{t:amplehodge}]
			Fix $m\gg0$ so that the conclusions of Propositions 
			\ref{p:properness} and  \ref{p:variationsource}  hold. 
			By the Nakai-Moshezon criterion for ampleness on proper algebraic spaces \cite[Theorem 3.11]{Kol90},
			it suffices to verify that $L_{\rm Hodge}\vert_Z^{\dim Z}> 0 $ for each irreducible closed subset $Z\subset X$ of positive dimension. 
			By \cite[Theorem 14.13]{BABWILD} (and its proof for the case when $\cM$ parametrizes pairs with more than one marked divisors in the log Fano boundary), the latter holds if, for each $x \in Z(\bk)$, the subset
			\[
			\{ y\in Z(\bk)\, \vert\, {\rm Src}(y)\simeq {\rm Src}(x) \} \subset Z(\bk)
			\]
			is finite.
			Since $M_m$ satisfies the conclusion of Proposition \ref{p:variationsource}, the finiteness holds and so $L_{\rm Hodge}\vert_Z^{\dim Z}> 0 $. 
			Therefore $L_{\rm Hodge}$ is ample on $M_m$.
		\end{proof}
		
		\begin{corollary}[Projectivity]\label{c:projectivity}
			If $m\gg0$, then $M_m$ is projective. 
		\end{corollary}
		
		\begin{proof}
			This follows immediately from Proposition \ref{p:properness} and Theorem \ref{t:amplehodge}.
		\end{proof}
		
		\begin{remark}
			The b-semiampleness conjecture of Prokhorov and Shokurov predicts that the moduli divisor of a lc trivial fibration $(X,D) \to Y$ is base point free; see  \cite[Conjecture 7.13.1]{PS09} for details. 
			The proof of the conjecture when $\dim X-\dim Y=2$ was completed in \cite{BABWILD} using the ampleness of the Hodge line bundle on 
			the CY moduli compactification of pairs of the form $(\bP^2, \tfrac{3}{d} C)$.
			The results of this paper give a  slightly different proof. 
			
			The proof of the b-semiampleness conjecture in relative dimension two  relies on
			\cite{Fil20}, which reduces the conjecture to the case when either $D=0$ or $X\to Y$ is a Mori fiber space (in which the general fiber of $(X,D)\to Y$ is a boundary polarized CY pair).
			When $D=0$, the conjecture holds  by \cite{Fuj03}.
			When $X\to Y$ is a Mori fiber space, \cite{Fil20} further reduces the conjecture  to the case when a general fiber of $X\to Y$ is $\bP^2$, which was then proven in \cite{BABWILD}.
			Using the results of this paper and, in particular, Theorem \ref{t:amplehodge}, we can prove an extension of  \cite[Proposition 4.15]{BABWILD} to that the moduli divisor of an lc trivial fibration $(X,D)\to Y$ is b-semiample when $X\to Y$ is a Mori fiber space of relative dimension two without reducing to the $\bP^2$ case.
		\end{remark}
		
		\subsection{Asymptotically good moduli spaces}

		We will now construct an asymptotically good moduli space for $\cM^{\CY}$ and then prove Theorem \ref{t:mainCY}. 
		
		\begin{definition}
			A morphism $\phi:\cX\to X$ from an algebraic stack to an algebraic space  is an \emph{asymptotically good moduli space} if there exists an ascending chain of open substacks of $\cX$ 
			\[
			\cU_1 \subset \cU_2 \subset \cU_3 \subset \cdots 
			\]
			such that $\cX = \cup_{i \geq 1} \cU_i$ and the composition $\cU_i \to \cX\to X$ is a good moduli space morphism for all $i \geq 1$.
		\end{definition}
		
		
		
		\begin{proposition}\label{p:existsAGMcy}
			There exists an asymptotically good moduli space morphism 
			\[
			\phi: \cM^{\CY}\to M^{\rm CY}
			\]
			and, for $m\gg0$, the composition 
			$
			\cM_m^{\sn} \hookrightarrow \cM^{\CY} \to M^{\rm CY}
			$
			is a good moduli space morphism. 
		\end{proposition}
		
		For the definitition and properties of the seminormalization of an algebraic stack, see  \cite[Section 13.2]{BABWILD}.
		
		\begin{proof}
			Since seminormalizaiton is functorial, taking the seminormalization of the stacks (and schemes) in \eqref{e:diagMmMm+1} induces a commutative diagram 
			\[
			\begin{tikzcd}
				\cM_m^{\rm sn} \arrow[r,hook]\arrow[d]& \cM_{m+1}^{\rm sn}\arrow[r,hook]\arrow[d] & \cM_{m+2}^{\rm sn} \arrow[r,hook]\arrow[d,]& ...\\
				M_m^{\rm sn} \arrow[r,] & M_{m+1}^{\rm sn} \arrow[r]& M_{m+2}^{\rm sn} \arrow[r]& ...
			\end{tikzcd}
			,\]
			where the maps in the top row are open embeddings. 
			Since the seminormalization of a good moduli space morphism is a good moduli space morphism by
			\cite[Lemma 13.7]{BABWILD}, the vertical maps are good moduli space morphisms.
			
			By Theorems \ref{t:stabilization} and \ref{t:amplehodge} and Proposition \ref{p:properness}, after possibly increasing $m$, the bottom row maps are proper morphisms of seminormal projective schemes and isomorphisms on $\bk$-points. 
			Thus the bottom row maps are all isomorphisms. Set $M^{\rm CY} := M_{m}^{\rm sn}$. 
			Since the diagram is commutative and the bottom row maps are isomorphisms, there are morphisms
			\[
			(\cM_l^{\rm sn} \to M^{\rm CY})_{l \geq m}
			\]
			that
			glue to give a morphism $\phi:\cM^{\CY} \to M^{\rm CY}$ such that the composition $\cM_l^{\rm sn}\to \cM^{\CY}\to M^{\rm CY}$ is a good moduli space. 
			Since $\cM^{\CY}= \cup_{l\geq m} \cM_l^{\rm sn} $,  $\phi$ is an asymptotically good moduli space.
		\end{proof}
		
		Theorem \ref{t:mainCY} is  a consequence of  results previously proven in this paper. 
		
		\begin{proof}[Proof of Theorem \ref{t:mainCY}]
			By Proposition \ref{p:existsAGMcy}, 
			there exists an asymptotically good moduli space morphism $\cM^{\CY}\to M^{\rm CY}$
			and $M^{\rm CY} \simeq M_m^{\rm sn}$ for $m$ sufficiently large.
			By Corollary \ref{c:projectivity}, $M_m$ is a projective scheme. 
			Since the seminormalization morphism $M^{\rm sn}_m\to M_m$ is finite, $M_m^{\rm sn}$ is a projective scheme as well. It remains to show that  (1)--(4)  hold. 
			\begin{enumerate}
				\item Since $\cM^{\CY}$ is a locally of finite type stack and $\phi$ is an asymptotically good moduli space morphism, \cite[Proposition 13.3.1]{BABWILD} implies that $\phi$ is universal among maps to algebraic spaces. 
				\item Since $\phi$ is an asymptotically good moduli space, \cite[Proposition 13.3.3]{BABWILD} implies that $\phi$ identifies two points $x=[(X,B+D)]$ and $x' = [(X',B'+D')]$ in $\cM^{\CY}(\bk)$ if and only if their closures in $\cM^{\CY}(\bk)$ intersect. 
				By Remark \ref{r:Sequiv}, the latter holds if and only if the pairs are S-equivalent. 
				
				\item Taking the seminormalization of the stacks and schemes in Proposition \ref{p:Mmwallcrossing} produces the desired wall crossing diagram. 
				
				\item Since the Hodge line bundle is ample on $M_m$ by Theorem \ref{t:amplehodge}, its pullback via the finite map $M^{\CY}\to M_m$ is also ample.
				
			\end{enumerate}
			Therefore the proof is complete.
		\end{proof}
		
		Next, we deduce Theorem \ref{t:mainintro}
		from Theorem \ref{t:mainCY}.
		This involves translating between slightly different setups and notations.
		
		\begin{proof}[Proof of Theorem \ref{t:mainintro}]
			Let  $(X,D)\to T$ be a family of marked boundary polarized CY surface pairs and write $\cM$ for the closure of the image of $T$ in the relevant moduli stack of boundary polarized CY pairs. 
			In the setup of Theorem \ref{t:mainintro}, 
			the stacks $\cM^{\K}$ and $\cM^{\rm KSBA}$ denote the open K and KSBA loci of $\cM^{\CY}$ respectively.
			
			The above notation conflicts with notation from earlier in this section.
			Indeed,  $\cM^{\K}$ and $\cM^{\KSBA}$ in Theorem \ref{t:mainintro} denote  the  seminormalizations of the stacks with the same notation in Theorem \ref{t:mainCY}. 
			Furthermore, since the seminormalization of a good moduli space morphism is a good moduli space morphism by \cite[Lemma 13.7]{BABWILD}, $M^{\K}$ and $M^{\KSBA}$ in Theorem \ref{t:mainintro} denote seminormalizations of the schemes with the same notation in Theorem \ref{t:mainCY}.
			With this translation,  Theorem \ref{t:mainCY} immediately implies Theorem \ref{t:mainintro}.
		\end{proof}

		\section{K3 surfaces and del Pezzo pairs}\label{s:examples}
		
		In this section, we discuss examples of the CY-moduli space constructed in Theorem \ref{t:mainintro} related to K3 surfaces and del Pezzo surfaces. In the process, we will prove Theorems \ref{t:delpezzointro} and \ref{t:nonsymplecticautintro}.

		\subsection{Moduli of del Pezzo pairs}\label{sec:dP}
		
		\subsubsection{Del Pezzo pairs}

		\begin{defn}\label{def:dP-pairs}
			Let $1\leq d\leq 9$ and $r\geq 1$ be integers. Let $\chi(m):= \chi(X,\omega_X^{[-m]}) = \frac{d}{2}(m^2+m) +1$ where $X$ is a smooth del Pezzo surface of degree $(-K_X)^2 = d$. 
			Let 
			\[
			\mathcal{DP}_{d,r}\subset \cM(\chi, r, 1, \tfrac{1}{r})
			\]
			denote the open substack parametrizing pairs $(X,B+D)$ where $X$ is a smooth del Pezzo surface of degree $d$, $B = 1\cdot \emptyset$, $D =\frac{1}{r}C$, and $C\in |-rK_X|$ is a smooth curve. 
			
			Let $\cDP_{d,r}^{\CY}$ be the seminormalization of the closure of $\cDP_{d,r}$ in $\cM(\chi, r, 1, \tfrac{1}{r})$. Following Section \ref{ss:KSBA+K}, let $\cDP_{d,r}^{\K}$ (resp. $\cDP_{d,r}^{\KSBA}$) be the open substack of $\cDP_{d,r}^{\CY}$ parametrizing boundary polarized CY pairs $(X,\frac{1}{r}C)$ in $\cDP_{d,r}^{\CY}$ such that $(X, \frac{1-\epsilon}{r}C)$ is K-semistable (resp. $(X, \frac{1+\epsilon}{r}C)$ is KSBA-stable) for $0<\epsilon\ll 1$.
		\end{defn}
		
		Since the deformation of a smooth del Pezzo surface $X$ is unobstructed, and the curve $C$ is Cartier on $X$, by arguments from $\bQ$-Gorenstein deformation theory (see e.g. \cite[Section 3]{Hac04} or \cite[Section 11.2]{BABWILD}, though as $X$ is smooth, the argument is much simpler than in the citations) 
		we know that $\mathcal{DP}_{d,r}$ is a smooth algebraic stack of finite type. 
		Moreover, since $(X, \frac{1-\epsilon}{r}C)$ is always K-stable for $(X,\frac{1}{r}C)\in \cDP_{d,r}$ by \cite[Theorem 2.10]{ADL21}, we know that $\mathcal{DP}_{d,r}$ is a Deligne-Mumford saturated open substack of $\cDP_{d,r}^{\K}$ using \cite{BX19}. Hence using Theorem \ref{t:Kmoduli} we know that $\mathcal{DP}_{d,r}$ admits a coarse moduli space 
		morphism to a normal quasi-projective scheme $\DP_{d,r}$. Note that the analogous statement holds using KSBA-moduli stacks and spaces if $r>1$.
		
		We also note that from the classification of smooth del Pezzo surfaces, $\mathcal{DP}_{d,r}$ and $\DP_{d,r}$ are irreducible except when $d=8$, where we have two irreducible components whose underlying surfaces are $\bP^1\times\bP^1$ or $\bF_1$.

		\begin{thm}\label{thm:delpezzo}
			For each integers $1\leq d \leq 9$ and  $r\geq 1$, there exists an asymptotically good moduli space morphism $\cDP_{d,r}^{\CY}\to \DP_{d,r}^{\CY}$ to a projective seminormal scheme $\DP_{d,r}^{\CY}$ whose closed points parametrize S-equivalence classes of pairs in $\mathcal{DP}_{d,r}^{\CY}(\bk)$.
			We have the following wall crossing diagram:
			\[
			\begin{tikzcd}
				\cDP^{\K}_{d,r} \arrow[r,hook] \arrow[d] &
				\cDP^{\CY}_{d,r} \arrow[d] & 
				\cDP^{\KSBA}_{d,r} \arrow[l,hook'] \arrow[d] \\
				\DP^{\K}_{d,r} \arrow[r]  & 
				\DP^{\CY}_{d,r}  & 
				\DP^{\KSBA}_{d,r} \arrow[l]
			\end{tikzcd}
			\]
			Here the top arrows are open immersions of algebraic stacks, the bottom arrows are birational morphisms when $r> 1$, and the vertical arrows are (asymptotically) good moduli space morphisms.
			Furthermore,  the Hodge line bundle is ample on $\DP_{d,r}^{\CY}$.
		\end{thm}
		
		\begin{proof}
			The existence and projectivity (including the ampleness of Hodge line bundle) of asymptotically good moduli space $F_{\rho}^{\CY}$ and the wall crossing diagram (except the birationality) follows from Theorem \ref{t:mainintro}. The birationality follows from Theorem \ref{t:mainCY}.3 and the fact that if $r>1$ then every pair $(X, \frac{1}{r}C)\in \cDP_{d,r}$ is klt.
		\end{proof}
		
		\begin{proof}[Proof of Theorem \ref{t:delpezzointro}]
			This follows directly from Theorem \ref{thm:delpezzo}.
		\end{proof}
		
		\begin{rem}\label{r:notfinitetype}
			Similar to \cite[Example 9.4]{BABWILD}, one can show that $\cDP_{d,r}^{\CY}$ is not of finite type by either using the infinite toric degenerations of smooth del Pezzo surfaces from \cite[Theorem 4.1]{HP10} or the stable toric degenerations to a cone over cycle of $\bP^1$'s with arbitrary length. As a result, in Theorem \ref{thm:delpezzo} we cannot remove the ``asymptotically'' assumption on the moduli space.
		\end{rem}
		
		\begin{rem}
			When $r=1$, the morphism $\DP^{\K}_{d,1} \to \DP^{\CY}_{d,1}$ is not birational except when $d=9$, while $\DP^{\KSBA}_{d,1}$ is empty. This can be seen from the simple fact that every pair $(X,C)\in \cDP_{d,1}$ admits a weakly special degeneration to $(X_0, C_0)$ where $X_0$ is the projective cone over $C$ with polarization $\cO_X(C)|_C$, and $C_0$ is the section at infinity. Thus $\DP^{\CY}_{d,1}$ is isomorphic to $\bP^1$ via the $j$-invariant map of $C$, while $\dim \DP^{\K}_{d,1} = \dim \DP_{d,1} = 10-d\geq 2$ if $d<9$. When $d=9$, i.e. a general surface is $\bP^2$, we know that $\DP^{\K}_{9,1} \to \DP^{\CY}_{9,1}$ is an isomorphism by \cite[Proposition 16.1]{BABWILD}. See \cite{MGPZ24} for an explicit description of $\DP_{d,1}^{\K}$ when $d\geq 2$.
		\end{rem}
		
		\subsubsection{Odd bidegree curves on $\bP^1\times\bP^1$}
		
		Let $r\geq 3$ be an odd integer. Let $\chi'(m):=(2m+1)^2$. Following \cite[Section 16.3]{BABWILD}, let $\cDP_{8, \frac{r}{2}}$ denote the open substack of $\cM(\chi', r, 1, \frac{2}{r})$ parametrizing pairs $(X, 1\cdot \emptyset + \frac{2}{r}C)$ where $X\cong \bP^1\times\bP^1$ and $C\in |\cO_X(r,r)|$ is a smooth curve. Let $\cDP_{8, \frac{r}{2}}^{\CY}$ be the seminormalization of the stack theoretic closure of $\cDP_{8, \frac{r}{2}}$ in  $\cM(\chi', r, 1, \frac{2}{r})$. We define $\cDP_{8, \frac{r}{2}}^{\K}$, $\cDP_{8, \frac{r}{2}}^{\KSBA}$ and their projective good moduli spaces $\DP_{8, \frac{r}{2}}^{\K}$, $\DP_{8, \frac{r}{2}}^{\KSBA}$ similar to Definition \ref{def:dP-pairs} and Theorem \ref{thm:delpezzo}. Here we change the notation from \cite[Definition 16.10]{BABWILD} to be consistent with earlier discussions on general del Pezzo pairs.
		
		\begin{thm}
			For each odd integer $r\geq 3$, there exists a good moduli space morphism $\cDP_{8,\frac{r}{2}}^{\CY} \to \DP_{8,\frac{r}{2}}^{\CY}$ to a projective scheme. We have the following wall crossing diagram of birational morphisms between good moduli spaces:
			\[
			\DP_{8,\frac{r}{2}}^{\K}\to \DP_{8,\frac{r}{2}}^{\CY}\leftarrow \DP_{8,\frac{r}{2}}^{\KSBA}
			\]
			Furthermore, the Hodge line bundle is ample on $\DP_{8,\frac{r}{2}}^{\CY}$.
		\end{thm}
		
		\begin{proof}
			The whole statement except the birationality and the ampleness of the Hodge line bundle follows from \cite[Theorem 16.11]{BABWILD} (since the stack $\cDP_{8,\frac{r}{2}}^{\CY}$ is bounded, we can avoid applying Theorem \ref{t:mainintro}). The birationality
			follows Theorem \ref{t:mainCY}.3 and  the fact that every pair $(X, \frac{2}{r}C)$ in $\DP_{8,\frac{r}{2}}$ is klt. 
			The ampleness of Hodge line bundle follows from Theorem \ref{t:amplehodge}.
		\end{proof}
		
		\begin{rem}
			As we shall see later in Example \ref{expl:K3}, the normalization of the moduli space $\DP_{8,\frac{3}{2}}^{\CY}$ is indeed isomorphic to the Baily--Borel compactification of one of Kond\={o}'s ball quotient models from \cite{Kon02}. Note that the KSBA moduli space $\DP_{8,\frac{3}{2}}^{\KSBA}$ was studied in detail in \cite{DH21}.
		\end{rem}
		
		\subsubsection{Del Pezzo surfaces with sum of lines}
		
		Let $(d,r)\in \{(4,4), (3,9), (2, 28), (1, 240)\}$. Let $\cY_{9-d}$ denote the seminormalization of the locally closed substack of $\cDP_{d, r}^{\CY}$ parametrizing pairs $(X, \frac{1}{r}C)$ where $X$ is a smooth del Pezzo surface of degree $d$, and $C = \sum_{i=1}^{dr} C_i$ is the unmarked divisor as the sum of  lines in $X$. Let $\cY_{9-d}^{\CY}$ denote the seminormalization of the closure of the image of $\cY_{9-d}$ in $\cDP_{d, r}^{\CY}$. We call $\cY_{9-d}^{\CY}$ the \emph{CY-moduli stack of unmarked del Pezzo surfaces of degree $d$.} 
		We define $\cY_{9-d}^{\K}$, $\cY_{9-d}^{\KSBA}$ and their projective good moduli spaces $\rmY_{9-d}^{\K}$, $\rmY_{9-d}^{\KSBA}$ similar to Definition \ref{def:dP-pairs} and Theorem \ref{thm:delpezzo}.
		
		\begin{thm}\label{thm:dP-lines}
			Notation as above. For each  integer $1\leq d\leq 4$, there exists an asymptotically good moduli space morphism $\cY_{9-d}^{\CY}\to \rmY_{9-d}^{\CY}$ to a projective seminormal scheme $\rmY_{9-d}^{\CY}$, which is called the \emph{CY-moduli space of unmarked del Pezzo surfaces of degree $d$}. We have the following wall crossing diagram of birational morphisms between (asymptotically) good moduli spaces:
			\[
			\rmY_{9-d}^{\K}\to  \rmY_{9-d}^{\CY}\leftarrow \rmY_{9-d}^{\KSBA}
			\]
			Furthermore, the Hodge line bundle is ample on $\rmY_{9-d}^{\CY}$.
		\end{thm}
		
		\begin{proof}
			The statement except the birationality follows from Theorem \ref{thm:delpezzo}. The birationality follows from the fact that for a general smooth del Pezzo surface $X$, the sum of its lines $C=\sum_{i=1}^{dr}C_i$ is a simple normal crossing divisor and so $(X, \frac{1}{r}C)$ is klt.
			Therefore Theorem \ref{t:mainCY}.3 implies  the birationality.
		\end{proof}
		
		\begin{rem}
			For each $c\in [0,1]\cap \bQ$, we denote by $\rmY_{9-d}(c)$ the seminormalization of the K-/CY-/KSBA-moduli compactification of $(X,cC)$ according to when $c$ is smaller than/equal to/larger than $\frac{1}{r}$, where $X$ is a smooth del Pezzo surface of degree $d$ and $C=\sum_{i=1}^{dr}C_i$ is the unmarked divisor as the sum of lines on $X$. In particular, we have $\rmY_{9-d}^{\K} = \rmY_{9-d}(\frac{1}{r}-\epsilon)$, $\rmY_{9-d}^{\CY} = \rmY_{9-d}(\frac{1}{r})$, $\rmY_{9-d}^{\KSBA} = \rmY_{9-d}(\frac{1}{r}+\epsilon)$ for $0<\epsilon\ll 1$. Then Theorem \ref{thm:dP-lines} implies that there is a log Calabi--Yau wall crossing at $c= \frac{1}{r}$. This together with the wall crossing theory for K-moduli \cite{ADL19, Zho23} and KSBA-moduli \cite{ABIP23, MZ23} gives a complete wall-crossing framework for moduli spaces $\rmY_{9-d}(c)$ in the entire interval $c\in [0,1]$. 
			
			Note that these moduli spaces have been studied before. When $c<\frac{1}{r}$, $Y_{9-d}(0)$ recovers the K-moduli compactification of smooth del Pezzo surfaces \cite{MM93, OSS16}, and explicit K-moduli wall crossings for $\rmY_{9-d}(c)$ were studied in \cite{Zha23} where it was shown that there are no K-moduli walls for $d=2,3,4$. When $c>\frac{1}{r}$, let $\ormY_{9-d}(c)$ be the corresponding marked KSBA moduli space. Then $\ormY_{9-d}(c)\to \rmY_{9-d}(c)$ is a Galois morphism as the quotient by the Weyl group $W(E_{9-d})$-action permuting the lines. Then $\ormY_{9-d}(1)$ recovers the moduli spaces studied in \cite{HKT09}, $\ormY_{6}(\frac{1}{9}+\epsilon)$ was studied in \cite{GKS21}, and explicit KSBA-moduli wall crossings for $\ormY_{9-d}(c)$ was studied in \cite{Sch23} for $d = 3,4$. 
			
			It is an interesting question to explicitly describe the CY-moduli space $\rmY_{9-d}^{\CY}$. 
			When $d=3,4$, $\rmY_{9-d}^{\K}$ has Picard rank $1$, 
			since the moduli space $\rmY_{9-d}^{\K}$ is isomorphic to $\rmY_{9-d}(0)$ by \cite{Zha23} which is also isomorphic to the corresponding GIT moduli spaces by \cite{MM93, OSS16}. Thus Theorem \ref{thm:dP-lines} implies that the normalization of $\rmY_{9-d}^{\CY}$ is isomorphic to $\rmY_{9-d}^{\K}$ and hence to the GIT moduli space. 
		\end{rem}

		\subsection{Moduli of K3 surfaces with a non-symplectic automorphism} 
		Throughout this subsection, assume that the base field $\bk=\bC$.
		We now consider the moduli of K3 surfaces with a non-symplectic automorphism $(S,\sigma)$. Recall that an automorphism $\sigma$ of a  smooth K3 surface $S$ is called \emph{non-symplectic} if $\sigma$ has finite order $n>1$, and $\sigma^*\omega_S = \zeta_n\omega_S$ where $\omega_S \in H^{2,0}(S)$ is a non-zero holomorphic $2$-form and $\zeta_n$ is a primitive $n$-th root of unity. In addition, we will always assume that the fixed locus of $\sigma$ contains a curve $R$ of genus at least $2$, which is called the $(\exists g\geq 2)$ assumption in \cite{AEH21}.
		By the Hodge index theorem, such a curve $R$ is unique if exists.
		
		Denote by $L_{K3}:= U^{\oplus 2} \oplus E_8^{\oplus 3}$ the K3 lattice. Let $\rho\in \rmO(L_{K3})$ be an isometry of order $n>1$.  A \emph{$\rho$-marking} $\phi$ of a smooth K3 surface with a non-symplectic automorphism $(S,\sigma)$ is an isometry of lattices $\phi: H^2(S,\bZ) \to L_{K3}$ such that $\sigma^* = \phi^{-1}\circ \rho \circ \phi$. We say $(S,\sigma)$ is \emph{$\rho$-markable} if such a $\rho$-marking exists. Recall that the period domain $\bD$ for complex K3 surfaces is 
		\[
		\bD:= \bP \{x\in L_{K3, \bC} \mid x\cdot x = 0, ~x\cdot \bar{x}>0\}.
		\]
		Let $L_{K3,\bC}^{\zeta_n}$ be the $\rho$-eigenspace of $L_{K3,\bC}$ of eigenvalue $\zeta_n$. 
		Let $\bD_{\rho}:=\bP(L_{K3,\bC}^{\zeta_n})\cap \bD$ be
		the period domain for $\rho$-markable K3 surfaces and $\Gamma_{\rho}< \rmO(L_{K3})$ where $\Gamma_{\rho}:=\{\gamma\in \rmO(L_{K3})\mid \gamma\circ \rho = \rho\circ \gamma\}$. By \cite{DK07, AEH21} we know that $\bD_{\rho}$ is a Hermitian symmetric domain of Type IV (resp. Type I) for  $n=2$ (resp. $n\geq 3$).
		
		By \cite{AEH21}, there exists a Deligne-Mumford stack $\sF_{\rho}^{\rm sm}$ parametrizing $\rho$-markable smooth K3 surfaces with a non-symplectic automorphism. Moreover, the separated quotient of $\sF_{\rho}^{\rm sm}$ is isomorphic to a separated Deligne-Mumford stack $\sF_{\rho}^{\rm ade}$ which parametrizes  pairs $(\oS, \oR)$ where $\pi:S\to \oS$ is the ample model of $R$ with $\oS$ a K3 surface with ADE singularities, and $\oR = \pi_* R$ is an ample Cartier divisor on $\oS$. Moreover, $\sF_{\rho}^{\rm ade}$ admits a coarse moduli space $F_{\rho}^{\rm ade}$ isomorphic to $(\bD_{\rho}\setminus \Delta_{\rho})/\Gamma_{\rho}$ where $\Delta_{\rho}$ is the discriminant locus (see \cite[Definition 2.5]{AEH21}). 
		
		Since $(\oS, \epsilon \oR)$ is KSBA-stable for $0<\epsilon\ll 1$ for every pair $(\oS, \oR)\in \sF_{\rho}^{\rm ade}$, we can take the seminormalization of its  KSBA compactification $\sF_{\rho}^{\rm slc}$, which is a proper Deligne-Mumford stack. Note that by \cite{KX20,Bir22,Bir23} it is known that $\sF_{\rho}^{\rm slc}$ is independent of the choice of $0<\epsilon\ll 1$ and admits a projective coarse moduli space $F_{\rho}^{\rm slc}$. Moreover, we have an open immersion of Deligne-Mumford stacks $\sF_{\rho}^{\rm ade} \hookrightarrow \sF_{\rho}^{\rm slc}$ which descends to an open immersion of coarse moduli spaces $F_{\rho}^{\rm ade} \hookrightarrow F_{\rho}^{\rm slc}$. 
		
		Next, we use the language of boundary polarized CY pairs to reinterpret the above setup.
		\begin{defn}
			Fix an isometry $\rho\in \rmO(L_{K3})$ of order $n>1$.
			For each $\rho$-markable smooth K3 surface $(S,\sigma)$ satisfying the assumption $(\exists g\geq 2)$, we associate a boundary polarized CY pair $(X, \frac{n-1}{n} C)$ in the following way: let $(\oS, \oR)$ be the ample model of $(S,R)$ as above, then $X:= \oS/\langle\overline{\sigma} \rangle$ where $\overline{\sigma}\in \Aut(\oS)$ is the descent of $\sigma$, and $C$ is the reduced image of $\oR$ under the quotient map $\oS\to X$ produces a boundary polarized CY pair by \cite[Theorem 4.2]{AEH21}. 
			
			Let $\sF_{\rho}$ be the moduli stack of $(X, \frac{n-1}{n} C)$ obtained from $\rho$-markable K3 surfaces. Let $\sF_{\rho}^{\CY}$ be the seminormalization of the closure of $\sF_{\rho}$ in the relevant moduli stack of boundary polarized CY pairs with index dividing $n$. 
			Following Section \ref{ss:KSBA+K}, let $\sF_{\rho}^{\K}$ (resp.  $\sF_{\rho}^{\KSBA}$) be the open substack of $\sF_{\rho}^{\CY}$ parametrizing boundary polarized CY pairs $(X,D)$ such that $(X, (1-\epsilon)D)$ is K-semistable (resp. $(X, (1+\epsilon)D)$ is KSBA-stable) for $0<\epsilon\ll 1$.
			
		\end{defn}
		
		For simplicity, we assume that all moduli stacks in this subsection are seminormal by replacing them with their seminormalization. 
		
		\begin{thm}[\cite{AEH21}]\label{thm:AEH}
			The morphism of stacks $\sF_{\rho}^{\rm ade}\to \sF_{\rho}$ induced by the quotient map $(\oS, \oR)\mapsto (X, \frac{n-1}{n}C)$ extends to a morphism $\sF_{\rho}^{\rm slc} \to \sF_{\rho}^{\KSBA}$ such that both morphisms are $\bmu_n$-gerbes. In particular, we have isomorphisms between coarse moduli spaces $F_{\rho}^{\rm ade} \cong F_{\rho}$ and $F_{\rho}^{\rm slc}\cong F_{\rho}^{\KSBA}$.
			
			Furthermore, the normalization of $F_{\rho}^{\rm slc}$ is a semitoroidal compactification of $\bD_{\rho}/\Gamma_{\rho}$.
		\end{thm}
		
		\begin{proof}
			The first paragraph follows from \cite[Theorem 4.2 and Corollary 4.3]{AEH21}. The last statement follows from \cite[Theorem 3.26]{AEH21} whose proof is largely based on \cite{AE23}.
		\end{proof}

		
		\begin{thm}\label{thm:K3-CY=BB}
			Notation as above. Then there exists an asymptotically good moduli space morphism $\sF_\rho^{\CY}\to F_{\rho}^{\CY}$ to a projective scheme $F_{\rho}^{\CY}$. We have the following wall crossing diagram of birational morphisms
			\[
			F_\rho^{ \rm K} \to
			F_{\rho}^{\CY}  \leftarrow
			F_{\rho}^{\KSBA} 
			\]
			Furthermore, the normalization of $F_{\rho}^{\CY}$ is isomorphic to the Baily--Borel compactification $(\bD_{\rho}/\Gamma_{\rho})^{\rm BB}$, and their Hodge line bundles coincide up to a positive scaling.
		\end{thm}
		
		\begin{proof}
			The existence and projectivity of the asymptotically good moduli space $F_{\rho}^{\CY}$ and the wall crossing diagram (except the birationality) follows from Theorem \ref{t:mainintro}. The birationality  follows from the irreducibility of these moduli spaces and Theorem \ref{t:mainCY}.3. We shall now focus on showing that the normalization of $F_{\rho}^{\CY}$ is isomorphic to the Baily--Borel compactification.
			
			By \cite[Theorem 6.6]{LMB00}, there exists a proper, surjective, and generically \'etale morphism $T\to \sF_{\rho}^{\rm slc}$ from a projective variety $T$. Replacing $T$ with a resolution, we may assume that $T$ is smooth and irreducible. Let $f:(\ocS, \epsilon \ocR)\to T$ be the pull-back of the universal family over $\sF_{\rho}^{\rm slc}$. By Theorem \ref{thm:AEH} we know that there exists a fiberwise $\bmu_n$-action on $\ocS$ over $T$ such that the quotient pair $(\cX, \frac{n-1}{n}\cC):=\ocS/\bmu_n$ satisfies that $g:(\cX, \frac{n-1+\epsilon}{n}\cC)\to T$ is the pull-back of the universal family over $\sF_{\rho}^{\KSBA}$ under the composition $T\to \sF_{\rho}^{\rm slc} \to \sF_{\rho}^{\KSBA}$. Denote by $\pi: \ocS \to \cX$ the quotient map. Since $nK_{\ocS/T} \sim \pi^* (nK_{\cX/T} + (n-1)\cC)\sim_{T} 0$, we have that $f_* \omega_{\ocS/T}^{[n]} \cong g_* \omega_{\cX/T}^{[n]}((n-1)\cC) = \lambda_{\Hodge, g, n}$ as line bundles on $T$. 
			
			Since $g: (\cX, \frac{n-1}{n}\cC)\to T$ is a family of boundary polarized CY pairs with index dividing $n$ satisfying that $T$ is smooth and a general fiber is normal, \cite[Proposition 14.7]{BABWILD} implies that $g$ is an lc-trivial fibration with zero discriminant divisor, $K_{\cX} + \frac{n-1}{n}\cC\sim_{\bQ} g^*(K_T+M_T)$,  $\cO_T(nM_T)\cong \lambda_{\Hodge, g,n}$, and the moduli b-$\bQ$-divisor $\bfM$ is the pull-back of $M_T$ (see \cite[Section 14.1]{BABWILD} for the relevant definitions). Since $\pi: (\ocS,0) \to (\cX, \frac{n-1}{n}\cC)$ is crepant, $f: \ocS \to T$ is also an lc-trivial fibration with zero discriminant divisor whose moduli b-$\bQ$-divisor is also $\bfM$. Let $h: T\to F_{\rho}^{\CY, {\rm n}}$ be the normalization of the composition $T \to \sF_{\rho}^{\KSBA} \to F_{\rho}^{\KSBA} \to F_{\rho}^{\CY}$. By Theorems \ref{t:mainintro}(2) and \ref{t:amplehodge}, we know that $L_{\Hodge}$ is ample on $F_{\rho}^{\CY}$. Abusing notation, we denote  the pull-back of $L_{\Hodge}$ to $F_{\rho}^{\CY, {\rm n}}$ by $L_{\Hodge}$. Then we have
			\[
			h^* L_{\Hodge}^{\otimes n} \cong \lambda_{\Hodge, g,n} \cong \cO_T(nM_T).
			\]
			
			Next, we seek to apply \cite[Theorem 1.2]{Fuj03} to relate $M_T$ with the Hodge line bundle on the Baily--Borel compactification. In order to do so, we need to perform a birational modification. First of all, we show that $\ocS$ is klt. Since a general fiber of $f$ is a K3 surface with ADE singularities, we know that no lc center of $\ocS$ dominates $T$. If there exists an lc center $Z$ of $\ocS$ with $W:=f(Z) \not\subset T$, then applying \cite[Lemma 2.7]{BABWILD} we can find a prime divisor $P$ containing $W$ such that $(\ocS, f^* P)$ is lc over a neighborhood of the generic point of $W$. This is a contradiction to $Z$ being an lc center. Thus $\ocS$ is klt. 
			
			Next, consider the minimal resolution $\varphi_{\eta}:\cS_{\eta}\to \ocS_{\eta}$ where $\eta$ is the generic point of $T$. By the uniqueness of minimal resolutions of surfaces, we know that $\cS_{\eta}$ is a smooth K3 surface over $K(T)$, and that each $\varphi_{\eta}$-exceptional divisor has log discrepancy $1$ over $\ocS_{\eta}$. Thus we may apply \cite[Corollary 1.4.3]{BCHM10} to get a projective birational morphism $\varphi: \cS \to \ocS$ such $\varphi$ only extracts $\varphi_{\eta}$-exceptional divisors. Hence $\varphi$ is crepant and its general fiber is a smooth K3 surface. By replacing $T$ with a log resolution, we may assume that there exists a simple normal crossing divisor $\Sigma$ on $T$ such that the restriction of $\tf=f\circ \varphi: \cS \to T$ over $T\setminus \Sigma$ gives a family of smooth $\rho$-markable K3 surfaces. By  Borel's extension theorem,  Lemma \ref{lem:BB-finite}, and \cite[Theorems 1.2 and 2.10]{Fuj03}, there exists an extended period map $\wp: T\to (\bD_{\rho}/\Gamma_{\rho})^{\rm BB}$  as a morphism between projective varieties and a positive integer $k$ such that 
			\[
			\cO_T(19k M_T) \cong (\iota^{\rm BB}\circ \wp)^*\cO_{(\bD_{2d}/\Gamma_{2d})^{\rm BB}}(1),
			\]
			where $(\bD_{2d}/\Gamma_{2d})^{\rm BB}$ is embedded in a projective space by automorphic forms of weight $k$. 
			By Lemma \ref{lem:BB-finite} we know that $\lambda_{\rho,{\rm BB}}:=(\iota^{\rm BB})^*\cO_{(\bD_{2d}/\Gamma_{2d})^{\rm BB}}(1)$ is an ample line bundle on $(\bD_{\rho}/\Gamma_{\rho})^{\rm BB}$ that is proportional to the Hodge line bundle by a positive scalar, such that $\cO_T(19k M_T) \cong \wp^* \lambda_{\rho, {\rm BB}}$.
			
			From the above discussion, we see that there are generically finite morphisms $h:T\to F_{\rho}^{\CY, {\rm n}}$ and $\wp: T\to (\bD_{\rho}/\Gamma_{\rho})^{\rm BB}$ whose targets are birational to each other such that $h^*L_{\Hodge}$ and $\wp^*\lambda_{\rho,{\rm BB}}$ are proportional by a positive scalar. Thus we conclude that $F_{\rho}^{\CY, {\rm n}}$ and $(\bD_{\rho}/\Gamma_{\rho})^{\rm BB}$ are isomorphic under which $L_{\Hodge}$ and $\lambda_{\rho,{\rm BB}}$ are proportional by a positive scalar. The proof is finished.
		\end{proof}
		
		\begin{lem}\label{lem:BB-finite}
			For a smooth $\rho$-markable K3 surface $(S,\sigma)$ with $R$ a $\sigma$-fixed curve of genus at least $2$, we can associate a quasi-polarized K3 surface $(S, L)$ where $L$ is a primitive nef and big line bundle on $S$ such that $R\sim L^{\otimes m}$ for some $m\in \bZ_{>0}$. Suppose $(L^2) = 2d$. Let $\bD_{2d}:=\bP([L]^{\perp}\otimes \bC)\cap \bD$ be the period domain of quasi-polarized K3 surfaces of degree $2d$, and $\Gamma_{2d}:=\{\gamma\in \rmO(L_{K3}) \mid \gamma([L])=[L]\}$.
			Then this induces a finite morphism $\iota^{\rm BB} :(\bD_{\rho}/\Gamma_{\rho})^{\rm BB} \to (\bD_{2d}/\Gamma_{2d})^{\rm BB}$.
		\end{lem}
		
		\begin{proof}
			First of all, the map $(S,\sigma)\mapsto (S, L)$ induces a morphism between moduli stacks $\sF_{\rho}^{\rm ade} \to \sF_{2d}$, which descends to a morphism between coarse moduli spaces $F_{\rho}^{\rm ade}\to F_{2d}$. Here $\sF_{2d}$ denotes the moduli stack of polarized K3 surfaces of degree $2d$ with ADE singularities; see \cite[Section 2C]{AE23}. Thus we have a morphism $(\bD_{\rho}\setminus \Delta_{\rho})/\Gamma_{\rho}\to \bD_{2d}/\Gamma_{2d}$ via the global Torelli theorem. 
			By \cite[Proof of Theorem 3.24 and Theorem 3.26]{AEH21}, this  extends to a morphism $\iota:\bD_{\rho}/\Gamma_{\rho}\to \bD_{2d}/\Gamma_{2d}$. Since these constructions can be recovered from variation of Hodge structures, we know that $\iota$ is induced by the totally geodesic embedding $\bD_{\rho}\hookrightarrow \bD_{2d}$ as $[L]$ is $\rho$-fixed. Thus by \cite{KK72} we know that $\iota$ extends to a morphism $\iota^{\rm BB}: (\bD_{\rho}/\Gamma_{\rho})^{\rm BB} \to (\bD_{2d}/\Gamma_{2d})^{\rm BB}$. Since both Baily--Borel compactifications are projective, the finiteness of  $\iota^{\rm BB}$ follows from its quasi-finiteness which is a consequence of the injectivity of $\bD_{\rho}^+\hookrightarrow \bD_{2d}^+$ (where ${}^+$ denotes the union with rational boundary components; see e.g. \cite[Definition 5.3]{AE23} or \cite[Definition 3.13]{AEH21}) and the discreteness of arithmetic groups.
		\end{proof}
		
		\begin{proof}[Proof of Theorem \ref{t:nonsymplecticautintro}]
			This follows directly from Theorem \ref{thm:K3-CY=BB}.
		\end{proof}
		
		\begin{rem}
			When $n\geq 3$, the moduli stack $\sF_{\rho}^{\CY}$ is of finite type as it parametrizes pairs of Type I or II by \cite[Theorem 8.15 and Proposition 8.18]{BABWILD}. Therefore, $F_{\rho}^{\CY}$ is indeed a good moduli space of  $\sF_{\rho}^{\CY}$. In other words, the moduli stack  is bounded when the corresponding period domain $\bD_{\rho}$ is a Type I Hermitian symmetric domain, i.e. a complex ball, in which case Theorem \ref{thm:K3-CY=BB} implies that the normalization of $F_{\rho}^{\CY}$ is isomorphic to the Baily--Borel compactification of a ball quotient. 
		\end{rem}
		
		\begin{rem}
			Let $(\cX, \frac{n-1}{n}\cC)\to \sF_{\rho}^{\CY}$ be the universal family.
			Since a general pair $(X, \frac{n-1}{n}C)\in \sF_{\rho}^{\CY}$ is a $\bmu_n$-quotient of an ADE K3 surface, it satisfies that $nK_X + (n-1) C\sim 0$. Thus by \cite[Lemma 2.11]{BABWILD} we know that $nK_{\cX/\sF_{\rho}^{\CY}} + (n-1)\cC \sim_{\sF_{\rho}^{\CY}} 0$.
			By taking fiberwise index $1$ cover of the universal family over $\sF_{\rho}^{\CY}$ (see e.g. \cite[Proof of Proposition 6.12]{ADL21}), we obtain a $\bmu_n$-gerbe $\widetilde{\sF}_{\rho}^{\CY}\to \sF_{\rho}^{\CY}$ together with a universal family $(\widetilde{\cS}, \widetilde{\cR})\to \widetilde{\sF}_{\rho}^{\CY}$ with a fiberwise $\bmu_n$-action such that $\sF_{\rho}^{\rm ade}\hookrightarrow \sF_{\rho}^{\rm slc}\hookrightarrow\widetilde{\sF}_{\rho}^{\CY}$ are  dense open substacks. Then Theorem \ref{thm:K3-CY=BB} together with the fact that  $\bmu_n$-gerbes are cohomologically affine implies that $\widetilde{\sF}_{\rho}^{\CY}$ admits an asymptotically good moduli space $\widetilde{F}_{\rho}^{\CY}$ isomorphic to $F_{\rho}^{\CY}$, whose normalization is isomorphic to $(\bD_{\rho}/\Gamma_{\rho})^{\rm BB}$.  Thus $\widetilde{\sF}_{\rho}^{\CY}\to \widetilde{F}_{\rho}^{\CY}$ (after normalization) provides a modular meaning of the Baily--Borel compactification $(\bD_{\rho}/\Gamma_{\rho})^{\rm BB}$.
		\end{rem}

		\begin{expl}\label{expl:K3}
			Here we collect some examples of moduli  of K3 surfaces with non-symplectic automorphisms that are directly related to smooth del Pezzo pairs. For more in-depth discussions on the topic, we refer to \cite{AN06, AEH21, AE22}.
			\begin{enumerate}
				\item The double cover of a smooth del Pezzo surface $X$ branched along a smooth bi-anti-canonical curve $C\in |-2K_X|$ is naturally a smooth K3 surface with a non-symplectic involution $\sigma$. Thus for each integer $1\leq d\leq 9$, the irreducible components of $\cDP_{d,2}^{\CY}$ coincide with $\sF_{\rho}^{\CY}$ for some $\rho$ depending only on the irreducible component. Note that $\cDP_{d,2}^{\CY}$ is irreducible exactly when $d\neq 8$, while it has two irreducible components when $d=8$. Thus Theorem \ref{thm:K3-CY=BB} implies that the normalization of every irreducible component of $\DP_{d,2}^{\CY}$ is isomorphic to $(\bD_{\rho}/\Gamma_{\rho})^{\rm BB}$.

				When $d=8$ or $9$, the irreducible components of  $\DP_{d,2}^{\CY}$ provide a modular meaning for the Baily--Borel compactification of the moduli space of hyperelliptic K3 surfaces. When $d=5$, $\DP_{5,2}^{\CY}$ provides a modular meaning for the Baily--Borel compactification of an arithmetic quotient of a Type IV Hermitian symmetric domain that is birational to the moduli space of genus $6$ curves \cite{AK11}.
				\item The triple cyclic cover of $\bP^1\times\bP^1$ branched along a smooth bidegree $(3,3)$-curve is naturally a smooth K3 surface with a non-symplectic automorphism $\sigma$ of order $3$. Let $\rho\in \rmO(L_{K3})$ be an isometry induced by $\sigma^*$. Thus Theorem \ref{thm:K3-CY=BB} implies that the normalization of $\DP_{8, \frac{3}{2}}^{\CY}$ is isomorphic to  $(\bD_{\rho}/\Gamma_{\rho})^{\rm BB}$. This provides a modular meaning for the Baily--Borel compactification of Kond\={o}'s ball quotient model for the moduli space of genus $4$ curves via K3 surfaces \cite{Kon02}. 
				
				For a similar example on the moduli of pairs $(\bP^2, \frac{3}{4}C)$, where $C$ is a plane quartic curve, and its relation to the ball quotient model for moduli of genus $3$ curves from \cite{Kon00}, see \cite[Proposition 16.2]{BABWILD}.
			\end{enumerate}
		\end{expl}
		
		\begin{rem}
			We expect that our machinery can be applied to the moduli of Enriques surfaces with a numerical polarization of degree $2$, often called Horikawa's model of Enriques surfaces. To be precise, let $(Z,[L])$ be a smooth Enriques surface $Z$ together with an ample class $[L]\in \Pic Z/\langle [K_Z]\rangle$. Let $M:=L^{\otimes 2}$ be the ample line bundle on $Z$. Then $|M|$ is base point free and defines a double cover $\pi: Z \to (W, \frac{1}{2}R_W)$ where $W$ is a del Pezzo surface of degree $4$ with $4A_1$- or $A_3+2A_1$-singularities. Denote by $R_Z$ the branch divisor in $Z$. Then in \cite{AEGS23} it was shown that the normalization of the KSBA compactification of pairs $(Z, \epsilon R_Z)$ is isomorphic to a specific semi-toroidal compactification of the relevant locally symmetric space $\bD(T_{\rm En})/\Gamma_{{\rm En},2}$ (see \cite[Section 2.3]{AEGS23} for definitions). Following the same strategy as the proof of Theorem \ref{thm:K3-CY=BB}, one should be able to show that the normalization of the boundary polarized CY moduli space of $(W, \frac{1}{2}R_W)$ is isomorphic to the Baily--Borel compactification $(\bD(T_{\rm En})/\Gamma_{{\rm En},2})^{\rm BB}$.
		\end{rem}

			
			


		\bibliographystyle{alpha}
		\bibliography{ref}
		
	\end{document}